\documentclass{article}
\usepackage[a4paper]{geometry}
\usepackage{amsmath}
\usepackage{amssymb}
\usepackage{amsthm}
\usepackage{mathrsfs}
\usepackage{mathtools}

\usepackage{tikz}
\usepackage{cleveref}
\usepackage{tikz-cd}
\usepackage{mathtools}

\usepackage{placeins}

\usepackage{float}

\usepackage{comment}

\usepackage{enumerate}

\usepackage{xinttools} 
\usetikzlibrary{calc}

\def\biglen{20cm} 
\tikzset{
  half plane/.style={ to path={
       ($(\tikztostart)!.5!(\tikztotarget)!#1!(\tikztotarget)!\biglen!90:(\tikztotarget)$)
    -- ($(\tikztostart)!.5!(\tikztotarget)!#1!(\tikztotarget)!\biglen!-90:(\tikztotarget)$)
    -- ([turn]0,2*\biglen) -- ([turn]0,2*\biglen) -- cycle}},
  half plane/.default={1pt}
}

\theoremstyle{definition}
\newtheorem{thm}{Theorem}[section]
\crefname{thm}{Theorem}{Theorems}
\newtheorem{cor}[thm]{Corollary}
\newtheorem{prop}[thm]{Proposition}
\crefname{prop}{Proposition}{Propositions}
\newtheorem{lem}[thm]{Lemma}
\crefname{lem}{Lemma}{Lemmas}
\newtheorem{clm}[thm]{Claim}
\newtheorem{conj}[thm]{Conjecture}
\newtheorem{quest}[thm]{Question}
\newtheorem{defn}[thm]{Definition}

\crefname{defn}{Definition}{Definitions}

\newtheorem{notn}[thm]{Notation}

\newtheorem{rmk}[thm]{Remark}

\newtheorem{obs}[thm]{Observation}
\newtheorem{conv}[thm]{Convention}
\newtheorem*{ack*}{Acknowledgements}

\newcommand{\cooo}{\widetilde{\operatorname{co}}}
\newcommand{\coo}{\widehat{\operatorname{co}}}

\newcommand{\co}{\operatorname{co}}

\newcommand{\pvh}{\textcolor{red}}
\newcommand{\hs}{\textcolor{blue}}
\newcommand{\mt}{\textcolor{green!50!black}}
\usepackage{titling}

\title{Sets in $\mathbb{Z}^k$ with doubling $2^k+\delta$ are near convex progressions}
\author{Peter van Hintum, Hunter Spink, Marius Tiba\thanks{We would like to thank our advisor Prof. B\'ela Bollob\'as for his continuous support, Prof. Alessio Figalli for helpful suggestions, and the referee for the incredibly detailed report.}}
\allowdisplaybreaks

\begin{document}

\maketitle

\begin{abstract}
For $\delta>0$ sufficiently small and $A\subset \mathbb{Z}^k$ with $|A+A|\le (2^k+\delta)|A|$, we show either $A$ is covered by $m_k(\delta)$ parallel hyperplanes, or satisfies $|\coo(A)\setminus A|\le c_k\delta |A|$, where $\coo(A)$ is the smallest convex progression (convex set intersected with an affine sub-lattice) containing $A$. This generalizes the Freiman-Bilu $2^k$ theorem, Freiman's $3|A|-4$ theorem, and recent sharp stability results of the present authors for sumsets in $\mathbb{R}^k$ conjectured by Figalli and Jerison.
\end{abstract}

\section{Introduction} 
One of the central questions in additive combinatorics is the inverse sumset problem of characterizing the finite subsets $A$ of abelian groups with small \textit{doubling constant} $|A+A|\cdot|A|^{-1} \le \lambda $ for fixed $\lambda>0$. In this paper, we will consider the inverse sumset problem in torsion-free abelian groups $\mathbb{Z}^k$, which has been studied from a variety of perspectives by Freiman \cite{Freiman}, Green and Tao \cite{GreenTao}, Chang \cite{Chang}, and Sanders \cite{Revisited} among others. 


 Motivated by the fact that for $\widetilde{A}\subset \mathbb{R}^k$ the doubling constant (with respect to volume) is at least $2^k$, we define $d_k(A)=|A+A|-2^k|A|$ for $A\subset \mathbb{Z}^k$. Our main result, \Cref{mainthm}, describes the structure of ``non-degenerate'' $A\subset \mathbb{Z}^k$ slightly beyond the critical doubling threshold, i.e. when $d_k(A)\le \Delta_k|A|$ or equivalently $|A+A|\le (2^k+\Delta_k)|A|$, for some absolute constant $\Delta_k$ depending only on $k$.
 
\begin{center}
\begin{tikzpicture}
\draw[black, fill=gray, fill opacity=0.5] (0,0)--(0.5,1)--(0,0.5)--(-0.5,1)--(0,0);
\node at (-1,0.5) {$\widetilde{A}=$};
\draw[black, fill=gray, fill opacity=0.5] (3,0)--(4,2)--(3.5,1.5)--(3,2)--(2.5,1.5)--(2,2)--(3,0);
\node at (1.5,0.5) {$\widetilde{A}+\widetilde{A}=$};
\draw[dotted] (4,2)--(3,1)--(2,2);
\begin{scope}[xshift=6cm]
\node at (-1,0.5) {$A=$};
\draw[color=gray, draw opacity=0.25] (0,0)--(0.5,1)--(0,0.5)--(-0.5,1)--cycle;
\begin{scope}
\clip (0,0)--(0.5,1)--(0,0.5)--(-0.5,1)--cycle;
\foreach \x in {-10,...,10} \foreach \y in {-20,...,20}
\draw[draw=black] (\x/10,\y/5+\x/10) rectangle ++(0.2pt,0.2pt);
\end{scope}
\draw[color=gray, draw opacity=0.25] (3,0)--(4,2)--(3.5,1.5)--(3,2)--(2.5,1.5)--(2,2)--(3,0);
\begin{scope}
\clip (3,0)--(4,2)--(3.5,1.5)--(3,2)--(2.5,1.5)--(2,2)--(3,0);
\foreach \x in {-50,...,50} \foreach \y in {-50,...,50}
\draw[draw=black] (\x/10,\y/5+\x/10) rectangle ++(0.2pt,0.2pt);
\end{scope}
\node at (1.5,0.5) {$A+A=$};
\end{scope}
\end{tikzpicture}
\end{center}

 For $\widetilde{A}\subset \mathbb{R}^k$ with the Lebesgue measure, the sets with $|\widetilde{A}+\widetilde{A}|-2^k|\widetilde{A}|=0$ are, up to measure $0$ sets, the convex sets. A natural class of discrete sets with similar small doubling properties are \emph{convex progressions}. For $A \subset \mathbb{Z}^k$, define the convex progression $\coo(A)$ to be the intersection of the real convex hull $\widetilde{\co}(A)$ with the affine sub-lattice $\Lambda_A$ spanned by $A$. Many authors require convex progressions to be symmetric, but in this paper we impose no such assumptions. Intuitively, for a convex progression $A=\widetilde{\co}(A)\cap \Lambda_A$, $A+A$ should be approximately the set of points in $\widetilde{\co}(2A)\cap \Lambda_{A+A}$, which since $\Lambda_{A+A}$ is a translate of $\Lambda_A$ has roughly $2^k$-times the number of points as $\widetilde{\co}(A)\cap \Lambda_A$ provided the convex sets are not so thin that the step-size of the lattice affects the approximation of the volumes of $\widetilde{\co}(A)$ and $\widetilde{\co}(2A)$ by suitably normalized lattice point counts.
 
The basic phenomenon that we could hope to expect is that $$d_k(A)|A|^{-1}\text{small implies } |\cooo(A)\setminus A||A|^{-1}\text{ small}.$$
Under some necessary hypotheses, in this paper we will be able to establish that when $d_k(A)|A|^{-1}\le \delta$, we have
\begin{itemize}
    \item Qualitatively, that $|\coo(A)\setminus A||A|^{-1}\le \omega(\delta)$ for some function $\omega(\delta)\to 0$ as $\delta\to 0$, and
    \item Quantitatively, that $|\coo(A)\setminus A||A|^{-1}\le c_k\delta$ for some absolute constant $c_k$.
\end{itemize}
These will be established in \Cref{quali} and \Cref{mainthm} respectively. Surprisingly, most of the work will be to establish \Cref{quali}, which is why we list it separately, and even this natural statement appears to have been unknown prior to our present work.
We now describe certain phenomena we must explicitly account for to make our theorems true.

The first is that there is a threshold $\Delta_k$ such that we can deduce no information about $|\coo(A)\setminus A||A|^{-1}$ if $\delta> \Delta_k$, and in fact it is possible that $|\coo(A)\setminus A||A|^{-1}\to \infty$. For example, if $A$ is the union  $$\{1,\ldots,n\}^k\cup (\{1,\ldots,n\}^k+(0,\ldots,0,N)),$$  then $d_k(A) \approx 3\cdot 2^k|A|$, but $|\coo(A)|=n^{k-1}(n+N)$, and for fixed $n$ we can make $|\coo(A)\setminus A||A|^{-1}$ arbitrarily large by taking $N\to \infty$. Hence we must assume that $\delta\le \Delta_k$ for some constant $\Delta_k$.
Let us introduce the notation that we will use throughout
\begin{align*}a\ll b\text{ meaning }&a\le b\text{ and there exists a fixed increasing function $f$}\\ &\text{depending only on $k$ such that }a\le f(b).\end{align*}
In this $\ll$ notation, the condition is
$$\delta \ll 1.$$
Just under this hypothesis, we can already establish a first important result bounding the pathology of $A\subset \mathbb{Z}^k$. Recall that the \emph{thickness} of a set $A\subset \mathbb{Z}^k$ is the smallest number of parallel hyperplanes required to cover $A$.
\begin{thm}\label{1.1parta}
For $k\geq 1$, there are positive constants $\Delta_k, m_k,\epsilon_k$ depending only on $k$ such that for $A\subset \mathbb{Z}^k$ with $d_k(A)\leq \Delta_k|A|$, either $A$ has thickness at most $m_k$, or $A$ lies in some rank $k$ generalized arithmetic progression $$B=B(n_1,\ldots,n_k;v_1,\ldots,v_k;b):=\left\{b+\sum_{i=1}^k\ell_iv_i: 0 \le \ell_i < n_i\right\}\subset \mathbb{Z}^k,$$
where $v_1,\ldots,v_k\in \mathbb{Z}^k$  independent, and $|A|\ge \epsilon_k |B|$.
\end{thm}

Next, let $h$ be the thickness of $A$. If $h$ does not tend to $\infty$ as $\delta\to 0$, then it is possible to achieve very small doubling without $A$ being very close to a convex progression because in this regime $d_{k-1}(A)=|A+A|-2^{k-1}|A|$ is a more appropriate statistic. For example, if $A=(\{1,\ldots,n_0\}\times\{1,\ldots,2n\}^{k-1}) \cup \{(-1,1,1,\ldots,1)\}$, where $n_0$ is constant and $n$ is much larger than $n_0$, then $A$ needs at least $n_0$ parallel hyperplanes to be covered, $d_k(A)<0$, but $|\coo(A) \setminus A| =   \frac{|A|-1}{2^{k-1} n_0}$. Hence we must assume that $$h^{-1}\ll \delta.$$

Under the hypotheses $h^{-1}\ll \delta \ll 1$, our main theorems \Cref{quali} and \Cref{mainthm} establish the qualitative statement $|\coo(A)\setminus A||A|^{-1}\le \omega(\delta)$ and the quantitative statement $|\coo(A)\setminus A||A|^{-1}\le c_k\delta$ for some constant $c_k$ depending only on $k$, mentioned before.

\begin{thm}\label{quali}\label{qual}
There exist $\ll$ dependencies $h^{-1}\ll \delta \ll 1$ and a function $\omega(\delta)\to0$ as $\delta\to 0$ such that if $A\subset \mathbb{Z}^k$ has $d_k(A) \le \delta |A|$ and the thickness of $A$ is at least $h$, then
$$|\coo(A)\setminus A||A|^{-1}\le \omega(\delta).$$
\end{thm}
A continuous analogue of \Cref{quali} proved by Christ \cite{Christ} and strengthened by Figalli and Jerison \cite{FigJerJems} was a key step in the study of stability results for the Brunn-Minkowski inequality in the non-convex setting. However, our methods are largely different from \cite{Christ} and \cite{FigJerJems}, especially because of phenomena which occur in the discrete setting which have no continuous analogue.
\begin{thm}\label{mainthm}\label{quant}
For $k\geq 1$, there are positive constants $c_k< (4k)^{5k}$, and positive constants $\Delta_k,g_k,m_k$ such that for $A\subset \mathbb{Z}^k$ with $d_k(A) \le \Delta_k|A|$, if the thickness of $A$ is at least $h\ge m_k$, then
$$|\coo(A)\setminus A|\le c_kd_k(A)+g_kh^{-\frac{1}{1+\frac{1}{2}(k-1)\lfloor k/2 \rfloor}}|A|.$$
\end{thm}
\begin{rmk}
\label{rmk:betterexponent}
We make a brief remark on the exponent $-\frac{1}{1+\frac{1}{2}(k-1)\lfloor k/2\rfloor}$ in \Cref{mainthm}, which is almost certainly not optimal. Our proof of \Cref{mainthm} reduces to the case that $\cooo(A)$ is a simplex, where the exponent is $-1$ (see \Cref{quantprop}, the exponent $-1$ is seen to be optimal by taking $A\subset \mathbb{Z}^k$ to be the set of lattice points in the convex hull of $0$ and the scaled standard basis vectors $(h-1)e_1,\ldots,(h-1)e_k$). For a general $A$ we first approximate $\cooo(A)$ from within by a polytope $\widetilde{P}$, and then triangulate $\widetilde{P}$ into simplices via a triangulation of $\partial \widetilde{P}$. To approximate the volume of $\cooo(A)$ by a polytope $\widetilde{P}$ with $|\widetilde{P}|\ge(1-\alpha)|\cooo(A)|$ can be done with $\ell=O(\alpha^{-\frac{2}{k-1}})$ vertices by Gordon, Meyer, and Reisner \cite{BestApproximation}, and Stanley's upper bound theorem \cite{Stanley} then implies that a triangulation of $\partial \widetilde{P}$ has at most $O(\ell^{\lfloor k/2 \rfloor})$ simplicies. These two bounds combine to give the exponent.
\end{rmk}
The following simple corollary of \Cref{mainthm} quantitatively strengthens the conclusion of \Cref{quali} to essentially best possible (up to the constant $c_k$).
\begin{cor}\label{maincor}
There is a positive constant $c_k<(4k)^{5k}$ and $\ll$-dependencies $h^{-1}\ll \delta \ll 1$ such that if $A\subset \mathbb{Z}^k$ has  thickness at least $h$, then $d_k(A)\le \delta|A|$ implies
$$|\coo(A)\setminus A||A|^{-1}\le c_k\delta.$$
\end{cor}
Note that for $A\subset \mathbb{Z}^k$ with thickness at least $h$ with $h^{-1}\ll \delta \ll 1$, we also have from \Cref{1.1parta} that $A$ has density at least $\epsilon_k$ in a $k$-dimensional generalized arithmetic progression. The converse implication (which is much easier) that there is a constant $c_k'$ and $\ll$-dependencies such that for $h^{-1}\ll \delta \ll 1$, if $A\subset \mathbb{Z}^k$ has thickness at least $h$, density at least $\epsilon_k'$ in a $k$-dimensional generalized arithmetic progression, and $|\coo(A)\setminus A|\le \delta |A|$, then $d_k(A)\le c_k'\delta |A|$, can be deduced from \Cref{nonreduced} and \Cref{cvxsd}. In this large-thickness regime, this therefore characterizes sets with small doubling as exactly those close to their convex progression hull with positive density inside some ambient full rank generalized arithmetic progression. For a set $A$ symmetric about a lattice point, the discrete John's theorem of Tao and Vu \cite{TAOVu} implies $\coo(A)$ has positive density in a $k$-dimensional generalized arithmetic progression, so in this case the density condition is superfluous.

\begin{cor}[vH,S,T \cite{HomoBM}]\label{mainfigalli}
There are positive constants $c_k< (4k)^{5k}, \Delta_k$ such that for $\widetilde{A} \subset \mathbb{R}^k$ of positive measure with $|\widetilde{A} +\widetilde{A} | \le (2^k+\Delta_k)|\widetilde{A} |$, we have $ |\cooo(\widetilde{A}) \setminus \widetilde{A}| \le c_k (|\widetilde{A} +\widetilde{A} | - 2^k|\widetilde{A} |).$ Here $|\cdot|$ denotes the outer Lebesgue measure.
\end{cor}
\begin{proof}
A proof by standard approximation techniques follows exactly as in \cite[p.3 footnote 2]{Semisum}.
\end{proof}
\Cref{mainfigalli}, the sharp stability of the Brunn-Minkowski inequality $|\widetilde{A}+\widetilde{C}|^{1/k}\ge |\widetilde{A}|^{1/k}+|\widetilde{C}|^{1/k}$ for equal sets $\widetilde{A}=\widetilde{C}$, was conjectured by Figalli and Jerison \cite{Semisum} and recently resolved by the authors of the present paper \cite{HomoBM} without any digression to the discrete setting. This result similarly yields a characterization of positive measure $\widetilde{A}\subset \mathbb{R}^k$ with $d_k(\widetilde{A}) \le O(\delta)|\widetilde{A}|$ as equivalently having $|\widetilde{\co}(\widetilde{A})\setminus \widetilde{A}| \le O(\delta)|\widetilde{A}|$.

Without thickness assumptions, Gardner and Gronchi \cite{Gardner} proved for $A,C\subset \mathbb{Z}^{k}$ not lying in hyperplanes an optimal lower bound for $|A+C|$, but the bound is far worse than   predicted by the Brunn-Minkowski inequality for measurable sets in $\mathbb{R}^k$. Under thickness assumptions, the situation is better. For example, by a result of Green and Tao \cite{GreenTao} (following an approach of Bollob\'as and Leader \cite{Bollobas}), if $A,C \subset B(n_1,\ldots,n_k;v_1,\ldots,v_k;b)$ and $|A|,|C| \ge \epsilon |B|$ then $|A+C|\ge (2^k + O(\epsilon\min(n_i))^{-1})\min(|A|,|C|)$. Showing a general form of the Brunn-Minkowski inequality for thick sets is open (see \cite[Conjecture 3.10.12]{Ruzsa}), though progress in this direction has been made by Cifre, Mar\'{\i}a, and Iglesias \cite{Cifre}.

 Proving further discrete analogues of stability results for the Brunn-Minkowski inequality for thick subsets $A,C \subset \mathbb{Z}^k$, such as that of Christ \cite{Christ} and Figalli and Jerison \cite{FigJerAdv} for general sets, or sharp stability results such as Barchiesi and Julin \cite{Barchiesi} for one of the sets being convex and the present authors \cite{1911.11945} for arbitrary two-dimensional sets, would be extremely interesting, and we believe would be a worthwhile goal to pursue.

We now seek to contextualize our results in the context of existing results within the additive combinatorics literature.

First, we recall Freiman's fundamental result \cite{Freiman} on sets with small doubling (later exposited by Bilu \cite{Bilu}).
\begin{thm}[Freiman \cite{Bilu,Freiman}]
There are constants $b_1(\lambda),b_2(\lambda), b_3(\lambda)$ such that for any finite subset $A\subset \mathbb{Z}^k$ with doubling constant less than $\lambda$ can be covered by $b_1(\lambda)$ translates of a generalized arithmetic progression of size at most $b_2(\lambda)|A|$, and of dimension at most $b_3(\lambda)$.
\end{thm}
Freiman originally formulated his theorem in terms of convex progressions (images of sets of the form $\coo(A')$ under affine linear maps) instead of generalized arithmetic progressions, and much of the literature focuses on this formulation. Generalizations of convex progressions were used implicitly by Bourgain \cite{Bourgain},  and Green-Sanders \cite{GreenSanders} (see Sander's extensive survey \cite{Revisited} for more information).

Our \Cref{mainthm} focuses entirely on those $\lambda$ with $\lambda \le 2^k+\Delta_k$, but provides an extremely sharp characterization in this regime. 

Green and Tao \cite{GreenTao} showed we can obtain optimal bounds for the dimension of the generalized arithmetic progressions, at the cost of the number of translates. In particular, for $\lambda\in [2^k,2^{k+1})$ we may take  $b_3(\lambda)=k$.
 
 Our \Cref{1.1parta} shows that when $\lambda \le 2^k+\Delta_k$, then under the non-degeneracy hypothesis that $A$ is not covered by $m_k$ parallel hyperplanes, we can take $b_1(\lambda)=1,b_2(\lambda)=\epsilon_k^{-1},b_3(\lambda)=k$.

\begin{rmk}
Although we focus on subsets of $\mathbb{Z}^k$, we remark briefly that Freiman's theorem has been generalized to arbitrary abelian groups by Green and Ruzsa \cite{GreenRusza}, and the recent literature on approximate groups seeks to describe analogous characterizations in non-abelian groups (see for example the seminal work of Breuillard, Green and Tao \cite{ApproxGroup}).

The constants $b_1(\lambda)$, $b_2(\lambda)$, and $\exp(b_3(\lambda))$ cannot all be brought down to polynomial as shown by Lovett and Regev \cite{counterexample}, but the analogous question reformulated in terms of (symmetric) convex progressions is open (the \emph{polynomial Freiman-Ruzsa conjecture}). Green and Ruzsa \cite{GreenRusza}, Chang \cite{Chang}, Bourgain \cite{Bourgain}, and Green and Tao \cite{GreenTao2} showed the constants could be reduced to $b_1(\lambda)=b_2(\lambda)=\exp(b_3(\lambda))=\exp(O(\lambda^C))$ for some constant $C$, improved by Schoen \cite{Schoen} to $\exp(\exp(O(\sqrt{\log(\lambda)}))$, and finally improved by Sanders \cite{Revisited} to $\exp(O(\log^{3+o(1)}\lambda))$.
\end{rmk}


Next, we recall Green and Tao's improvement \cite{GreenTao} to the classical Freiman-Bilu $2^k$-theorem \cite{Bilu,Freiman}, the central result relating the doubling of a set $A\subset \mathbb{Z}^k$ to its thickness.
\begin{thm}[Freiman-Bilu $2^k$ theorem \cite{Bilu,Freiman,GreenTao}]\label{mainbilu}
Given $\delta>0$, there is a constant $m_k(\delta)$ such that if $A\subset \mathbb{Z}^k$ has $|A+A|\le (2^k-\delta)|A|$, then the thickness of $A$ is at most $m_k(\delta)$.
\end{thm}
The Freiman-Bilu theorem shows that the correct notion of degeneracy in $\mathbb{Z}^k$ is being covered by a bounded number of parallel hyperplanes, and non-degenerate sets $A$ have doubling constant bounded below by roughly $2^k$. There is a large literature of classifications of subsets $A\subset \mathbb{Z}^k$ with doubling at most $2^k-\delta$ (see e.g. Fishburn \cite{Fishburn}, Freiman \cite{Freiman},  Grynkiewicz and Serra \cite{Grynkiewicz}, and Stanchescu \cite{Stanchescu1,Stanchescu1.5,Stanchescu2,Stanchescu3}). For $k=1$ Freiman's $3|A|-4$ theorem \cite{Freiman}, \Cref{3A-4}, and subsequent improvements by Jin \cite{Jin} go beyond this threshold, but even for $k=2$ there do not appear to have been any such results beyond $2^k-\delta$.

Our \Cref{mainthm} formally implies the Freiman-Bilu theorem, and extends the scope of the theorem beyond the $2^k$ threshold.

Finally, we recall Freiman's $3|A|-4$ theorem, which marked the beginning of the study of inverse problems in additive combinatorics.
\begin{thm}[Freiman's $3|A|-4$ theorem \cite{Freiman}]
\label{3A-4}
Let $A\subset \mathbb{Z}$ be a subset of the integers with $d_1(A)\le |A|-4$. Then $|\coo(A)\setminus A| \le d_1(A)+1$.
\end{thm}
This result is sharp, both in the linear bound on $d_1(A)$ in terms of $|A|$ (there are examples of sets $A$ with fixed $d_1(A)=|A|-3$ and $|\coo(A)\setminus A|$ arbitrarily large in terms of $|A|$), and in the linear bound on $|\coo(A)\setminus A|$ in terms of $d_1(A)$. 

\Cref{maincor} generalizes Freiman's $3|A|-4$ theorem to arbitrary dimension, achieving the sharpest possible asymptotics (but not the sharpest possible constants $c_k$, and $\Delta_k$ realizing the bound $\delta \ll 1$, which would be extremely interesting to determine). As remarked earlier, the new restriction in \Cref{maincor} on the thickness of $A$ is necessary, but has no analogue when $k=1$ (as subsets $A\subset \mathbb{Z}$ cannot exhibit lower dimensional degeneracies). 

We finish by mentioning a particularly nice intermediate result we show during the proof of \Cref{qual}, whose optimal constants seem to be an important bottleneck for improving the constant $c_k$. For a real-valued function $f$ on a convex progression $A=\coo(A)$, we define the \emph{infimum-convolution} (see e.g. Str\"{o}mberg's extensive survey \cite{Infconvsurvey}) $f^{\square}:A+A\to \mathbb{R}$ by
$$f^{\square}(z)=\min_{x+y=z}\{f(x)+f(y)\}.$$
\begin{thm}[\Cref{squareprop}]
\label{squarepropintro}
There exist constants $c_{k+1}'<(4k)^{5k}$ and $g_{k+1}'$ such that the following is true. Let $A\subset \mathbb{Z}^k$ be the lattice points inside a simplex with vertices in $\mathbb{Z}^k$, with $A$ of thickness at least $h$, and $f:A\to [0,1]$ a function. Then for $\widehat{f}:A\to [0,1]$ the lower convex hull function, we have
$$\sum_{x\in A}(f-\widehat{f})(x)\le c'_{k+1}\left( 2^{k+1}\sum_{x\in A}f(x)-\sum_{x'\in A+A} f^{\square}(x')\right)+g_{k+1}'h^{-1}|A|.$$
\end{thm}
 \Cref{squarepropintro} follows from \Cref{mainthm} (with the better exponent $h^{-1}$ because $\cooo(A)$ is a simplex, see \Cref{rmk:betterexponent}), applied to the epigraph $A'_{f,N,M}=\{(a,x):a\in A,x \in [Nf(a),M]\}\subset A\times [0,M]$ for $1 \ll N \ll M$ large constants for fixed $A$ (note that the choice of $M$ allows us to avoid having the condition of having small doubling).  
 
 As mentioned in our theorem statements, we can take $c_k,c_k'<(4k)^{5k}$. We also have the lower bound $c_k \ge \frac{2^k}{k}$, attained for example by the set $A$ with $ \frac{1}{n}A:=((\widetilde{T} \times [-2,0]) \cup (V(\widetilde{T}) \times \{1\})) \cap (\frac{1}{n}\mathbb{Z})^k$, where $\widetilde{T}\subset \mathbb{R}^{k-1}$ is a fixed simplex with vertices $V(\widetilde{T})\subset \mathbb{Z}^{k-1}$ and $n$ sufficiently large. Similarly, $c_k'\ge \frac{2^k}{k}$ by taking the functional version of this example, namely with $f:(n\widetilde{T}\cap \mathbb{Z}^k)\to [0,1]$ whose value is $0$ at the vertices and $1$ elsewhere.
  We believe that the optimal values of $c_k,c_k'$ lie closer to the lower bound $\frac{2^k}{k}$.
\begin{quest}
What are the optimal values of $c_k$ and $c_k'$?
\end{quest}

\subsection{Outline of the Paper}
\label{outlineofpaper}
We start by proving  \Cref{1.1parta}, bootstrapping a result of Green and Tao \cite{GreenTao} that $A$ is covered by a bounded number of generalized arithmetic progressions of dimension $k$ and size at most $|A|$, to show that we can reduce to a single generalized arithmetic progression of size $O(|A|)$.

Once we have this result, we are able to work with equivalent reformulations of \Cref{quali} and \Cref{mainthm} involving sets $A$ of positive density $\epsilon_0$ inside thick boxes.

As mentioned before, most of the work is devoted to proving \Cref{quali}. The strategy is to construct in stages a highly structured set $A_{\star}$ from $A$ with the properties that $|A \Delta A_{\star}| =o_{\epsilon_0}(1) |A|$ and $d_k(A_{\star}) =o_{\epsilon_0}(1)|A_{\star}|=o_{\epsilon_0}(1)|A|$ (where for fixed $\epsilon_0$ we have $o_{\epsilon_0}(1)\to 0$ as $\delta \to 0$). The additional structure of $A_{\star}$ enables us to conclude that  $|\coo(A_{\star}) \setminus A_{\star}| =o_{\epsilon_0}(1)|A_{\star}|$, which finally implies that $|\coo(A) \setminus A| =o_{\epsilon_0}(1)|A|$.

At each stage we produce a new set $A_{\text{new}}$ from an existing set $A_{\text{old}}$ which satisfies $|A\Delta A_{\text{old}}|=o_{\epsilon_0}(1)|A|$ and $d_k(A_{\text{old}})=o_{\epsilon_0}(1)|A|$. With a single exception, this is done by throwing away for $x\in \mathbb{Z}^{n-1}$ ``rows'' $R_x:=A_{\text{old}}\cap \mathbb{Z}\times \{x\}$ in the first coordinate direction which are are in some sense unstructured. If we can show $|A_{\text{old}}\setminus A_{\text{new}}|=o_{\epsilon_0}(1)|A|$, then $|A\Delta A_{\text{new}}|=o_{\epsilon_0}(1)|A|$ and hence we have  $d_k(A_{\text{new}})=o_{\epsilon_0}(1)|A|$.

To bound $|A_{\text{old}} \setminus A_{\text{new}}|$ from above, we introduce an operation $(+)$ in order to create a ``reference set'' $A_{\text{old}}(+)A_{\text{old}}\subset A_{\text{old}}+A_{\text{old}}$, whose size we can guarantee to be approximately $2^k|A_{\text{old}}|$. For one dimensional sets $X, Y$ we define $X(+)Y:= (X+\min Y)\cup (Y+ \max X) \subset X+Y$; in general, we define $A_{\text{old}}(+)A_{\text{old}}:=\bigsqcup_{\vec{v}\in \{0,1\}^{k-1}}\bigsqcup_{x}R_x(+)R_{x+\vec{v}} \subset A_{\text{old}}+A_{\text{old}}$.

In order to control the size of the unstructured rows $U:=\bigsqcup_{x\text{ - unstructured}}R_x$, we construct a set $D \subset U+A_{\text{old}} \subset A_{\text{old}}+A_{\text{old}} $ of comparable size to $U$, disjoint from $A_{\text{old}}(+)A_{\text{old}}$. Then, we will obtain $|A_{\text{old}} \setminus A_{\text{new}}|=|U| \approx |D| \le |A_{\text{old}}+A_{\text{old}}| - |A_{\text{old}}(+)A_{\text{old}}| \approx d_k(A_{\text{old}})=o_{\epsilon_0}(1)|A|$.

For the last step in the proof of \Cref{quali} and for the proof of \Cref{mainthm}, we use two versions of an argument inspired by the one used in \cite{HomoBM}. For the last part of \Cref{quali}, we prove that functions on convex domains with small infimum-convolution are close to their convex hulls (which is essentially \Cref{squarepropintro}). For \Cref{mainthm}, we proceed as follows. By choosing an appropriately small $\Delta_k(\epsilon_0),$ \Cref{quali} ensures $|\coo(A)\setminus A|\cdot |A|^{-1}$ is as small as we like. This guarantees a large interior region of $\coo(A+A)$ is contained in $A+A$. We control the size of $\coo(A)\setminus A$ by inductively controlling the size of $\coo(A)\setminus A$ restricted to certain homothetic copies of $\widetilde{\co}(A)$ used to cover a thickened boundary of $\widetilde{\co}(A)$. This will allow us to show that $|\coo(A+A)\setminus (A+A)|\leq (2^k-c_k')|\coo(A)\setminus A|+o_{\epsilon_0}(1)|B|$ for some constant $c_k'$, which allows us to conclude \Cref{mainthm}.

In Section 2, we prove \Cref{1.1parta}. In Section 3, we use \Cref{1.1parta} to establish equivalent versions \Cref{equivquali} and \Cref{equivmainthm} of \Cref{quali} and \Cref{mainthm} respectively involving positive density subsets of boxes. In Section 4, we make some initial definitions, conventions and observations, which will be used throughout the remainder of the paper. Finally, in Sections 5 and 6, we prove  \Cref{equivquali} and \Cref{equivmainthm} with a simultaneous induction on dimension.

\section{Proof of \Cref{1.1parta}}
\label{1.1asection}
We will now prove \Cref{1.1parta}. To do this, we will need the following special case of the main result of Green and Tao \cite{GreenTao}.

\begin{thm}[Special case of \cite{GreenTao}]\label{GTspecialcase}
There exist constants $w_k$ such that for any $A\subset \mathbb{Z}^k$ with $d_k(A)\le |A|$, there exists a generalized arithmetic progression $P$ of dimension at most $k$ with $|P|\le |A|$, along with vectors $x_1,\ldots,x_{w_k}$ such that $$A\subset \bigcup_{i=1}^{w_k}P+x_i.$$
\end{thm}

\begin{proof}[Proof of \Cref{1.1parta}]

We apply \Cref{GTspecialcase}, obtaining a generalized arithmetic progression $P$ and $x_1,\ldots,x_{w_k}\in \mathbb{Z}^k$ such that $A\subset \bigcup_{i=1}^{w_k}P+x_i$. Take $n_k^0$ to be a large threshold, chosen later. If $P$ is contained inside a hyperplane, then $A$ is covered by $w_k$ parallel hyperplanes and we are done. Similarly, if one of the side lengths of $P$ is at most $n_k^0$, then we can cover $P$ by $n_k^0$ parallel hyperplanes, so $A$ can be covered by $w_kn_k^0$ parallel hyperplanes. Therefore, we may assume that $$P=B(n_1,\ldots,n_k;v_1,\ldots,v_k;0)$$ is non-degenerate and $n_i\ge n_k^0$ for all $i$. By applying a linear transformation from $GL_k(\mathbb{Q})$ taking $v_i$ to the standard basis vectors $e_i$ and then scaling up to clear denominators, we may assume that $v_i=be_i$, where  $b\in \mathbb{N}$.

\begin{clm}\label{clm5.2}
There exist a factor $b'$ of $b$ such that $b'\ge w_k!^{-w_k} b$ and the following holds. If we consider the decomposition
$$A=A_1\sqcup \ldots \sqcup A_r$$
with $r\le w_k$ associated to the cosets $y_1,\ldots,y_r\in (
\mathbb{Z}/b'\mathbb{Z})^k$, then after possibly relabeling we have $|A_1|\ge |A_j|$ for all $j$, and for every $p\ne 1$ there exists $j_p \ne 1$ such that
$$y_p+y_{j_p}\ne y_1+y_k$$ for $k\in \{1,\ldots,r\}$.
\end{clm}

\begin{proof}
Set $b_{0}'=b$, and $y_{1,0},\ldots,y_{r_0,0}\in (\mathbb{Z}/b_{0}'\mathbb{Z})^k=( \mathbb{Z}/b\mathbb{Z})^k$ be the distinct representatives of $A$, or equivalently the distinct representatives of $x_1,\ldots,x_{w_k}$. We note that in particular, this implies that $r_0\le w_k$.

We recursively construct factors $b_{j+1}'$ of $b_{j}'$ with $b_{j+1}'\ge w_k!^{-1} b_{j}'$ such that if $y_{1,j},\ldots,y_{r_j,j}\in (\mathbb{Z}/(b'_{j}\mathbb{Z}))^k$ are the distinct representatives of $A$, then the following is true. If we consider the associated coset decomposition
$$A=A_{1,j}\sqcup \ldots \sqcup A_{r_j,j},$$
possibly relabeling so that $|A_{1,j}|\ge |A_{p,j}|$ for $1 \le p \le r_j$, then either for every $p\ne 1$ there exists a $\lambda(p,j)\ne 1$ such that
$$y_{p,j}+y_{\lambda(p,j),j}\ne y_{1,j}+y_{\ell,j}$$ for $1 \le \ell \le r_j$, or else we have $r_{j+1}<r_j$.

Suppose that $y_{p,j}$ does not have the property that there exists $\lambda(p,j)$ such that $y_{p,j}+y_{\lambda(p,j),j}\ne y_{1,j}+y_{\ell,j}$ for all $\ell$. This is equivalent to saying that $y_{p,j}$ has the property that for all $\lambda$, there is an $\ell$ such that $(y_{p,j} - y_{1,j})+ (y_{\lambda,j} -y_{1,j})= y_{\ell,j}-y_{1,j}$. Then the cyclic group generated by $y_{p,j}-y_{1,j}$ lies entirely inside $\{0,y_{2,j}-y_{1,j},\ldots,y_{r_j,j}-y_{1,j}\}$, so has order at most $r_j \le w_k$.  Setting $b_{j+1}'=b_{j}'/\gcd(b_{j}',w_k!)$, we obtain that $y_{p,j}-y_{1,j}=0$ in $( \mathbb{Z}/b_{j+1}'\mathbb{Z})^k$, so $r_{j+1}<r_j$.

As $r_j$ can decrease at most $w_k$ times from $r_0\le w_k$, there exists a $j\le w_k$ for which $r_j=r_{j+1}$. Taking $j_p=\lambda(p,j)$, $b'=b_{j}'$ and $y_p=y_{p,j}$ the distinct representatives of $A$ in $(\mathbb{Z}/b_{j}'\mathbb{Z})^k=(\mathbb{Z}/b'\mathbb{Z})^k$, we obtain the desired result.
\end{proof}

Returning to the proof of \Cref{1.1parta}, let $P'=B(n_1,\ldots,n_k;b'e_1,\ldots,b'e_k;0)$ where $b'$ is furnished by \Cref{clm5.2}. Let $w_k'=w_k \cdot w_k!^{k\cdot w_k}$, and $x_1',\ldots,x_{w_k'}'$ be translation vectors such that $$A\subset \bigcup_{i=1}^{w_k'}P'+x_i'.$$
Because $|P'|=|P|$ and $|P|\le |A|$ (from \Cref{GTspecialcase}) we also have $|P'|\le |A|$.

\begin{clm}\label{clm5.3}
There exists an $x$ such that $A\subset x+(b'\mathbb{Z})^k$.
\end{clm}
Before we begin the proof of the claim we need the following lemma.

\begin{lem}\label{error}
For  $X,Y\subset A$, then
$$|X+Y|\ge 2^k\min(|X|,|Y|)-2^{2k}(w_k'n_k^0)^{-1}w_k'^k|A|.$$
\end{lem}
\begin{proof}
Let $C(X),C(Y)$ be obtained by compressing $X,Y$ in each of the coordinate directions. Then $C(X),C(Y)$ are contained in the down-set $C(A)\subset C(\bigcup P'+x_i')$, which in turn is contained inside a box of side lengths $w_k'n_1,\ldots,w_k'n_k\ge w_k'n_k^0$, which has volume at most $w_k'^k\prod n_i \le w_k'^k|A|$ (as $|P'|\le |A|$). Therefore by \cite[Corollary 2.7]{GreenTao}, we obtain
$$|X+Y|\ge |C(X)+C(Y)|\ge 2^k\min(|X|,|Y|)-2^{2k}(w_k'n_k^0)^{-1}w_k'^k|A|.$$
\end{proof}

\begin{proof}[Proof of \Cref{clm5.3}]Let $$A=A_1\sqcup \ldots \sqcup A_r$$ be the coset decomposition as in \Cref{clm5.2} with $|A_1|$ maximal, and $r \le w_k'$. We want to show that $r=1$.

We have for any $p\ne 1$ that
\begin{align*}|A+A|&\ge |A_p+A_{j_p}|+\sum_i|A_1+A_i|\\
&\ge |A_p|+\sum_{i=1}^r (2^k|A_i|-2^{2k}(w_k'n_k^0)^{-1}w_k'^k|A|)\\
&\ge |A_p|+2^k|A|-2^{2k}(n_k^0)^{-1}w_k'^k|A|,\end{align*}
so averaging over all $p\ne 1$, we obtain
\begin{align}\label{5.3first}d_k(A)\ge \frac{1}{w_k'-1}|A\setminus A_1|-2^{2k}(n_k^0)^{-1}w_k'^k|A|.\end{align} On the other hand, assuming $r \ge 2$ we have
$$|A+A|\ge |A_1+A_2|+|A_1+A_1| \ge |A_1|+2^k|A_1|-2^{2k}(w_k'n_k^0)^{-1}w_k'^{k}|A|$$
so
\begin{align}\label{5.3second}d_k(A)\ge |A|-(2^k+1)|A\setminus A_1|-2^{2k}(w_k'n_k^0)^{-1}w_k'^k|A|.\end{align} Adding $(w_k'-1)(2^k+1)$ of \eqref{5.3first} to \eqref{5.3second}, and using $\Delta_k|A| \ge d_k(A)$, we obtain  $$((w_k'-1)(2^k+1)+1)\Delta_k|A| \ge |A| - ((w_k'-1)(2^k+1)+w_k'^{-1})2^{2k}(n_k^0)^{-1}w_k'^k|A|,$$ which gives the desired contradiction provided $\Delta_k$ is sufficiently small and $n_k^0$ is sufficiently large.
\end{proof}

Returning to the proof of \Cref{1.1parta},  after translating $A$ we can assume that $A\subset (b'\mathbb{Z})^k$, so we may scale down and assume that $b'=1$.

We now show that the boxes are in some sense ``near'' each other.

\begin{clm}
There exists a constant $f_k$ so that for $P''=B(f_kn_1,\dots,f_kn_k;e_1,\dots,e_k;0)$ we have that
$A\subset P''+x$ for some $x$ (we note $f_k=2^{w_k'}-1$ works).
\end{clm}
\begin{proof}
Recall that $x_1',\ldots,x_{w_k'}'$ are the translation vectors for $P'=B(n_1,\dots,n_k;e_1,\dots,e_k;0)$ which cover $A$. Suppose that $|A\cap (P'+x_1)|$ is maximal, so $|A\cap (P'+x_1)|\ge \frac{1}{w_k'}|A|$. Let $A_1=A\cap (P'+x_1)$.

We first show that the width in the $j$-direction is bounded by a fixed multiple of $n_j$ for each $j$. So fix a $j$ and let $\pi_j(x)$ denote the $j$'th coordinate of $x$. For a subset $\mathcal{B}\subset \{1,\ldots,w_k'\}$, let $h_j(\mathcal{B})$ be the difference between the largest and smallest values in $\pi_j(\bigcup_{i\in \mathcal{B}'}(P'+x_i'))$. Suppose we have a set $\mathcal{B}\subset \{1,\ldots,w_k'\}$ containing $1$. If $\mathcal{B}$ is not the whole set and there is no $i \not \in \mathcal{B}$ such that $h_j(\mathcal{B}\cup \{i\})\le 2h_j(\mathcal{B})+n_j$, then taking $\mathcal{B}'$ to be either $\{x_\ell': \pi_j(x_\ell')\ge \min \{\pi_j(x_i')\}_{i \in \mathcal{B}}\}$ or $\{x_\ell': \pi_j(x_\ell')\le \max \{\pi_j(x_i')\}_{i \in \mathcal{B}}\}$ (whichever has $\mathcal{B}'^c$ non-empty), the sets $$Z_1=\bigcup_{i \in \mathcal{B}'} P'+x_i',\quad Z_2=\bigcup_{i\in (\mathcal{B}')^c}P'+x_i'$$ have the following property. Let $z$ be the closest point of $Z_2\cap A$ in the $e_j$-direction to $Z_1$. Then by considering the projections under $\pi_j$, the sets
$$(Z_1\cap A)+(Z_1\cap A),z+A_1,(Z_2\cap A)+(Z_2\cap A)$$
are disjoint. For example, if $\mathcal{B}'$ was equal to the first of these possibilities, then $\min\pi_j(z+A_1)=\min \pi_j(Z_2\cap A)+\min \pi_j(A_1) \ge  \min \pi_j(Z_2\cap A)+\max \pi_j(Z_1)-h_j(\mathcal{B})> 2\max \pi_j(Z_1\cap A)$ and $2\min\pi_j(Z_2\cap A)=2\pi_j(z)> \pi_j(z)+\max\pi_j(Z_1)=\pi_j(z)+\max \pi_j(A_1)=\max \pi_j(z+A_1)$.

By \Cref{error} applied to these sets we obtain
\begin{align*}|A+A|&\ge |(Z_1\cap A)+(Z_1\cap A)|+|(Z_2\cap A)+(Z_2\cap A)|+|z+A_1|\\
&\ge 2^k|Z_1\cap A|+2^k|Z_2\cap A|+|A_1|-2^{2k+1}(w_k'n_k^0)^{-1} w_k'^k|A|\\
&\ge 2^k|A|+\frac{1}{w_k'}|A|-2^{2k+1}(w_k'n_k^0)^{-1} w_k'^k|A|,
\end{align*}
a contradiction provided $\Delta_k$ is sufficiently small and $n^0_k$ sufficiently large. Hence, if $\mathcal{B}$ is not the whole set $\{1,\dots,w_k'\}$, then there is an $i\not\in \mathcal{B}$ such that $h_j(\mathcal{B}\cup\{i\})\leq 2h_j(\mathcal{B})+n_j$.

Start with $\mathcal{B}=\{1\}$ and $h_j(\{1\})=n_j-1$. Repeatedly applying this, we see that $h_j(\{1,\ldots,w_k'\})\le (2^{w_k'}-1)n_j-1$.

Then for $f_k=2^{w_k'}-1$, we deduce there is an $x$ such that $P'+x_i'\subset P''+x$ for all $i$.\end{proof}

Returning to the proof of \Cref{1.1parta}, taking $\epsilon_k\leq f_k^{-k}$, then as $|P'|\le |A|$, we find $A$ is a set of density at least $\epsilon_k$ inside the generalised arithmetic progression $P''+x$.
\end{proof}

\section{Equivalent statements for positive density subsets of boxes}
\label{equivsection}
To prove \Cref{quali} and \Cref{mainthm}, we will work with a slightly different formulation involving positive density subsets of boxes by invoking \Cref{1.1parta}. This new formulation will have new parameters $n_{k,0}$ and $\epsilon_0$ in place of $h$, with dependencies $$n_{k,0}^{-1}\ll \delta \ll \epsilon_0 \le 1.$$

\begin{defn}
Let $B(n_1,\ldots,n_k):=\prod_{i=1}^k \{1,\ldots,n_i\}\subset \mathbb{Z}^k$.
\end{defn}
\begin{thm}[Equivalent Reformulation of \Cref{quali}]
\label{equivquali}
There exist $\ll$-dependencies $n_{k,0}^{-1}\ll \delta \ll \epsilon_0 \le 1$ and a function $\omega(\delta)\to0$ as $\delta\to 0$ such that the following is true. If $A\subset B=B(n_1,\ldots,n_k)$ with $\min \{n_i\} \ge n_{k,0}$, has $|A|\ge \epsilon_0|B|$ and $d_k(A)\le \delta |B|$, then
$$|\coo(A)\setminus A||A|^{-1}\le \omega(\delta).$$
\end{thm}
\begin{proof}[Proof that \Cref{equivquali} is equivalent to \Cref{quali}]
First, note that by replacing $\delta$ with $\delta\epsilon_0^{-1}$, \Cref{equivquali} is equivalent to the same statement except with $d_k(A)\le \delta |B|$ replaced with $d_k(A)\le \delta |A|$. It is this modified statement that we will show is equivalent to \Cref{quali}.

Let us assume the modified statement \Cref{equivquali} is true, and suppose we have an $A$ satisfying the hypotheses of \Cref{quali}. For the constants $m_k$ and $\Delta_k$ from \Cref{1.1parta}, we may choose the $h^{-1}\ll \delta \ll 1$ sufficiently strong so that $\delta \le \Delta_k$ and $h^{-1} \le m_k^{-1}$ so that we may apply \Cref{1.1parta}. This gives a generalized arithmetic progression $B=B(n_1,\ldots,n_k;v_1,\ldots,v_k;b)$ with $v_1,\ldots,v_k$ linearly independent, containing $A$ with $|A|\ge \epsilon_k|B|$. As the vectors $v_1,\ldots,v_k$ and the translation $b$ do not affect these statements, we may assume $B=B(n_1,\ldots,n_k)$. Because $B$, and hence $A$, can be covered with  $\min\{n_i\}$ parallel hyperplanes, we have $\min\{n_i\}\ge h$. Therefore, taking $\epsilon_0=\epsilon_k$ in the modified statement \Cref{equivquali}, we may then choose the $\ll$-dependencies $h^{-1}\ll \delta \ll 1$ sufficiently strong so that they imply the required $\ll$-dependencies in $n_{k,0}^{-1}\ll \delta \ll \epsilon_k=\epsilon_0 \le 1$ for the modified statement \Cref{equivquali} and we conclude.

Conversely, suppose \Cref{quali} is true and we have an $A$ satisfying the hypotheses of the modified statement \Cref{equivquali}. Then by \Cref{hypboxsmall}, every hyperplane intersects $B$, and hence $A$, in at most $\min \{n_i\}^{-1}|B|$ elements. Hence since $|A|\ge \epsilon_0|B|$, we need at least $\min\{n_i\}\epsilon_0$ parallel hyperplanes to cover $A$, so the thickness of $A$ is at least $\min\{n_i\}\epsilon_0\ge n_{k,0}\delta$. Choosing the $\delta\ll \epsilon_0$ dependency sufficiently strong that it implies the $\delta \ll 1$ dependency of \Cref{quali} and the $n_{k,0}^{-1}\ll \delta$ dependency sufficiently strong so that $n_{k,0}\delta$ is at least the smallest $h$ satisfying the $h^{-1}\ll \delta$ dependency of \Cref{quali}, we may therefore apply \Cref{quali} and conclude.
\end{proof}
\begin{thm}[Equivalent reformulation of \Cref{mainthm}]
\label{equivmainthm}
There exists a constant $c_k<(4k)^{5k}$ and $\ll$-dependencies $\delta \ll \epsilon_0 \le 1$ and constants $g_k(\epsilon_0)$ such that if $A\subset B=B(n_1,\ldots,n_k)$ with $|A|\ge \epsilon_0|B|$ and $d_k(A)\le \delta |B|$, then
$$|\coo(A)\setminus A|\le c_kd_k(A)+g_k(\epsilon_0)\min\{n_i\}^{-\frac{1}{1+\frac{1}{2}(k-1)\lfloor k/2 \rfloor}}|A|.$$
\end{thm}
\begin{proof}[Proof that \Cref{equivmainthm} is equivalent to \Cref{mainthm}]
First, note that by replacing $\delta$ with $\delta\epsilon_0^{-1}$, \Cref{equivmainthm} is equivalent to the same statement except with $d_k(A)\le \delta |B|$ replaced with $d_k(A)\le \delta |A|$. It is this modified statement that we will show is equivalent to \Cref{mainthm}.

Let us assume the modified statement \Cref{equivmainthm} is true and suppose we have an $A$ satisfying the hypotheses of \Cref{mainthm}. We may choose $m_k$ sufficiently large and $\Delta_k$ sufficiently small in \Cref{mainthm} so that we may apply \Cref{1.1parta}. This gives a generalized arithmetic progression $B=B(n_1,\ldots,n_k;v_1,\ldots,v_k;b)$ with $v_1,\ldots,v_k$ linearly independent, containing $A$ with $|A|\ge \epsilon_k|B|$. As the vectors $v_1,\ldots,v_k$ and the translation $b$ do not affect these statements, we may assume that $B=B(n_1,\ldots,n_k)$. Taking $\epsilon_0=\epsilon_k$ constant, we may then choose $\Delta_k$ sufficiently small so that the $\delta \ll \epsilon_0=\epsilon_k$ dependency of the modified statement \Cref{equivmainthm} is satisfied. Therefore, we deduce
$$|\coo(A)\setminus A|\le c_kd_k(A)+g_k(\epsilon_k)\min\{n_i\}^{-\frac{1}{1+\frac{1}{2}(k-1)\lfloor k/2 \rfloor}}|A|.$$
Because $B$, and hence $A$, is covered by $\min\{n_i\}$ parallel hyperplanes, we have $\min\{n_i\}\ge h$, and we conclude \Cref{quali} with $g_k=g_k(\epsilon_k)$.

Conversely, suppose \Cref{mainthm} is true and we have an $A$ satisfying the hypotheses of the modified statement \Cref{equivmainthm}. Then by \Cref{hypboxsmall}, every hyperplane intersects $B$, and hence $A$, in at most $\min\{n_i\}^{-1}|B|$ elements. Hence since $|A|\ge \epsilon_0|B|$, we need at least $\min\{n_i\}\epsilon_0$ parallel hyperplanes to cover $A$, so the thickness of $A$ is $\ge \min\{n_i\}\epsilon_0$. For the constant $m_k$ from \Cref{mainthm}, if $\min\{n_i\}\le m_k\epsilon_0^{-1}$ then by taking $g_k(\epsilon_0)$ sufficiently large, we can ensure $c_kd_k(A)+g_k(\epsilon_0)\min\{n_i\}^{-\frac{1}{1+\frac{1}{2}(k-1)\lfloor k/2 \rfloor}}|A| \ge -2^kc_k|A|+(2^kc_k+\epsilon_0^{-1})|A|\ge |B|$ so \Cref{equivmainthm} holds trivially. Hence we may assume that the thickness of $A$ is at least $m_k$. Now, for the constant $\Delta_k$ from \Cref{mainthm}, we may take the $\delta\ll \epsilon_0$ dependency sufficiently strong so that we may assume $d_k(A)\le \Delta_k|A|$. Therefore we may apply \Cref{mainthm}, and conclude the modified statement \Cref{equivmainthm} with $g_k(\epsilon_0)=\max(g_k\epsilon_0^{-\frac{1}{1+\frac{1}{2}(k-1)\lfloor k/2 \rfloor}},(2^kc_k+\epsilon_0)(m_k\epsilon_0^{-1})^{\frac{1}{1+\frac{1}{2}(k-1)\lfloor k/2 \rfloor}})$.
\end{proof}

\section{Definitions, Conventions, and Observations}
In this section, we introduce our definitions and conventions, as well as observations we will be using throughout the remaining sections.
\subsection{Definitions and Conventions}
As was noted in the introduction, we will use the notation
\begin{align*}a\ll b\text{ meaning }&a\le b\text{ and there exists a fixed increasing function $f$}\\ &\text{depending only on $k$ such that }a\le f(b).\end{align*} When we write
$$f\ll g,h$$
we mean that separately $f\ll g$ and $f\ll h$, and when we write $$f\ll g \ll h$$
we mean separately $f\ll g$ and $g\ll h$.
\begin{defn}
For $A'\subset \mathbb{Z}^k$, we introduce the following notation:
\begin{itemize}
    \item $\widetilde{\co}(A')\subset \mathbb{R}^k$ for the convex hull,
    \item $\co(A')=\widetilde{\co}(A')\cap \mathbb{Z}^k$,
    \item  $\Lambda_{A'}=\langle A'-a\rangle+a\subset \mathbb{Z}^k$ for any $a\in A'$, the affine sublattice of $\mathbb{Z}^k$ spanned by $A'$, and
    \item $\coo(A')=\co(A')\cap \Lambda_{A'}$, the smallest convex progression containing $A'$.
\end{itemize}
\end{defn}
\begin{defn}
We say that $A'$ is \emph{reduced} if $\Lambda_{A'}=\mathbb{Z}^k$.
\end{defn}
We will typically denote regions of $\mathbb{R}^k$ with a tilde such as $\widetilde{A'}\subset \mathbb{R}^k$. 
By abuse of notation, we will use $|\cdot |$ to refer both to cardinality of sets, and for volumes of sets. It will be clear with the tilde notation whether we intend to use discrete or continuous volume, and from context what dimension we are considering.
\begin{conv}
When we define a polytope or affine subspace $\widetilde{P}\subset \mathbb{R}^k$, we let $P=\widetilde{P}\cap \mathbb{Z}^k$.
\end{conv}
We recall that for numbers $n_1,\ldots,n_k$, we defined the discrete box
$$B(n_1,\ldots,n_k)=\prod_{i=1}^k \{1,\ldots,n_i\}.$$
We write $B$ instead of $B(n_1,\ldots,n_k)$ when $n_1,\ldots,n_k$ are clear from context.

\begin{defn}
We define the projection $$\pi:\mathbb{Z}^k=\mathbb{Z}\times\mathbb{Z}^{k-1}\to \{0\}\times\mathbb{Z}^{k-1}$$ given by $\pi(x_1,\ldots,x_k)=(0,x_2,\ldots,x_k)$ to be the projection away from the first coordinate.\end{defn}
\begin{defn}
For $\widetilde{A'}\subset \mathbb{R}^k$ a subset such that $\cooo(\widetilde{A'})$ is a polytope with integral vertices, we define $V(\widetilde{A'})\subset \mathbb{Z}^{k}$ to be the vertices of $\cooo(\widetilde{A}')$, and $V_\pi(\widetilde{A'}):=\pi(V(\widetilde{A}'))\subset \{0\}\times \mathbb{Z}^{k-1}$.
\end{defn}
\begin{defn}
A \emph{row} of $A'\subset \mathbb{Z}^k$ is $R_x=\pi^{-1}(x)\cap A'$ for some $x\in \{0\}\times \mathbb{Z}^{k-1}$.
\end{defn}
\begin{conv}
When talking about the rows of a set $A'$, we will use the notation $R_x$ without further clarification. It will always be clear from context which set $A'$ is being referred to.
\end{conv}
\begin{defn}
For $X,Y\subset \mathbb{Z}$, define
$X(+)Y:=(X+\min Y)\cup (Y+\max X)\subset X+Y$ if $X,Y$ are both nonempty, and $\emptyset$ otherwise. For $A'\subset \mathbb{Z}^k$, we define the disjoint union
$$A'(+)A':=\bigsqcup_{\vec{v}\in \{0\}\times\{0,1\}^{k-1}}\bigsqcup_{x\in \{0\}\times \mathbb{Z}^{k-1}}R_x(+)R_{x+\vec{v}}\subset A'+A'.$$
\end{defn}

Finally, we introduce a small constant which we will use to absorb errors into exponents through the paper.
\begin{defn}\label{cdef}
We let $c=10^{-10}$.
\end{defn}

\subsection{Observations}
The first observation guarantees that hyperplanes $\widetilde{H}$ have small intersections with discrete boxes. In particular, large subsets of $B$ cannot be covered by few hyperplanes.
\begin{obs}
\label{hypboxsmall}
Given a hyperplane $\widetilde{H}$ and a box $B=B(n_1,\ldots,n_k)$, we have
$$|\widetilde{H}\cap B|\le \min\{n_i\}^{-1}|B|.$$
In particular, a subset $A'\subset B$ with $|A'|>m\min\{n_1\}^{-1}|B|$ cannot be covered by $m$ hyperplanes.
\end{obs}
\begin{proof}
There exists an $i$ such that $e_i$ is not parallel to $H$. Let $\pi_i$ be the projection $\mathbb{Z}^k\to \mathbb{Z}^{k-1}$ omitting the $i$'th coordinate. Then $\pi_i(\widetilde{H}\cap B)$ injects into $\pi_i(B)$, and therefore $|\widetilde{H}\cap B|\le n_i^{-1}|B|$.
\end{proof}
The next observation will be used later to assume $A$ is reduced in \Cref{equivmainthm} and \Cref{equivquali}.
\begin{obs}\label{nonreduced}
For all $\epsilon_0>0$, the following holds. Given $n_i$ sufficiently large in terms of $\epsilon_0$, for a subset $A'\subset B=B(n_1,\ldots,n_k)$ with $|A'|\geq \epsilon_0|B|$, we can find a subset $A''\subset B(2^kn_1,\ldots,2^kn_k)$ such that $A''$ is reduced, $|A'|=|A''|$,  $d_k(A')=d_k(A'')$, and $|\coo(A')\setminus A'|=|\co(A'')\setminus A''|$.
\end{obs}
\begin{proof}
By taking $n_i>\epsilon_0^{-1}$ for all $i$, we first note that $A'$ is not contained inside a hyperplane by \Cref{hypboxsmall}. Take some $a\in A'$. Then $A'-a\subset C:=\prod_{i=1}^k \{-n_i+1,\ldots,n_i-1\}$ and the affine sub-lattice $\Lambda_{A'-a}$ is actually a subgroup $\langle v_1, \hdots, v_k\rangle\subset \mathbb{Z}^k$ generated by linearly independent vectors $v_i=(v_{i,1}, \hdots, v_{i,k})$. 

Without loss of generality, suppose $n_1\le \ldots \le n_k$. By applying row operations to the matrix whose rows are $v_1,\ldots,v_k$ (as in the algorithm for Smith normal form), we may we may assume that $|v_{j, i}| = 0$ for all $j>i$, and $|v_{j,i}| \le |v_{i,i}|$ for all $j \le i$. Because $v_i,\ldots,v_k$ are linearly independent, we have $v_{i,i}\ne 0$. Consider a point $p=p_1v_1 + \hdots + p_kv_k \in C$. We will show by induction on $i$ that $|p_i|\le 2^{i-1}(n_i-1)$.
Indeed, by considering the $i$th coordinate, we have that $$|p_1 v_{1,i} + p_2 v_{2,i}+ \hdots +p_i v_{i,i}| \le n_{i}-1$$ and hence
\begin{align*}|p_i| |v_{i,i}| &\le |p_1| |v_{1,i}| + \ldots + |p_{i-1}| |v_{i-1,i}| + n_{i}-1 \\&\le (|p_1|+ \ldots + |p_{i-1}| + n_{i}-1)|v_{i,i}|\\&\le (1+\sum_{j=0}^{i-2}2^j)(n_i-1)|v_{i,i}|=2^{i-1}(n_i-1)|v_{i,i}|.
\end{align*} This shows that $A'-a\subset \{p_1 v_1 + \ldots p_kv_k \text{ : } |p_i| \le 2^{i-1}(n_{i}-1)  \}$. If we let $$A''':=\{(p_1, \hdots p_k) \text{ : } p_1v_1 + \hdots p_kv_k \in A'-a \}\subset \prod_{i=1}^k \{-2^{i-1}(n_i-1),\ldots,2^{i-1}(n_i-1)\},$$ then $A'''$ is reduced in $\mathbb{Z}^k$ since $\Lambda_{A'-a}=\langle v_1,\ldots,v_k\rangle$, and is obtained from $A'$ by applying an element of $GL_n(\mathbb{Q})$ followed by a translation, so $|A'|=|A'''|$, $d_k(A''')=d_k(A')$, and $|\co(A''') \setminus A'''| =|\coo(A')\setminus A'|.$ We conclude by taking $A''$ to be a suitable translation of $A'''$.
\end{proof}
We now prove an observation lower bounding $d_k$ for subsets of boxes, an easy corollary of a Lemma of Green and Tao \cite{GreenTao}.
\begin{obs}
\label{negdk}
For any subsets $X\subset B=B(n_1,\ldots,n_k)$ and $Y\subset \pi(B)$ we have $$d_{k}(X) \ge -2^{2k}\min\{n_i\}^{-1}|B|\text{, and }d_{k-1}(Y) \ge -2^{2(k-1)}\min\{n_i\}^{-1}n_1^{-1}|B|.$$
More generally, for $X_1,X_2\subset B$ and $Y_1,Y_2\subset \pi(B)$ we have
\begin{align*}|X_1+X_2|&\ge 2^{k}\min(|X_1|,|X_2|)- 2^{2k}\min\{n_i\}^{-1}|B|\text{, and}\\
|Y_1+Y_2|&\ge 2^{k-1}\min(|Y_1|,|Y_2|) -2^{2(k-1)}\min\{n_i\}^{-1}n_1^{-1}|B|.
\end{align*}
\end{obs}
\begin{proof}
Because $B$ and $\pi(B)$ are downsets, the result follows from  \cite[Lemma 2.8]{GreenTao}, and the trivial estimates that the size of each coordinate projection of $B+B$ and $\pi(B)+\pi(B)$ have sizes at most $2^k\min\{n_i\}^{-1}|B|$ and $2^{k-1}\min\{n_i\}^{-1}n_1^{-1}|B|$, respectively.
\end{proof}
We frequently need the following observation when considering $A'(+)A'$ to show it has size roughly $2^k|A'|$ as described in \Cref{outlineofpaper}.
\begin{obs}
\label{Surfobs}
Let $Y\subset \pi(B)$ with $B=B(n_1,\ldots,n_k)$, and let $0\ne \vec{v}\in \{0\}\times\{0,1\}^{k-1}$. Then
$$\left|\{x \in \{0\}\times \mathbb{Z}^{k-1} : |\{x,x+\vec{v}\}\cap \co(Y)|=1\}\right| \le 2(k-1) \min\{n_i\}^{-1}n_1^{-1}|B|.$$
\end{obs}
\begin{proof}
Consider all lines in the direction $\vec{v}$ intersecting $\pi(B)$. On each such line there are at most 2 values of $x$ such that $|\{x,x+\vec{v}\}\cap \co(Y)|=1$. Because each such line intersects two facets of $\pi(B)$, and each facet has size at most $\min\{n_i\}^{-1}|\pi(B)|$, there are at most $(k-1) \min\{n_i\}^{-1} |\pi(B)|$ such lines which intersect $\pi(B)$.
\end{proof}
The next observation relates $d_k$ between sets and subsets. In particular, it allows us to guarantee that all auxiliary sets we construct in the proof of \Cref{equivquali} are reduced and have similar $d_k$ solely because they are close in symmetric difference to the original set $A$.
\begin{obs}
\label{dkobs}
If $X\subset Y$, then \begin{align}\label{dk}d_k(X)\le d_k(Y)+2^k|Y\setminus X|.\end{align} In particular, there exists $\ll$-dependencies such that for $\min\{n_i^{-1}\}\delta \ll \epsilon_0$, for  reduced $A\subset B=B(n_1,\ldots,n_k)$ with $|A|\ge \epsilon_0|B|$, $d_k(A)\le \delta |B|$, if $A'\subset B$ has \begin{align}\label{redcond}|A\Delta A'|\le 2^{-(k+1)}\epsilon_0|B|,\end{align} then $A'$ is reduced.
\end{obs}
\begin{proof}
For \eqref{dk}, we have
$$d_k(X)=|X+X|-2^k|X|\le |Y+Y|-2^k|Y|+2^k|Y\setminus X|=d_k(Y)+2^k|Y\setminus X|.$$
For  \eqref{redcond}, it suffices to show $A\cap A'$ is reduced, so we may assume $A'\subset A$. Assume for the sake of contradiction that $A'$ is not reduced. Then there is an $a\in A$ such that $a+A'$ is disjoint from $A'+A'$. Hence, we have $|A+A|\ge |A'+A'|+|A'|$, and in particular, $$\delta |B| \ge d_k(A)\ge d_k(A')-(2^k+1)|A\setminus A'|+|A| \ge \left(-2^{2k}\min \{n_i\}^{-1}+\left(\frac{1}{2}-\frac{1}{2^{k+1}}\right)\epsilon_0\right)|B|.$$
Here the third inequality follows from \Cref{negdk} and the bound $|A|\ge \epsilon_0|B|$ from the hypothesis. The contradiction now comes from the fact that the $n_i^{-1}$ and $\delta$ can be chosen much smaller than $\epsilon_0$.
\end{proof}

We next have an observation which allows us to transition between convex sets and reduced convex progressions with a loss proportional to the surface area of a containing box.

\begin{obs}\label{contdisc}
Let $B=B(n_1,\ldots,n_k)$, and suppose we have a convex polytope $\widetilde{P}\subset \cooo(B)$. Then with $P=\widetilde{P}\cap \mathbb{Z}^k$, we have
$\left||\widetilde{P}|-|P|\right|\leq 2k(k+1) \min\{n_i\}^{-1}|B|$. This is more generally true for any subset $\widetilde{P}\subset\cooo(B)$ given as the intersection of finitely many open and closed half-spaces.
\end{obs}
\begin{proof}
By perturbing the defining half-spaces slightly, we may replace $\widetilde{P}$ with a polytope without changing $P$, so we assume $\widetilde{P}$ is a polytope from now on.

Consider the set $X:=\left\{z\in \mathbb{Z}^k: (z+[0,1]^k)\cap \partial\widetilde{P}\neq \emptyset\right\}.$ We first show $|X|$ is small.
\begin{clm}
$|X|\leq 2k(k+1)\min\{n_i\}^{-1}|B|$
\end{clm}
\begin{proof}[Proof of claim.]
For $1\le i \le k$, let $\pi_i:\mathbb{Z}^k\to\mathbb{Z}^{i-1}\times \{0\}\times \mathbb{Z}^{k-i}$ be the projection $\pi_i(x_1,\ldots,x_k)=(x_1,\ldots,x_{i-1},0,x_{i+1},\ldots,x_k)$. Let     $f_i^+,f_i^{-}:\pi_i(X) \to \mathbb{Z}$ be defined by
\begin{align*}
    f_i^+:x \mapsto \max(\pi_i^{-1}(x)\cap X)\qquad f_i^{-}:x \mapsto \min(\pi_i^{-1}(x)\cap X),
\end{align*}
and for every $x=(x_1,\ldots,x_{i-1},0,x_{i+1},\ldots,x_k)\in \pi_i(X)$, let \begin{align*}X^+_{i,x}&=\{(x_1,\dots,x_{i-1},j,x_{i+1},\dots,x_k):f^+_i(x)-k\leq j\leq f^+_i(x)\}\\ X^-_{i,x}&=\{(x_1,\dots,x_{i-1},j,x_{i+1},\dots,x_k):f^-_i(x)\leq j\leq f^-_i(x)+k\}\end{align*} be the $k+1$ elements of $\mathbb{Z}^k$ in the $x$-row in direction $i$ of $X$ below the maximum element and above the minimum element  respectively. From these definitions, it is immediate that $$\left|\bigcup_{i\in \{1,\ldots,k\},x\in \pi_i(X)} X_{i,x}^+\cup X_{i,x}^-\right|\leq 2k(k+1)\min\{n_i\}^{-1}|B|,$$ so it suffices to show that $X\subset\bigcup X_{i,x}^+\cup X_{i,x}^-$.

Suppose for the sake of contradiction that there is some $z\in X\setminus\left(\bigcup X_{i,x}^+\cup X_{i,x}^-\right)$. Then $$f^+_i(\pi_i(z))\geq z_i+k+1 \text{ and }f_i^-(\pi_i(z))\leq z_i-k-1$$ for all $i$, so there are $r_i^+,r_i^- \ge k+1$ such that $ z+r_i^+e_i+[0,1]^k$ and $z-r_i^-e_i+[0,1]^k$ intersect $\widetilde{P}$. As $z+[0,1]^k$ intersects $\widetilde{P}$ and  $\widetilde{P}$ is convex, for all $i\in [k]$ there are points $$y^+_i\in (z+(k+1)e_i+[0,1]^k)\cap \widetilde{P},\qquad y^-_i\in (z-(k+1)e_i+[0,1]^k)\cap \widetilde{P}.$$

Denoting $int$ for interior, we claim that $$z+[0,1]^k\subset int(\widetilde{co}(\{y^+_1,\dots,y^+_k,y^-_1,\dots,y^-_k\}))\subset int (\widetilde{P}).$$
The second inclusion is immediate, so we focus on the first. Write $y_i^+=z+(\frac{1}{2},\ldots,\frac{1}{2})+p_i^+$ and $ y_i^-=z+(\frac{1}{2},\ldots,\frac{1}{2})+p_i^-$ where $p_i^+= (k+1)e_i+\epsilon_i^+$ and $p_i^-= -(k+1)e_i+\epsilon_i^-$ with $\epsilon_i^{\pm}\in [-\frac{1}{2},\frac{1}{2}]^k$. Then this is equivalent to showing
$$\left[-\frac{1}{2},\frac{1}{2}\right]^k\subset int(\widetilde{\co}(\{p_1^+,\ldots,p_k^+,p_1^-,\ldots,p_k^-\})).$$
We will show that $\cooo(\{p_1^+,\ldots,p_k^+,p_1^-,\ldots,p_k^-\})$ has facets $\cooo(p_1^{\pm},\ldots,p_k^{\pm})$ for the $2^k$ choices of $\pm$, and $\left[-\frac{1}{2},\frac{1}{2}\right]^k$ lies on the same side of these facets as $\cooo(\{p_1^+,\ldots,p_k^+,p_1^-,\ldots,p_k^-\})$. To show this, let $\epsilon\in [-\frac{1}{2},\frac{1}{2}]^k$. We claim that it suffices to show $p_1^{+},\ldots,p_k^{+}$ are affinely independent, and that $\epsilon$ and $p_1^{-}$ lie on the same side of the hyperplane $\widetilde{H}$ through $p_1^{+},\ldots,p_k^{+}$ (and the analogous symmetrical statements where $1$ is replaced by some $j\in \{1,\ldots,k\}$ and the signs above $p_i$ are possible swapped for each $i$). Indeed, if this is the case, then by symmetry, all vertices lie on the same side of $\widetilde{H}$, which implies that  $\cooo(\{p_1^+,\ldots,p_k^+\})$ is a facet, and $\cooo(\{p_1^+,\ldots,p_k^+,p_1^-,\ldots,p_k^-\})$ lies on the same side of this facet as $\epsilon$. This is equivalent in turn to showing that, for $w\in \{p_1^-,\epsilon\}$, the determinants of the matrices whose columns are $p_i^+-w$ for $1 \le i \le k$ have the same signs. We will in fact show that this sign is positive for both.

For $w=\epsilon$, the matrix we are considering is $M+(k+1)I$ where $M$ has as its $i$th column $\epsilon_i^+-\epsilon$. Note that $M$ has entries of magnitude at most $1$, so the spectral radius of $M$ is at most $k$. But if $\det(M+(k+1)I)\le 0$, then, as $\det(M+\lambda I) \to \infty$, as $\lambda \to \infty$ there must exist $\lambda\ge k+1$ with $\det(M+\lambda I)=0$. But this would imply that $-\lambda$ is an eigenvalue and thus $|-\lambda|>k$ is at most the spectral radius, contradicting that the radius is at most $k$. Hence $\det(M+(k+1)I)>0$ as desired.

For $w=p_1^-$, note that we have already shown that $\epsilon_1^-$ lies on the positive side of $\widetilde{H}$, and $p_1^-\in \epsilon_1^-+\mathbb{R}_{\le 0}e_1$. Hence it suffices to show that the point $\epsilon_1^+-Ne_1$ lies on the positive side of $\widetilde{H}$ for all $N>0$ sufficiently large. This is equivalent to saying that the matrix $M_N$ whose $i$th column is $p_i^++Ne_1-\epsilon_1^+$ has positive determinant for all $N> 0$ sufficiently large. Subtracting the first column from all subsequent columns and then considering the coefficient of $N$ in $\det(M_N)$, which is now only contributed by the first column, this follows from an identical argument.

Hence we have $z+[0,1]^k\subset int(\widetilde{P})$, contradicting $(z+[0,1]^k)\cap \partial \widetilde{P}\ne \emptyset$.


\end{proof}

Returning to the proof of \Cref{contdisc}, consider the translates of $[0,1]^k$ by $P$, i.e., $P+[0,1]^k$. Each of these translates is either contained in $\widetilde{P}$ or intersects $\partial{\widetilde{P}}$. Hence, $|P|\leq |\widetilde{P}|+|X|$. On the other hand, consider the set of all integer translates of $[0,1]^k$ intersecting $\widetilde{P}$. All of these translates intersect $\partial{\widetilde{P}}$ or are of the form $a+[0,1]^k$ with $a\in P$. As these clearly cover $\widetilde{P}$, we find $|\widetilde{P}|\leq |P|+|X|$.
\end{proof}
Finally, the following observation implies $A'$ being close to its discrete convex hull implies $d_k(A')$ is small.
\begin{obs}\label{cvxsd}
Given a set $A'\subset B=B(n_1,\ldots,n_k)$, we have
$$d_k(A')\leq 2^k|\co(A')\setminus A'|+2^{k+2}k(k+1)\min\{n_i\}^{-1}|B|.$$
\end{obs}
\begin{proof}
By \Cref{dkobs}, we only need to show $d_k(\co(A'))\le 2^{k+2}k(k+1)\min\{n_i\}^{-1}|B|$. This follows as, by \Cref{contdisc} applied once to $\cooo(A')+\cooo(A')=2\cooo(A')$ and once to $\cooo(A')$, we have
\begin{align*}
d_k(\co(A'))&= |\co(A')+\co(A')|-2^k |\co(A')|\\
&\leq |(\cooo(A')+\cooo(A'))\cap \mathbb{Z}^k|-2^k |\co(A')|\\
&\leq |\cooo(A')+\cooo(A')|-2^k |\co(A')|+2^{k+1}k(k+1) \min\{n_i\}^{-1}|B|\\
&= 2^k|\cooo(A')|-2^k |\co(A')|+2^{k+1}k(k+1) \min\{n_i\}^{-1}|B|\\
&\leq 2^{k+2}k(k+1) \min\{n_i\}^{-1}|B|.
\end{align*}
\end{proof}

\section{Proof of \Cref{qual} for $k$ given \Cref{quant} for $k-1$}
\label{1.6section}
For $k=1$, \Cref{qual} and \Cref{quant} are implied by Freiman's $3|A|-4$ theorem \cite{Freiman}, \Cref{3A-4}, so we suppose from now on that $k\ge 2$.
In this section, we prove \Cref{qual} for dimension $k$ given \Cref{quant} for dimension $k-1$. A few important notes before we begin.
\begin{itemize}
    \item We be \textbf{exclusively} working in the equivalent reformulations \Cref{equivquali} (of \Cref{quali}) and \Cref{equivmainthm} (of \Cref{mainthm}) as established in \Cref{equivsection}. Our hypotheses on $A$ are therefore the ones from \Cref{equivquali}, that $n_{k,0}^{-1}\ll \delta \ll \epsilon_0\le 1$, and $A\subset B=B(n_1,\ldots,n_k)$ with $\min\{n_i\}\ge n_{k,0}$, $d_k(A)\le \delta |B|$, and $|A| \ge \epsilon_0|B|$, and our desired conclusion is still that $|\coo(A)\setminus A||A|^{-1}\le \omega(\delta)$ with $\omega(\delta)\to 0$ as $\delta\to 0$.
    \item By \Cref{nonreduced} we may and shall assume that $A$ is reduced.
    \item We will denote $\epsilon\ge \epsilon_0$ to be the density of $A$ in $B$, so we have \begin{align}\label{Ainfo}|A|=\epsilon|B|\text{,  and }d_k(A)\le \delta |B|.\end{align}
\end{itemize}

\subsection{Outline of the proof}
We will create sets
$$A\supset A_1\supset A_2\supset A_3\supset A_4\supset A_5\subset A_+\supset A_\star$$
(note that $A_5\subset A_+$)
such that $|A\Delta A_\star|$ is small, and $A_\star$ has a large number of properties which allow us to show that $A_\star$ is close to $\co(A_\star)$. From this we will be able to conclude that $A$ is close to $\co(A)$.

In \Cref{inductiveh}, we derive a general reduction to sets for which the projection under $\pi$ satisfies the induction hypothesis.

In \Cref{prelimreductions}, we construct $A\supset A_1\supset A_2\supset A_3$ such that $A_3$ is reduced, has large rows $R_x$ close to $\coo(R_x)$, and has $\pi(A_3)$  close to $\co(\pi(A_3))$.

In \Cref{reductionssamesize}, we construct $A_3\supset A_4\supset A_5$ such that $A_5$ has the same properties as $A_3$ and the arithmetic progressions $\coo(R_x)$ have the same step size $d$.

In \Cref{almostintervals}, we show that $d=1$, i.e. $\coo(R_x)=\co(R_x)$ is an interval for all rows $R_x$ of $A_5$.

In \Cref{reductionAplus}, we show that filling in the rows of $A_5$ to make a set $A_+\supset A_5$ preserves the properties that $A_5$ had (this is the only step where we deviate from throwing away a subset of rows).

In \Cref{sectionAstarapprox}, we show that we can approximate $A_+$ with a subset $A_\star\subset A_+$ which has simultaneously
\begin{enumerate}
    \item Few vertices on $\widetilde{\co}(A_\star)$
    \item $\pi(A_\star)$ close to $\co(\pi(A_\star))$
    \item The technical condition \Cref{pullover}.
\end{enumerate}
Up to this point, we were able to show that $|A\Delta A_+| \le \delta^{O(1)}|A|$. However obtaining $A_\star$ involves a double recursion, and we are only able to show $|A\Delta A_\star|=o(1)|A|$ where $o(1)\to 0$ as $\delta \to 0$.

In \Cref{Astarclosetoco}, we show that $A_\star$ is close to its convex hull. The key step is to convert the problem to one of bounding the size of the epigraph of a certain infimum-convolution of a function by the size of the epigraph of the original function.

Finally, in \Cref{AclosetocoA} we finish the proof of \Cref{equivquali} by showing that $A_\star$ being close to its convex hull implies $A$ is close to its convex hull.

\subsection{Exploiting the inductive hypothesis}\label{inductiveh}
In this section we prove a result, relying on the inductive hypothesis, which we will frequently apply that allows us to remove a small number of rows from a set $A'$ to ensure that the projection $\pi(A')$ is close to $\coo(\pi(A'))$. Recall we introduced in \Cref{cdef} a small constant $c$.
\begin{prop}
\label{alphaprop'}
There exist $\ll$-dependencies such that for constants
$$n_{k,0}^{-1}\ll\sigma\ll\epsilon_0,\alpha,\lambda\le 1\text{ and }\alpha<\lambda,$$
the following holds.

 Let $A'\subset B$ with $|A'|=\epsilon'|B|\ge \frac{\epsilon_0}{2}|B|$ and $d_k(A')\le \sigma^{\lambda}|B|$. Then there exists  $A''\subset A'$ formed as a union of rows $R_x$ of $A'$ with  \begin{align*} |\coo(\pi(A''))\setminus \pi(A'')| \le\sigma^{\alpha} |\pi(B)|,\quad |A'\setminus A''| \le \sigma^{\lambda-\alpha-c}|B|.
\end{align*}
Furthermore, if $A'$ is reduced then $A''$ is reduced (when nonempty) and in particular $\coo(\pi(A''))=\co(\pi(A''))$.
\end{prop}
\begin{proof}[Proof of \Cref{alphaprop'}]
We can take $A''=\emptyset$ if $\lambda \le \alpha+c$, so suppose $\lambda > \alpha+c$. Let
\begin{align*}E_i=\{x\in \pi(A') : |\pi^{-1}(x)\cap A'| \ge i\},\qquad F_i=\{x \in \pi(A'+A') : |\pi^{-1}(x)\cap (A'+A')| \ge i\}.
\end{align*}
Note that $E_1\supset E_2 \supset \ldots$ and $F_1 \supset F_2 \supset \ldots$, and we have 
\begin{align}\label{EiFiequation}
    |A'|=\sum_{i=1}^{n_1} |E_i|,\text{ and }|A'+A'|=\sum_{i=1}^{2n_1-1} |F_i|.
\end{align}
We note that $E_i+E_i \subset F_{2i-1},F_{2i-2}$, so we have (observing $\sigma^{\lambda}\le \sigma$ as $\sigma,\lambda\le 1$)
\begin{align*}|A'+A'| &\ge -2^{k-1}n_1^{-1}|B|+ 2\sum_{i=1}^{n_1} |E_i+E_i|\\&\ge -\sigma^{\lambda}|B|+2\sum_{i=1}^{n_1}|E_i+E_i|.
\end{align*}
Subtracting $2^k|A'|=2\sum_{i=1}^{n_1} 2^{k-1}|E_i|$, we obtain
$d_k(A') \ge -\sigma^{\lambda}|B|+2\sum_{i=1}^{n_1} d_{k-1}(E_i),$
so by the hypothesis $\sigma^{\lambda}|B|\ge d_k(A')$, we see that
\begin{align}\label{sigmalambdagreaterEi}
\sigma^\lambda|B| \ge \sum_{i=1}^{n_1}d_{k-1}(E_i).\end{align}
Let $i_0$ be the first index with $d_{k-1}(E_{i_0})\le \sigma^{\alpha+c/2} |E_{i_0}|$, which exists as otherwise by \eqref{EiFiequation},\eqref{sigmalambdagreaterEi},
$$\sigma^{\lambda}|B|\ge\sigma^{\alpha+c/2}|A'|\ge \sigma^{\alpha+c/2}\frac{\epsilon_0}{2}|B|>\sigma^{\lambda}|B|.$$ Let $A'':=\pi^{-1}(E_{i_0})\cap A'\subset A'$ be the union of all rows of size at least $i_0$. By construction,
\begin{align}
\label{dk-1A''small}
d_{k-1}(\pi(A''))=d_{k-1}(E_{i_0})\le \sigma^{\alpha+c/2}|E_{i_0}|= \sigma^{\alpha+c/2}|\pi(A'')|.
\end{align}
Also as $|E_i|$ is decreasing in $i$, $\sum_{i=1}^{i_0-1} |E_i|\ge \frac{i_0-1}{n_1}|A'|$ by \eqref{EiFiequation}. Thus by \eqref{sigmalambdagreaterEi} and \Cref{negdk}, and the fact that $d_{k-1}(E_i)\ge \sigma^{\alpha+c/2}|E_i|$ by minimality of $i_0$, we have
\begin{align*}
\sigma^{\lambda}|B| \ge\sum_{i=1}^{n_1} d_{k-1}(E_i)\ge& \sigma^{\alpha+c/2}\frac{i_0-1}{n_1}|A'|-n_12^{2(k-1)}n_{k,0}^{-1}n_1^{-1}|B|\\
\ge& \sigma^{\alpha+c/2}\frac{i_0-1}{n_1}\cdot\frac{\epsilon_0}{2}|B|-\sigma^{\lambda}|B|.
\end{align*}
 Thus we obtain $$i_0-1 \le 4\sigma^{\lambda-\alpha-c/2} \epsilon_0^{-1}n_1\le \sigma^{\lambda-\alpha-\frac{2}{3}c}n_1\le \sigma^{c/3}.$$
As the  $|\pi(A'\setminus A'')|\le |\pi(B)|=n_1^{-1}|B|$ nonempty rows in $A'\setminus A''$ have size at most $i_0-1\le \sigma^{c/3}n_1$, we have $$|A'\setminus A''| \le \sigma^{c/3}|B|.$$

We have $|A'\setminus A''|\le 2^{-(k+2)}\epsilon_0|B|$ so $A''$ is reduced by \Cref{dkobs} (applied using $A'$ in place of $A$, assuming $A'$ is reduced), and $|A''|\ge \frac{\epsilon_0}{4}|B|$. In particular, $\pi(A'')$ is reduced (when $A'$ is reduced) and $|\pi(A'')|\ge \frac{\epsilon_0}{4}|\pi(B)|$. The set $\pi(A'')$ has $d_{k-1}(\pi(A''))\le \sigma^{\alpha+c/2}|\pi(A'')|$ by \eqref{dk-1A''small}, and has density at least $\frac{\epsilon_0}{4}$ in $\pi(B)$, which has side lengths at least $n_{k,0}$. By \Cref{hypboxsmall}, the number of parallel hyperplanes needed to cover $\pi(A'')$ is at least $\frac{\epsilon_0}{4}n_{k,0}$. By choosing our $\gg$ dependencies sufficiently strong, we can ensure that we can apply \Cref{maincor} for dimension $k-1$ with $h=\frac{\epsilon_0}{4}n_{k,0}$ and $\delta=\sigma^{\alpha+c/2}$, and deduce that $$|\co(\pi(A''))\setminus \pi(A'')|=|\coo(\pi(A''))\setminus \pi(A'')|\le c_{k-1}\sigma^{\alpha+c/2}|\pi(B)|\le \sigma^{\alpha}|\pi(B)|.$$

\end{proof}

\subsection{Reductions Part 1: All rows are dense in large APs}
\label{prelimreductions}
We start by constructing in a sequence of steps a set $A_3\subset A$ such that $|A\setminus A_3|$ is small, $\pi(A_3)$ is close to  $\co(\pi(A_3))$ and the rows $R_x$ of $A_3$ are large and close to $\coo(R_x)$. In the continuous setting, a similar preliminary reduction was carried out at the beginning of \cite{FigJerJems}.

\subsubsection{$A_1\subset A$ has $\pi(A_1)$ close to its convex progression: Construction }
There exist $\ll$ dependencies $n_{k,0}^{-1}\ll\delta\ll\epsilon_0\le 1$ such that we can apply \Cref{alphaprop'} to $A$ with $\sigma=\delta$, $\alpha=\frac{1}{2}$, $\lambda=1$ and $\epsilon'=\epsilon \ge \frac{\epsilon_0}{2}$ (by \eqref{Ainfo}) to obtain a reduced set $A_1\subset A$ with
\begin{align}\label{A1info}|\co(\pi(A_1))\setminus \pi(A_1)| \le \delta^{\frac{1}{2}} |\pi(B)|,\quad|A\setminus A_1|\le \delta^{\frac{1}{2}-c}|B|.\end{align}
By \Cref{dkobs}, we have \begin{align}\label{A1dk}d_k(A_1)\le \delta |B|+2^k\delta^{\frac{1}{2}-c}|B|\le \delta^{\frac{1}{2}-2c}|B|.\end{align} 
\subsubsection{$A_2$ has large rows close to their convex progressions: Setup }
We show that, assuming $\co(\pi(A'))\setminus \pi(A')$ is small, we can create a subset $A''\subset A'$ by deleting rows with big doubling or small size without changing the size of $A'$ too much.
\begin{prop}
\label{betaprop}
There exist $\ll$-dependencies such that for constants
$$n_{k,0}^{-1}\ll\delta\ll\epsilon_0,\lambda,\alpha,\beta,\gamma\le 1\text{ and } \beta<\alpha<\lambda$$
the following holds.

Let $A'\subset B$ with \begin{align*}d_k(A')
\le \delta^{\lambda} |B|,\qquad
|\co(\pi(A'))\setminus \pi(A')|\le \delta^{\alpha}|\pi(B)|. \end{align*} If $A''\subset A'$ is the union all rows $R_x$ which satisfy $d_1(R_x)\le \delta^{\beta} n_1$ and $|R_x|\geq \delta^{\gamma} n_1$, then $$|A'\setminus A''|\le (\delta^{\alpha-\beta-c}+\delta^{\gamma})|B|.$$
\end{prop}
\begin{proof}
Let $A'''$ be the union all rows $R_x$ of $A'$ which satisfy $d_1(R_x)\le \delta^{\beta} n_1$. For $0\ne\vec{v}\in \{0\}\times\{0,1\}^{k-1}$ and $x\in \{0\}\times \mathbb{Z}^{k-1}$, we have
$$|R_x+R_{x+\vec{v}}|-|R_x|-|R_{x+\vec{v}}|\ge \begin{cases}0&|\{x,x+\vec{v}\}\cap \pi(A')|=0\\-n_1 & |\{x,x+\vec{v}\}\cap \pi(A')|=1\\-1 & |\{x,x+\vec{v}\}\cap \pi(A')|= 2\end{cases}$$ From $|\co(\pi(A'))\setminus \pi(A')|\le\delta^{\alpha}n_1^{-1}|B|$ and \Cref{Surfobs}, we have
\begin{align*}\left|\left\{x\in\{0\}\times \mathbb{Z}^{k-1}: |\{x,x+\vec{v}\}\cap \pi(A')|=1\right\}\right|&\le 2\left|\co(\pi(A'))\setminus \pi(A')\right|+2(k-1)n_1^{-1}n_{k,0}^{-1}|B|\\
&\le \delta^{\alpha-c/4}n_1^{-1}|B|\end{align*}
and
\begin{align*}|\{x\in \{0\}\times\mathbb{Z}^{k-1}: |\{x,x+\vec{v}\}\cap \pi(A')|=2\}|\le |\pi(B)|\le n_{k,0}^{-1}|B|\le \delta^{\alpha-c/4}|B|.\end{align*}

Hence, as $\sum_{x\in\{0\}\times \mathbb{Z}^{k-1}}|R_x|=|A'|$, we have (taking $\vec{v}\in \{0\}\times \{0,1\}^{k-1}$ and $x\in \{0\}\times \mathbb{Z}^{k-1}$)
\begin{align*}
    |A'+A'| \ge& \sum_{\vec{v}}\sum_{x}|R_x+R_{x+\vec{v}}|\\
    =&\left(\sum_{j=0}^2\sum_{0\ne\vec{v}}\sum_{|\{x,x+\vec{v}\}\cap \pi(A')|=j}|R_x+R_{x+\vec{v}}|\right)+\sum_{x\in \pi(A')}|R_x+R_{x}|\\
    \ge& \left(\sum_{0\ne \vec{v}}\sum_{x}|R_{x}|+|R_{x+\vec{v}}|\right)-2(2^{k-1}-1)\delta^{\alpha-c/4}|B|+\sum_{x\in \pi(A')}|R_x+R_{x}|\\
    \geq&(2^{k}-2)|A'|-\delta^{\alpha-c/2}|B|+\sum_{x\in \pi(A')}|R_x+R_x|.
\end{align*}
In particular, as $\sum_{x \in \pi(A')}|R_x|=|A'|$ and $d_1(R_x)\ge -1$ for all $x$, we have
\begin{align*}\delta^{\lambda}|B|\ge d_k(A')\ge& -\delta^{\alpha-c/2}|B|+\sum_{x\in \pi(A')}d_1(R_x)\\\ge&-\delta^{\alpha-c/2}|B|-n_{k,0}^{-1}|B|+ \sum_{x\in \pi(A'\setminus A''')}d_1(R_x)\\
\ge&-\delta^{\alpha-3c/4}|B|+|\pi(A'\setminus A''')|\delta^{\beta} n_1,
\end{align*}
so
$$|A'\setminus A'''|\le n_1|\pi(A'\setminus A''')|\le (\delta^{\alpha-\beta-3c/4}+\delta^{\lambda-\beta})|B|\le (\delta^{\alpha-\beta-3c/4}+\delta^{\alpha-\beta})|B|\le \delta^{\alpha-\beta-c}|B|.$$
Finally note that $A''\subset A'''$ satisfies $|A'''\setminus A''|\leq \delta^{\gamma}n_1|\pi(B)|\leq \delta^{\gamma}|B|$, from which the conclusion follows.
\end{proof}

\subsubsection{$A_2$ has large rows close to their convex progression: Construction }
There exist $\ll$ dependencies $n_{k,0}^{-1}\ll\delta\ll\epsilon_0\le 1$ such that we can apply \Cref{betaprop} $A_1$ with $\lambda=\frac{1}{2}-2c$, $\alpha=\frac{1}{2}-3c$, $\beta=\frac{3}{10}$, and $\gamma=\frac{1}{5}$ (by
\eqref{A1info},\eqref{A1dk}) to obtain a subset $A_2\subset A_1$. Then for all rows $R_x\subset A_2$, we have
\begin{align}\label{A2betagamma}d_1(R_x)\le \delta^{\frac{3}{10}} n_1,\quad|R_x|\ge \delta^{\frac{1}{5}} n_1,\end{align} and by \eqref{A1info} we additionally have \begin{align}\label{AA2}|A\setminus A_2|\le |A\setminus A_1|+|A_1\setminus A_2|\le \left(\delta^{\frac{1}{2}-c}+\delta^{\frac{1}{5}-4c}+\delta^{\frac{1}{5}}\right)|B|\le \delta^{\frac{1}{5}-5c}|B|.\end{align}
By \Cref{dkobs} and \eqref{A1dk}, we have that $A_2$ is reduced and 
\begin{align}\label{A2dk}
    d_k(A_2)\le \left(\delta+2^k\delta^{\frac{1}{5}-5c}\right)|B|\le \delta^{\frac{1}{5}-6c}|B|.
\end{align}
Freiman's $3k-4$ theorem \cite{Freiman}, \Cref{3A-4}, says that for any $R\subset \mathbb{Z}$, we have
$$d_1(R)\ge \min(|R|-3,|\coo(R)\setminus R|-1).$$
Therefore, because $\delta^{\frac{3}{10}} n_1 <\delta^{\frac{1}{5}} n_1-3$, we have by \eqref{A2betagamma} that every row $R_x$ of $A_2$ satisfies
\begin{align}\label{coint}\text{ }|\coo(R_x)\setminus R_x|\le \delta^{\frac{3}{10}}n_1+1\le \delta^{\frac{1}{10}}|R_x|+1\le 2\delta^{\frac{1}{10}}|R_x|.
\end{align}

\begin{rmk}
\label{dxrmk}
In particular, this means that for each non-empty row $R_x$ of $A_2$, we have $|R_x|>\left\lceil \frac{|\coo(R_x)|}{2}\right\rceil$, so there exist two elements $z_1, z_2 \in R_x$  with $z_1-z_2=(d_x, 0, 0 \hdots, 0)$, where $d_x$ is the common difference in the arithmetic progression $\coo(R_x)$.
\end{rmk}

\subsubsection{$A_3\subset A_2$ has $\pi(A_3)$ close to its convex progression: Construction }

There exist $\ll$ dependencies $n_{k,0}^{-1}\ll\delta\ll\epsilon_0\le 1$ such that we can apply \Cref{alphaprop'} to $A_2$ with $\sigma=\delta$, $\alpha=\frac{1}{10}$, $\epsilon'\ge\epsilon-\delta^{\frac{1}{5}-5c}\ge \frac{\epsilon_0}{2}$ and $\lambda=\frac{1}{5}-6c$ (by \eqref{AA2},\eqref{A2dk}) to obtain a reduced set $A_3\subset A_2$ with
\begin{align}\label{A3info}|\co(\pi(A_3))\setminus \pi(A_3)|\le \delta^{\frac{1}{10}}|\pi(B)|, \qquad |A_2\setminus A_3|\le \delta^{\frac{1}{10}-7c}|B|.\end{align}
By \eqref{AA2} and \eqref{A3info}, we have
\begin{align}\label{AA3}|A\setminus A_3|\le |A\setminus A_2|+|A_2\setminus A_3|\le \left(\delta^{\frac{1}{5}-5c}+\delta^{\frac{1}{10}-7c}\right)|B|\le \delta^{\frac{1}{10}-8c}|B|,\end{align}
and by \Cref{dkobs}, we have
\begin{align}\label{A3dk}d_k(A_3)\le \left(\delta+2^k\delta^{\frac{1}{10}-8c}\right)|B|\le \delta^{\frac{1}{10}-9c}|B|.
\end{align}
Finally, as the rows of $A_3$ are a subset of the rows of $A_2$, by \eqref{A2betagamma} and \eqref{coint}, we have
\begin{align}
\label{A3beta}
    |R_x|\ge \delta^{\frac{1}{5}}n_1,\qquad |\coo(R_x)\setminus R_x|\le 2\delta^{\frac{1}{10}}|R_x|
\end{align}
for all rows $R_x$ of $A_3$.
\subsection{Reductions Part 2: All rows are in APs of the same step size}
\label{reductionssamesize}
We now find a set $A_5\subset A_3$ which has the same properties as $A_3$, and furthermore has the property that, for each row $R_x$, the arithmetic progressions $\coo(R_x)$ have the same step sizes. To do this, we carefully analyze a discrete analogue of Voronoi cells.

Let $d_x$ be the smallest consecutive difference between two consecutive elements in $R_x$, which as noted in \Cref{dxrmk} is also the common difference of $\coo(R_x)$, and let $d=\min d_x$.
\subsubsection{$A_4\subset A_3$ has all rows in same step size APs: Setup}
 We now show that the rows with $d_x>d$ carry small weight.

\begin{prop}\label{approp}
There exist $\ll$-dependencies such that for constants
$$n_{k,0}^{-1}\ll\delta\ll\epsilon_0,\alpha,\lambda\le 1\text{ and }\alpha<\lambda$$
the following holds.

Let $A'\subset B$ with $d_k(A')\le \delta^{\lambda}|B|$ and $|\co(\pi(A'))\setminus \pi(A')|\le \delta^{\alpha}|\pi(B)|$. Let $d_x$ be the smallest consecutive difference between two elements of row $R_x\subset A'$, and let $d=\min d_x$. If $A''\subset A'$ is the subset of rows $R_x$ with $d_x=d$, then  $$|A'\setminus A''|\le \delta^{\alpha-c}|B|.$$
\end{prop}

\begin{proof}[Proof of \Cref{approp}] We start by first proving some claims. \Cref{dxdy} shows that $|R_x+R_y|$ is large if $d_x \ne d_y$, and \Cref{xfnotyf} creates a large set of disjoint row sums of this form. \Cref{diagonal} is used to prove \Cref{diagonal2}, which shows that this set of disjoint row sums has small intersection with $A'(+)A'$. Finally, \Cref{rs} shows $A'(+)A'$ is large, and we can carry out the proof outline described in \Cref{outlineofpaper}.

\begin{clm}
\label{dxdy}
Let $X,Y \subset \mathbb{Z}$ with $|X| \ge 2$, $Y\ne\emptyset$, such that the smallest differences $d_X,d_Y$ between consecutive elements of $X$ and $Y$ respectively satisfy $d_X<d_Y$. Then $$|X+Y| \ge |X|+2|Y|-2.$$
\end{clm}
\begin{proof}
Consider  elements $x,x' \in X$ such that $x'-x=d_X$. Let $X_{<x}$ be the set of elements less than $x$ in $X$ and analogously $X_{>x'}$ those elements greater than $x'$. Then the following four sets
$$X_{<x}+\min(Y), (Y+x),  (Y+x'),X_{> x'}+\max(Y),$$ are disjoint subsets of $X+Y$.
\end{proof}
Now, we define $$f:\pi(A')\setminus \pi(A'')\to \pi(A'')$$ by letting $f(x)\in \pi(A'')$ be a closest point to $x$ in Euclidean distance (breaking ties arbitrarily). Fibers of $f$ should be thought of as a discrete analogue of Voronoi cells associated to $\pi(A'')$.

\begin{clm}\label{xfnotyf}
We have  $x+f(x)\ne y+f(y)\text{ for distinct }x,y\in \pi(A')\setminus \pi(A'')$.
In particular,
$$Z_1:=\bigsqcup_{x_1\in \pi(A')\setminus \pi(A'')}R_{x_1}+R_{f(x_1)}\subset A'+A'$$
is a disjoint union.
\end{clm}

\begin{proof}
Indeed, otherwise $x,y,f(x),f(y)$ form a parallelogram with distinct vertices with diagonals $xf(x)$ and $yf(y)$. However, in any parallelogram (even degenerate as long as the vertices are distinct), the longest diagonal is longer than all sides. Hence, if say $xf(x)$ is the longest diagonal, then $|x-f(x)|>|x-f(y)|$, a contradiction.
\end{proof}

Let
$$Z:=A'(+)A'=\bigsqcup_{\vec{v}\in \{0\} \times \{0,1\}^{k-1}}\bigsqcup_{x_2\in \{0\} \times \mathbb{Z}^{k-1}}R_{x_2}(+)R_{x_2+\vec{v}}\subset A'+A'.$$
We now analyze when $R_{x_1}+R_{f(x_1)}$ and $R_{x_2}(+)R_{x_2+\vec{v}}$ can intersect. \begin{clm}\label{diagonal} If $x_1\in \pi(A')\setminus \pi(A'')$, $x_2\in\{0\}\times \mathbb{Z}^{k-1}$, $\vec{v}\in\{0\}\times \{0,1\}^{k-1}$ are such that $(R_{x_1}+R_{f(x_1)}) \cap (R_{x_2}(+)R_{x_2+\vec{v}}) \not= \emptyset$,  then either  $\{x_1,f(x_1)\}=\{x_2,x_2+\vec{v}\}$ or $x_2,x_2+\vec{v}\in \pi(A') \setminus \pi(A'')$.
\end{clm}
\begin{proof}
The points $x_1, x_2, f(x_1), x_2+\vec{v}$ form a (possibly degenerate) parallelogram with diagonals $x_1f(x_1)$ and $x_2(x_2+\vec{v})$. Assuming that $\{x_1,f(x_1)\}\not=\{x_2,x_2+\vec{v}\}$, this parallelogram has distinct diagonals.

The number of odd coordinates of $x_1-f(x_1)$ is the same as the number of odd coordinates of $x_1+f(x_1)=2x_2+\vec{v}$, which is the same as the number of non-zero coordinates of $v$. Hence,
$|x_1-f(x_1)|\ge |\vec{v}|$, or equivalently $|x_1-f(x_1)|\ge |x_2-(x_2+\vec{v})|$. Therefore, $x_1f(x_1)$ is the longest diagonal of the above parallelogram. As in a parallelogram (even degenerate as long as the diagonals do not coincide) the largest diagonal is strictly longer than all sides, we deduce that the diagonal $x_1f(x_1)$ is strictly longer than $x_1x_2$ and $x_1(x_2+\vec{v})$. By definition of $f(x_1)$, this implies $x_2,x_2+\vec{v}\not\in  \pi(A'')$. As $R_{x_2},R_{x_2+\vec{v}}$ are nonempty, we also have $x_2,x_2+\vec{v}\in \pi(A')$ and the result follows.
\end{proof}

\begin{clm}
\label{diagonal2}
 For any $x_1\in \pi(A') \setminus \pi(A'')$, we have $$|(R_{x_1}+R_{f(x_1)})\setminus Z|\ge |R_{x_1}|-1.$$
\end{clm}

\begin{proof}
We have $|(R_{x_1}+R_{f(x_1)})\setminus Z|=|(R_{x_1}+R_{f(x_1)})\setminus (R_{x_2}(+)R_{x_2+\vec{v}})|$ for the unique $x_2\in \{0\}\times \mathbb{Z}^{k-1}$ and $\vec{v}\in \{0\}\times \{0,1\}^{k-1}$ such that $x_1+f(x_1)=x_2+(x_2+\vec{v})$. Clearly $|(R_{x_1}+R_{f(x_1)})|\ge |R_{x_1}|-1$, so assume $(R_{x_1}+R_{f(x_1)}) \cap (R_{x_2}(+)R_{x_2+\vec{v}}) \not= \emptyset$. By \Cref{diagonal} we have that $\text{either  }\{x_1,f(x_1)\}=\{x_2,x_2+\vec{v}\} \text{ or }x_2,x_2+\vec{v}\in \pi(A') \setminus \pi(A'')$. In the former case, by \Cref{dxdy} we have that $$|(R_{x_1}+R_{f(x_1)}) \setminus (R_{x_2}(+)R_{x_2+\vec{v}})|=|R_{x_1}+R_{f(x_1)}|-|R_{x_1}|-|R_{f(x_1)}|+1\ge |R_{x_1}|-1.$$
Assume now we are in the latter case. Let $z_1,z_2\in R_{f(x_1)}$ with $z_1-z_2=(d,0,\ldots,0)$. As the smallest difference in $R_{x_2}(+)R_{x_2+\vec{v}}$ is strictly larger than $d$, for every element $z\in R_{x_1}$, either $z+z_1$ or $z+z_2$ is not in $R_{x_2}(+)R_{x_2+\vec{v}}$, and if there were $z,z'\in R_{x_1}$ with $z+z_1=z'+z_2$, then $z'-z=(d,0,\dots,0)$, contradicting $x_1\not\in \pi(A'')$. Hence
$$|(R_{x_1}+R_{f(x_1)}) \setminus (R_{x_2}(+)R_{x_2+\vec{v}})|\ge |R_{x_1}|>|R_{x_1}|-1.$$
\end{proof}

\begin{clm}\label{rs}
We have $|Z| \ge 2^k|A'| - 2^k\delta^{\alpha-c/2}|B|$.
\end{clm}

\begin{proof}
Note that for $x_2\in \{0\}\times \mathbb{Z}^{k-1}$, $0\ne\vec{v} \in \{0\}\times \{0,1\}^{k-1}$, we have
$$|R_{x_2}(+)R_{{x_2}+\vec{v}}|-|R_{x_2}|-|R_{x_2+\vec{v}}|\ge\begin{cases}0 & |\{x_2,x_2+\vec{v}\}\cap \pi(A')|=0\\ -n_1 & |\{x_2,x_2+\vec{v}\}\cap \pi(A')|=1\\ -1& |\{x_2,x_2+\vec{v}\}\cap \pi(A')|=2.\end{cases}$$
From $|\co(\pi(A'))\setminus \pi(A')|\le\delta^{\alpha}n_1^{-1}|B|$ and \Cref{Surfobs}, we have
\begin{align*}|\{x\in \{0\}\times\mathbb{Z}^{k-1}: |\{x,x+\vec{v}\}\cap \pi(A')|=1\}|&\le 2|\co(\pi(A'))\setminus \pi(A')|+2(k-1)n_1^{-1}n_{k,0}^{-1}|B|\\
&\le \delta^{\alpha-c/4}n_1^{-1}|B|\end{align*}
and
\begin{align*}|\{x\in \{0\}\times\mathbb{Z}^{k-1}: |\{x,x+\vec{v}\}\cap \pi(A')|=2\}|\le |\pi(B)|\le n_{k,0}^{-1}|B|\le \delta^{\alpha-c/4}|B|.\end{align*}

As $\sum_{x\in\{0\}\times \mathbb{Z}^{k-1}}|R_x|=|A'|$, we have (taking $\vec{v}\in \{0\}\times \{0,1\}^{k-1}$ and $x\in \{0\}\times \mathbb{Z}^{k-1}$)
\begin{align*}
    |Z| =& \sum_{\vec{v}}\sum_{x}|R_x(+)R_{x+\vec{v}}|\\
    =&\left(\sum_{j=0}^2\sum_{0\ne\vec{v}}\sum_{|\{x,x+\vec{v}\}\cap \pi(A')|=j}|R_x(+)R_{x+\vec{v}}|\right)+\sum_{x\in \pi(A')}|R_x(+)R_x|\\
    \ge& -2(2^{k-1}-1)\delta^{\alpha-c/4}|B|-2^{k-1}n_{k,0}^{-1}|B|+\sum_{\vec{v} }\sum_{x}|R_{x}|+|R_{x+\vec{v}}|\\
    \geq&2^{k}|A'|-2^k\delta^{\alpha-c/2}|B|.
\end{align*}

\end{proof}

Returning to the proof of \Cref{approp}, note that $$|\pi(A')\setminus \pi(A'')|\le |\pi(B)|\le n_{k,0}^{-1}|B|\le \delta^{\alpha-c/2}|B|.$$
Thus, by \Cref{xfnotyf}, \Cref{diagonal2}, and \Cref{rs}, we have

\begin{align*}|A'+A'|&\ge |Z\cup Z_1|\\
&=|Z|+\sum_{x_1\in \pi(A')\setminus \pi(A'')}|(R_{x_1}+R_{f(x_1)})\setminus Z|\\
&\ge2^k|A'| -2^k\delta^{\alpha-c/2}|B| +\sum_{x_1\in \pi(A')\setminus \pi(A'')}(|R_{x_1}|-1)\\
&\ge2^k|A'|-\delta^{\alpha-3c/4}|B|+|A'\setminus A''|.
\end{align*}

We conclude that
$$|A'\setminus A''| \le d_k(A')+\delta^{\alpha-3c/4}|B|\le \delta^{\lambda} |B|+\delta^{\alpha-3c/4}|B|\le \delta^{\alpha-c}|B|.$$

\end{proof}
\subsubsection{$A_4\subset A_3$ has all rows in same step size APs: Construction}
There exist $\ll$ dependencies $n_{k,0}^{-1}\ll\delta\ll\epsilon_0\le 1$ such that we can apply \Cref{approp} to $A_3$ with $\lambda=\frac{1}{10}-9c$ $\alpha=\frac{1}{10}-10c$ (by \eqref{A3info},\eqref{A3dk}), obtaining $A_4\subset A_3$. Then we have
\begin{align}\label{A4info}
    |A_3\setminus A_4|\le \delta^{\frac{1}{10}-11c}|B|,
\end{align}
and thus by \eqref{AA3} and \eqref{A4info},
\begin{align}\label{AA4}
    |A\setminus A_4|\le |A\setminus A_3|+|A_3\setminus A_4|\le (\delta^{\frac{1}{10}-8c}+\delta^{\frac{1}{10}-11c})|B|\le \delta^{\frac{1}{10}-12c}|B|.
\end{align}
By \Cref{dkobs}, $A_4$ is reduced and we have
\begin{align}\label{A4dk}
    d_k(A_4)\le (\delta+2^k\delta^{\frac{1}{10}-12c})|B|\le \delta^{\frac{1}{10}-13c}|B|.
\end{align}
By construction and \Cref{dxrmk}, $\coo(R_x)$ has the same step size $d$ for all rows $R_x$ of $A_4$. Finally, by \eqref{A3beta}, we have
\begin{align}\label{A4beta}
    |R_x|\ge \delta^{\frac{1}{5}}n_1,\qquad |\coo(R_x)\setminus R_x|\le 2\delta^{\frac{1}{10}}|R_x|
\end{align}
for all rows $R_x$ of $A_4$.

\subsubsection{$A_5\subset A_4$ with $\pi(A_5)$ close to its convex progression: Construction}
There exist $\ll$ dependencies $n_{k,0}^{-1}\ll\delta\ll\epsilon_0\le 1$ such that we can apply \Cref{alphaprop'} to $A_4$ with $\sigma=\delta$, $\lambda=\frac{1}{10}-13c$, $\alpha=\frac{1}{20}$ and $\epsilon'\ge \epsilon-\delta^{\frac{1}{10}-12c}\ge \frac{\epsilon_0}{2}$ (by \eqref{A4dk},\eqref{AA4}) to obtain a reduced set $A_5\subset A_4$ with
\begin{align}\label{A5daggerinfo}
 |\co(\pi(A_5))\setminus \pi(A_5)|\le \delta^{\frac{1}{20}}|\pi(B)|\qquad
    |A_4\setminus A_5|\le \delta^{\frac{1}{20}-14c}|B|.
\end{align}
By \eqref{AA4} and \eqref{A5daggerinfo}, we have
\begin{align}\label{AA5dagger}
    |A\setminus A_5|\le |A\setminus A_4|+|A_4\setminus A_5|\le (\delta^{\frac{1}{10}-12c}+\delta^{\frac{1}{20}-14c})|B|\le \delta^{\frac{1}{20}-15c}|B|,
\end{align}
and by \Cref{dkobs}, we have 
\begin{align}\label{A5dk}
    d_k(A_5)\le (\delta+2^k\delta^{\frac{1}{20}-15c})|B|\le \delta^{\frac{1}{20}-16c}|B|.
\end{align}
Furthermore, all rows of $A_5$ are also rows of $A_4$, so have the same step size $d$ and satisfy \eqref{A4beta}, so
\begin{align}
\label{A5beta}
       |R_x|\ge \delta^{\frac{1}{5}}n_1,\qquad |\coo(R_x)\setminus R_x|\le 2\delta^{\frac{1}{10}}|R_x|
\end{align}
for all rows $R_x$ of $A_5$.

\subsection{Reductions Part 3: Showing the rows of $A_5$ are almost intervals}
\label{almostintervals}
We now show that the arithmetic progressions $\coo(R_x)$ containing the rows $R_x$ of $A_5$ are in fact intervals i.e. $\coo(R_x)=\co(R_x)$.

We suppose by way of contradiction that for all rows $R_x$ of $A_5$, the arithmetic progression $\coo(R_x)$ has the same step size $d\ne 1$.
\begin{defn}Let $\pi':\mathbb{Z}^k\to \mathbb{Z}$ be the projection onto the second coordinate. For a set $A'\subset \mathbb{Z}^k$, we let a ``hyperplane'' $H_y$ be $\pi'^{-1}(y)\cap A'$.
\end{defn}

We shall make a series of temporary reductions in order to arrive at a contradiction, and we shall notate sets used in this proof by contradiction with the dagger symbol $\dagger$.

\begin{rmk}
The hyperplanes $H_y$ of a set $A'$ are unions of rows $R_x$. 
\end{rmk}

\subsubsection{$A_6^{\dagger}\subset A_5$ has big hyperplanes with small doubling: Setup}
First, we show that, assuming that $\co(\pi(A'))\setminus \pi(A')$ is small, we can create a subset $A''\subset A'$ by deleting hyperplanes with big doubling or small size without changing the size of $A'$ too much. This is analogous to \Cref{betaprop} for rows.

\begin{prop}\label{hypbetaprop}
There exist $\ll$-dependencies such that for constants
$$n_{k,0}^{-1}\ll\delta\ll\epsilon_0,\lambda,\alpha,\beta,\gamma\le 1\text{ and } \beta<\alpha<\lambda$$
such that the following holds.

Let $A'\subset B$ with $$d_k(A')\le \delta^{\lambda}|B|,\qquad |\co(\pi(A'))\setminus \pi(A')|\le \delta^{\alpha}|\pi(B)|.$$ If $A''\subset A'$ is the union of all hyperplanes $H_y$ with $d_{k-1}(H_y)\le \delta^{\beta} n_2^{-1}|B|$ and $|H_y|\ge \delta^{\gamma}n_2^{-1}|B|$, then
$$|A'\setminus A''|\le (\delta^{\alpha-\beta-c}+\delta^{\gamma})|B|.$$
\end{prop}
\begin{proof}
Let $A'''\subset A'$ be the union of the hyperplanes $H_y$ with $d_{k-1}(H_y)\le \delta^{\beta}n_2^{-1}|B|$. For $0\ne \vec{v}\in \{0\}\times\{1\}\times \{0,1\}^{k-2}$, we have
$$|R_x+R_{x+\vec{v}}|-|R_x|-|R_{x+\vec{v}}|\ge \begin{cases}0&|\{x,x+\vec{v}\}\cap \pi(A')|=0\\-n_1 & |\{x,x+\vec{v}\}\cap \pi(A')|=1\\-1 & |\{x,x+\vec{v}\}\cap \pi(A')|= 2\end{cases}.$$
From $|\co(\pi(A'))\setminus \pi(A')|\le \delta^{\alpha}n_1^{-1}|B|$ and \Cref{Surfobs}, we have
\begin{align*}|\{x\in\{0\}\times \mathbb{Z}^{k-1}: |\{x,x+\vec{v}\}\cap \pi(A')|=1\}|&\le 2|\co(\pi(A'))\setminus \pi(A')|+2(k-1)n_1^{-1}n_{k,0}^{-1}|B|\\
&\le \delta^{\alpha-c/4}n_1^{-1}|B|\end{align*}
and
\begin{align*}|\{x\in\{0\}\times \mathbb{Z}^{k-1}: |\{x,x+\vec{v}\}\cap \pi(A')|=2\}|\le |\pi(B)|\le n_{k,0}^{-1}|B|\le \delta^{\alpha-c/4}|B|.\end{align*}
Hence as $\sum_{x \in \{0\}\times \mathbb{Z}^{k-1}}|R_x|=|A'|$, and $\sum_{y\in \mathbb{Z}} |H_y|=|A'|$, we have (taking $\vec{v}\in \{0\}\times \{1\}\times \{0,1\}^{k-2}$ and $x\in \{0\}\times \mathbb{Z}^{k-1}$)
\begin{align*}|A'+A'|\ge& \left(\sum_{y\in \mathbb{Z}} |H_y+H_y|\right)+\sum_{j=0}^2\sum_{\vec{v}}\sum_{|\{x,x+\vec{v}\}\cap \pi(A')|=j}|R_x+R_{x+\vec{v}}|\\
\ge& \left(\sum_{y \in \mathbb{Z}}|H_y+H_y|\right)+\left(\sum_{\vec{v}}\sum_{x}|R_x|+|R_{x+\vec{v}}|\right) -2\cdot 2^{k-2}\delta^{\alpha-c/4}|B|\\
\ge&\left(\sum_{y\in \mathbb{Z}}|H_y+H_y|\right)+2^{k-1}|A'|-\delta^{\alpha-c/2}|B|\\=&\sum_{y\in \pi(A')}d_{k-1}(H_y)+2^k|A'|-\delta^{\alpha-c/2}\end{align*}

Now, as $\sum_{y \in \mathbb{Z}}|H_y|=|A'|$, by \Cref{negdk} with $H_y$ and the box $\pi'^{-1}(y)\cap B$, we have
\begin{align*}
    \delta^{\lambda}|B|\ge d_k(A')\ge& \left(\sum_{y\in \pi'(A')}d_{k-1}(H_y)\right)-\delta^{\alpha-c/2}|B|\\
    \ge& \left(\sum_{y\in \pi'(A')\setminus \pi'(A''')}d_{k-1}(H_y)\right)-n_2\cdot 2^{2(k-1)}n_{k,0}^{-1}n_2^{-1}|B|-\delta^{\alpha-c/2}|B|\\
    \ge& |\pi'(A'\setminus A''')|\cdot \delta^{\beta} n_2^{-1}|B|-\delta^{\alpha-3c/4}|B|\\
    \ge&\delta^{\beta} |A'\setminus A'''|-\delta^{\alpha-3c/4}|B|.
\end{align*}
Therefore,
$$|A'\setminus A'''| \le (\delta^{\lambda-\beta}+\delta^{\alpha-\beta-3c/4})|B|\le \delta^{\alpha-\beta-c}|B|.$$
Finally, note that $A''\subset A'''$ satisfies $|A'''\setminus A''|\le n_2\delta^{\gamma}n_2^{-1}|B|=\delta^{\gamma}|B|$, from which the conclusion follows.
\end{proof}
\subsubsection{$A_6^{\dagger}\subset A_5$ has big hyperplanes with small doubling: Construction }
There exist $\ll$ dependencies $n_{k,0}^{-1}\ll\delta\ll\epsilon_0\le 1$ such that we can apply  \Cref{hypbetaprop} to $A_5$ with $\lambda=\frac{1}{20}-16c$, $\alpha=\frac{1}{40}$, $\beta=\frac{1}{80}$, and $\gamma=\frac{1}{160}-c$, obtaining $A_6^{\dagger}\subset A_5$ satisfying 
\begin{align}\label{A6daggerinfo}|A_5\setminus A_6^{\dagger}|\le (\delta^{\frac{1}{80}-c}+\delta^{\frac{1}{160}-c})|B|\le \delta^{\frac{1}{160}-2c} |B|
\end{align}
and for every hyperplane $H_y\subset A_6^{\dagger}$ we have
\begin{align}\label{hypdoub}d_{k-1}(H_y)\le \delta^{\frac{1}{80}}n_2^{-1}|B|, \qquad |H_y|\ge \delta^{\frac{1}{160}-c} n_2^{-1}|B|.\end{align}
We also have by \eqref{AA5dagger} and \eqref{A6daggerinfo} that
\begin{align}\label{AA6dagger}
    |A\setminus A_6^{\dagger}|\le |A\setminus A_5|+|A_5\setminus A_6^{\dagger}|\le (\delta^{\frac{1}{20}-15c}+\delta^{\frac{1}{160}-2c})|B|\le \delta^{\frac{1}{160}-3c}|B|,
\end{align}
and by \Cref{dkobs} we have $A_6^{\dagger}$ is reduced and
\begin{align}\label{A6dk}
    d_k(A_6^{\dagger})\le (\delta+2^k\delta^{\frac{1}{160}-3c})|B|\le \delta^{\frac{1}{160}-4c}|B|.
\end{align}
Consider the set $H_y$ contained inside a box $B_y:=\pi'^{-1}(y)\cap B$ with sides at least $n_{k,0}$. We have $|H_y| \ge \delta^{\frac{1}{160}-c} |B_y|$ and $d_{k-1}(H_y) \le \delta^{\frac{1}{160}+c}|H_y|$.  By \Cref{hypboxsmall}, the number of parallel hyperplanes needed to cover $H_y$ is at least $n_{k,0}\delta^{\frac{1}{160}-c}\gg \delta^{\frac{1}{160}+c}$. The $\ll$ dependencies can be chosen strong enough to imply those needed for \Cref{maincor} for dimension $k-1$, and we deduce
\begin{align}\label{convhyp}|\coo(H_y)\setminus H_y|\le c_{k-1}\delta^{\frac{1}{160}+c}|B_y|\le\delta^{\frac{1}{160}}|B_y|.\end{align}

\begin{obs}\label{stepsize}
For a hyperplane $H_y\subset A_6^{\dagger}$, the smallest affine sublattice $\Lambda_{H_y}\subset \pi'^{-1}(y)=\mathbb{Z}\times \{y\}\times \mathbb{Z}^{k-2}$ containing $H_y$ has the property that the nonempty rows of $\Lambda_{H_y}$ have step size $d$.
\end{obs}
\begin{proof}
For each row $R_x$ contained in a hyperplane $H_y$, the arithmetic progression $\coo(R_x)$ has step size $d$. Let $d'$ be the uniform step size of the nonempty rows of $\Lambda_{H_y}$ (which exists by Lagrange's theorem), and hence of the nonempty rows of $\coo(H_y)$. Assume for the sake of contradiction $d'\not=d$. As $d'$ divides $d$, for every row $R_x$ of $H_y$ and corresponding row $R_x'$ of $\coo(H_y)$, we have 
$|R_x'\setminus R_x|\ge |R_x|-1 \ge \frac{1}{2}|R_x|$ (as each row $R_x$ has at least $2$ elements by \eqref{A5beta}). Adding this over all rows $R_x$ of $H_y$, we obtain from \eqref{hypdoub}
$$|\coo(H_y)\setminus H_y| \ge \frac{1}{2}|H_y|\ge \frac{1}{2}\delta^{\frac{1}{160}-c} |B_y|$$ 
contradicting \eqref{convhyp} that
$\delta^{\frac{1}{160}}|B_y|\ge |\coo(H_y)\setminus H_y|.$
\end{proof}
\subsubsection{$A_7^{\dagger}\subset A_6^{\dagger}$ with $\pi(A_7^{\dagger})$ close to its convex progression: Construction}
There exist $\ll$ dependencies $n_{k,0}^{-1}\ll\delta\ll\epsilon_0\le 1$ such that we can apply \Cref{alphaprop'} to  $A_{6}^{\dagger}$ with $\sigma=\delta$, $\lambda=\frac{1}{160}-4c$, $\alpha=\frac{1}{320}$, $\epsilon'\ge\epsilon-\delta^{\frac{1}{160}-3c}\ge \frac{\epsilon_0}{2}$ (by \eqref{AA6dagger},\eqref{A6dk}) to obtain a reduced set $A_7^{\dagger}\subset A_6^{\dagger}$ with
\begin{align}
\label{A7daggerinfo}
 |A_6^{\dagger}\setminus A_7^{\dagger}|\le \delta^{\frac{1}{320}-5c}|B|,\qquad |\co(\pi(A_7^{\dagger}))\setminus \pi(A_7^{\dagger})|\le \delta^{\frac{1}{320}}|\pi(B)|.
\end{align}
By \eqref{AA6dagger} and \eqref{A7daggerinfo}, we have
\begin{align}\label{AA7dagger}
    |A\setminus A_7^{\dagger}|\le |A\setminus A_6^{\dagger}|+|A_6^{\dagger}\setminus A_7^{\dagger}| \le (\delta^{\frac{1}{160}-3c}+\delta^{\frac{1}{320}-5c})|B|\le \delta^{\frac{1}{320}-6c}|B|,
\end{align}
and by \Cref{dkobs} we have 
\begin{align}\label{A7daggerdk}
    d_k(A_7^{\dagger})\le (\delta+2^k\delta^{\frac{1}{320}-6c})|B|\le \delta^{\frac{1}{320}-7c}|B|.
\end{align}
As we pass from $A_6^{\dagger}$ to $A_7^{\dagger}$, the affine sub-lattice $\Lambda_{H_y}$ shrinks when we consider the $H_y$ now with respect to $A_7^{\dagger}$, so by \Cref{stepsize} the nonempty rows of $\Lambda_{H_y}$ have step size at least $d$. As the nonempty rows of $A_7^{\dagger}$ have step size $d$, this forces the nonempty rows of $\Lambda_{H_y}$ to have step size exactly $d$.  

We note that we do not know that the hyperplanes of $A_7^{\dagger}$ have big size or small doubling.
\subsubsection{$A_8^{\dagger}\subset A_7^{\dagger}$ has $\pi(H)$ reduced for all hyperplanes $H$: Construction }
There exist $\ll$ dependencies $n_{k,0}^{-1}\ll\delta\ll\epsilon_0\le 1$ such that the following holds.
Let $A_8^{\dagger}\subset A_7^{\dagger}$ be the union of all hyperplanes $H_y$ of $A_7^{\dagger}$ such that $\pi(H_y)$ is reduced in $\{0\}\times \{y\}\times \mathbb{Z}^{k-2}$. Recall we let $B_y=\pi'^{-1}(y)\cap B$, so $H_y\subset B_y$.


If $\pi(H_y)$ is not reduced inside $\{0\}\times\{y\}\times \mathbb{Z}^{k-2}$, then there is a direction $e_j$ with $j\in \{3,\hdots, k\}$ such that $\pi(H_y)\cap(\pi(H_y)+e_j)=\emptyset$. Hence, letting $\pi_j:\mathbb{Z}^k\to \mathbb{Z}^{k-1}$ be the projection away from the $j$th coordinate and $S_x=\pi_j^{-1}(x)\cap \pi(H_y)$ for $x\in \{0\}\times \{y\}\times \mathbb{Z}^{k-2}$, we have $|\co(S_x)\setminus S_x| \ge |S_x|-1$.
Summing the above inequality over all $x\in \pi_j(\pi(H_y))\subset \pi_j(\pi(B_y))$, we deduce
$$|\co(\pi(H_y))\setminus \pi(H_y)|\ge \sum_{x\in \pi_j(\pi(H_y))}|\co(S_x)\setminus S_x| \ge |\pi(H_y)|-n_{k,0}^{-1}|\pi(B_y)|.$$
Adding this over all $y$ with $\pi(H_y)$ not reduced, we obtain by \eqref{A7daggerinfo} that
\begin{align}\nonumber|\pi(A_7^{\dagger}\setminus A_8^{\dagger})|&=\sum_{\pi(H_y)\text{ not reduced}}|\pi(H_y)|\\\nonumber&\le n_{k,0}^{-1}|\pi(B)|+\sum_{\pi(H_y)\text{ not reduced}}|\co(\pi(H_y))\setminus \pi(H_y)|\\\label{piA7A8}&\le n_{k,0}^{-1}|\pi(B)|+|\co(\pi(A_7^{\dagger}))\setminus \pi(A_7^{\dagger})|\le 2\delta^{\frac{1}{320}}|\pi(B)|\le \delta^{\frac{1}{320}-c}|\pi(B)|,
\end{align}
so in particular we have
$$|A_7^{\dagger}\setminus A_8^{\dagger}|\le \delta^{\frac{1}{320}-c}|B|.$$
Hence by \eqref{AA7dagger} we have
\begin{align}\label{AA8dagger}
    |A\setminus A_8^{\dagger}|\le |A\setminus A_7^{\dagger}|+|A_7^{\dagger}\setminus A_8^{\dagger}|\le  \left(\delta^{\frac{1}{320}-6c}+\delta^{\frac{1}{320}-c}\right)|B|\le \delta^{\frac{1}{320}-7c}|B|,
\end{align}
and \Cref{dkobs} shows  $A_8^{\dagger}$ is reduced and
\begin{align}\label{A8dk}d_k(A_8^{\dagger})\le \left(\delta+2^k\delta^{\frac{1}{320}-7c}\right)|B|\le \delta^{\frac{1}{320}-8c}|B|.\end{align}
Note that $|\co(\pi(A_8^{\dagger}))\setminus \pi(A_8^{\dagger})|\le |\co(\pi(A_7^{\dagger}))\setminus \pi(A_7^{\dagger})|+|\pi(A_7^{\dagger}\setminus A_8^{\dagger})|$, so we have by \eqref{A7daggerinfo},\eqref{piA7A8} that
\begin{align}\label{copiA8dagger}|\co(\pi(A_8^{\dagger}))\setminus \pi(A_8^{\dagger})|\le \left(\delta^{\frac{1}{320}}+\delta^{\frac{1}{320}-c}\right)|\pi(B)|\le \delta^{\frac{1}{320}-2c}|\pi(B)|.\end{align}
As we pass from $A_7^{\dagger}$ to $A_8^{\dagger}$, the affine sub-lattice $\Lambda_{H_y}$ shrinks when we consider the $H_y$ now with respect to $A_8^{\dagger}$, so the nonempty rows of $\Lambda_{H_y}$ have step size at least $d$. As the nonempty rows of $A_8^{\dagger}$ have step size $d$, this forces the nonempty rows of $\Lambda_{H_y}$ to have step size exactly $d$. Furthermore, the reducedness of $\pi(H_y)$ implies $\pi(\Lambda_{H_y})=\{0\}\times \{y\}\times\mathbb{Z}^{k-2}$.

\subsubsection{Contradiction}
There exist $\ll$ dependencies $n_{k,0}^{-1}\ll\delta\ll\epsilon_0\le 1$ such that we are now able to derive a contradiction (we will only use these dependencies after \Cref{clm319}). Let $H_{y_1},\ldots,H_{y_\ell}$ be the nonempty hyperplanes of $A_8^{\dagger}$ with $y_1<\ldots<y_\ell$, and for notational convenience, set $H_{i}:=H_{y_i}$ and $\Lambda_i:=\Lambda_{H_{y_i}}$. As just noted, we have $\pi(\Lambda_i)=\{0\}\times \{y_i\}\times \mathbb{Z}^{k-2}$ for all $i$, and as the nonempty rows of $\Lambda_i$ have step size $d$ the rows of $\Lambda_i$ are all of the form $z'+d\mathbb{Z}$. Let $$\Phi_i:\pi(\Lambda_{i})=\{0\}\times\{y_i\}\times \mathbb{Z}^{k-2}\to \mathbb{Z}/d\mathbb{Z}$$ be the affine-linear function\footnote{meaning there exists $a\in \{0\}\times \{0\}\times(\mathbb{Z}/d\mathbb{Z})^{k-2}$ and $b\in \mathbb{Z}/d\mathbb{Z}$ such that $\Phi_i(w)=w\cdot a+b$} defined by taking $\Phi_i(0,y_i,z)=z'\mod d$ where $(z'+d\mathbb{Z},y_i,z)$ is a row in $\Lambda_{i}$.

We create $r$ subintervals $I_i\subset \{1,\ldots,\ell\}$ satisfying the following properties.
\begin{itemize}
\item $1\in I_1$.
\item $\bigsqcup_{j\in I_i}H_j$ is not reduced in $\mathbb{Z}^k$, and $H_{1+\max I_i}\sqcup\bigsqcup_{j \in I_i}H_j$ is reduced, for $1\leq i\le r-1$.
\item $\min I_{i+1}=\max I_i+\begin{cases} 0 & H_{\max I_i}\sqcup H_{1+\max I_{i}}\text{ is not reduced.}\\1 & H_{\max I_i}\sqcup H_{1+\max I_i}\text{ is reduced.}
    \end{cases}$
\end{itemize}
These conditions uniquely determine intervals $I_1,\ldots,I_r$ which cover $\{1,\ldots,\ell\}$, and as $A_8^{\dagger}$ is reduced we have $r\ge 2$. 

\begin{rmk}\label{minmaxrmk} If $\max I_i=\min I_{i+1}$, then $|I_i|\ge 2$. If instead $\max I_i+1=\min I_{i+1}$, then with $j=\max I_i$ we have $y_{j}+1=y_{j+1}$, and $\Phi_j(w)-\Phi_{j+1}(w+e_2): \{0\}\times \{y_{j}\}\times \mathbb{Z}^{k-2}\to \mathbb{Z}$ is non-constant.
\end{rmk}

For $1\le i \le r-1$, let $z_i\in H_{1+\max I_i}$ be a point not in the affine sub-lattice containing $\bigsqcup_{j \in I_i}  H_{j}$. Let $$f:\bigsqcup_{j\in [1,\min I_r-1]}H_j\to A_8^{\dagger}$$ be defined by setting
$$f\left(\bigsqcup_{j \in [\min I_i,\min I_{i+1}-1]}H_j\right)=\{z_i\}.$$

\begin{clm} If $z',z''\in \bigsqcup_{j\in[1,\min I_r-1]}H_j$ are distinct, then we have
$$z'+f(z')\ne z''+f(z'').$$
\end{clm}
\begin{proof}Indeed, if they were equal then
$$\pi'(z')+\pi'(f(z'))=\pi'(z'')+\pi'(f(z'')),$$ and if without loss of generality $\pi'(z')\le \pi'(z'')$, then $\pi'(f(z'))\le \pi'(f(z''))$,  so we must have $\pi'(z')=\pi'(z'')$. Therefore $f(z')=f(z')$, so $z'=z''$, a contradiction.
\end{proof}
Hence the set $$Z_1=\bigsqcup_{i=1}^{r-1}\bigsqcup_{j\in [\min I_i,\max I_{i}-1]}(H_j+z_i)\subset A_8^{\dagger}+A_8^{\dagger}$$ is a disjoint union as $[\min I_{i},\max I_i-1]\subset [\min I_i,\min I_{i+1}-1]$.
\begin{clm} For $ \vec{v}\in \{0\}\times\{0,1\}^{k-1}$  and $x\in \{0\}\times \mathbb{Z}^{k-1}$, and $z\in H_t$ with  $t\in [\min I_i, \max I_{i}-1]$ for some $1\le i \le r-1$, we have
$$z+f(z)\not \in R_x + R_{x+\vec{v}}.$$
\end{clm}
\begin{proof}
Assume to the contrary that $z+f(z)\in R_x+R_{x+\vec{v}}$. First note that $\pi'(x)\ge y_{\min I_i}$ since $$2y_{\min I_i}\le \pi'(z)+\pi'(f(z))=\pi'(x)+\pi'(x+\vec{v})\le 2\pi'(x)+1.$$
Next, we note that $\pi'(x+\vec{v})\le y_{\max I_i}$. Indeed, if $\pi'(x+\vec{v})>y_{\max I_i}$, then as $R_x$ and $R_{x+\vec{v}}$ are non-empty, we have $\pi'(x+\vec{v})\geq y_{1+\max I_i}$ and $\pi'(x)\geq y_{\max I_i}$. However, then
$$\pi'(x+\vec{v})+\pi'(x)\geq y_{1+\max I_i}+y_{\max I_i}>\pi'(f(z))+\pi'(z),$$
a contradiction. Hence, $z,R_x,R_{x+\vec{v}}$ are all contained in the affine sub-lattice containing $\sqcup_{j \in I_i}H_j$, and $f(z)$ is not in this affine sub-lattice by construction, contradicting $z+f(z)\in R_x+R_{x+\vec{v}}$.
\end{proof}
Hence the sets
$$Z:=A_8^{\dagger}(+)A_8^{\dagger}= \bigsqcup_{\vec{v}\in \{0\}\times \{0,1\}^{k-1}}\bigsqcup_{x\in \{0\}\times\mathbb{Z}^{k-1}}R_x(+)R_{x+\vec{v}}\subset A_8^{\dagger}+A_8^{\dagger}$$
and $Z_1$ are disjoint.

The set of indices $\mathcal{I}=[1,\min I_{r}-1]\setminus \bigcup_{i=1}^{r-1} [\min I_i,\max I_{i}-1]$ whose corresponding hyperplanes $H_j$ were not accounted for by $Z_1$ are precisely those indices $j$ such that there exists $1 \le i\le r-1$ with  $j=\max I_i=\min I_{i+1}-1$. We will now find a third set $Z_2$ disjoint from $Z$ and $Z_1$ which accounts for the hyperplanes with indices in $\mathcal{I}$.

Consider two consecutive hyperplanes $H_j$ and $H_{j+1}$ with $j\in \mathcal{I}$, and let $i$ be such that $j=\max I_i$ and $j+1=\min I_{i+1}$. Note that by \Cref{minmaxrmk}, we have $y_j+1=y_{j+1}$ and the affine-linear function $\Phi_j(w)-\Phi_{j+1}(w+e_2):\{0\}\times \{y_j\}\times \mathbb{Z}^{k-2}\to \mathbb{Z}/d\mathbb{Z}$ is non-constant. Express $\Phi_j(w)-\Phi_{j+1}(w+e_2)=w\cdot a+b$ for some $a\in \{0\}\times \{0\}\times (\mathbb{Z}/d\mathbb{Z})^{k-2}$ with $a \ne 0$ and $b\in \mathbb{Z}/d\mathbb{Z}$. In particular, there is an index $s_j\in \{3,\ldots,k\}$ such that $a_{s_j}\ne 0$, which implies that the standard basis vector $e_{s_j}$ satisfies $$\Phi_j(w+e_{s_j})-\Phi_{j+1}(w+e_{s_j}+e_2)=(w+e_{s_j})\cdot a +b\ne w\cdot a+b=\Phi_j(w)-\Phi_{j+1}(w+e_2)$$ for all $w$. Rearranging, $$\Phi_j(w+e_{s_j})+\Phi_{j+1}(w+e_2)\ne \Phi_j(w)+\Phi_{j+1}(w+e_{s_j}+e_2)$$ for all $w$. Hence we have
$$\left(R_{w+e_2}+R_{w+e_{s_j}}\right) \cap \left(R_{w}+R_{w+e_2+e_{s_j}}\right) = \emptyset$$ for all $w$ since they lie in different translated $d\mathbb{Z}$-progressions.

\begin{clm} For $w\in \{0\}\times \{y_j\}\times \mathbb{Z}^{k-2}$ with $j\in \mathcal{I}$, we have
$$(R_{w+e_2}+R_{w+e_{s_j}})\cap Z=\emptyset.$$
\end{clm}
\begin{proof}
If $(w+e_2)+(w+e_{s_j})=x+(x+\vec{v})$ for some $x\in \{0\}\times \mathbb{Z}^{k-1}$ and $\vec{v}\in \{0\}\times \{0,1\}^{k-1}$, then by looking at the odd coordinates we see that $\vec{v}=e_2+e_{s_j}$ and hence $x=w$. But then
$$(R_{w+e_2}+R_{w+e_{s_j}})\cap Z=(R_{w+e_2}+R_{w+e_{s_j}})\cap (R_{w}(+)R_{w+e_2+e_{s_j}})=\emptyset.$$
\end{proof}
Hence the disjoint union
$$Z_2:=\bigsqcup_{j\in \mathcal{I}}\bigsqcup_{w \in \{0\}\times \{y_j\}\times \mathbb{Z}^{k-2}}R_{w+e_2}(+)R_{w+e_{s_j}}$$
is disjoint from $Z$. Finally, we prove a claim which implies $Z_1$ is disjoint from $Z_2$.

\begin{clm}
\label{clm319}
For any $1 \le s \le r-1$ and for any $z\in H_t$ with $t \in [\min I_s, \max I_s-1]$ and for any $w\in \{0\}\times \{y_j\}\times \mathbb{Z}^{k-2}$ with $j \in \mathcal{I}$, we have $$z+f(z)\not \in R_{w+e_2}+R_{w+e_{s_j}}.$$
\end{clm}

\begin{proof} Assume for the sake of contradiction that $z+f(z)\in R_{w+e_2}+R_{w+e_{s_j}}$. Let $i$ be such that $j=\max I_i=\min I_{i+1}-1$. First, suppose that $\pi'(z)=y_t\le y_j-1$. Then $s \le i$ as if $i<s$ then $j=\min I_{i+1}-1< \min I_s\leq t$, and therefore $y_j<y_t$, a contradiction. Therefore $\pi'(f(z))= y_{\max I_s +1}\le y_{\max I_i +1}= y_{j+1}=y_j+1$ by \Cref{minmaxrmk}, so we obtain the contradiction
$$\pi'(z+f(z))\le y_j-1+(y_j+1)<2y_j+1=\pi'(w+e_2)+\pi'(w+e_{s_j}).$$
Next, suppose $\pi'(z)=y_t=y_j$. Then $j=t$, contradicting that $\mathcal{I}$ is disjoint from $[\min I_s,\max I_s-1]$ by construction of $\mathcal{I}$. Finally, suppose that $\pi'(z)=y_t\ge y_j+1$. Then as $\pi'(f(z))> \pi'(z)$, we have the contradiction
$$\pi'(z+f(z))\ge 2y_j+2>2y_j+1=\pi'(w+e_2)+\pi'(w+e_{s_j}).$$
\end{proof}
Hence $Z,Z_1$ and $Z_2$ are disjoint subsets of $A_8^\dagger+A_8^\dagger$. Note for $x_1,x_2\in \{0\}\times \mathbb{Z}^{k-1}$ we have
$$|R_{x_1}(+)R_{x_2}|-|R_{x_1}|-|R_{x_2}|\ge \begin{cases}0&1_{x_1\in \pi(A_8^{\dagger})}+1_{x_2\in \pi(A_8^{\dagger})}=0\\-n_1 & 1_{x_1\in \pi(A_8^{\dagger})}+1_{x_2\in \pi(A_8^{\dagger})}=1\\-1 & 1_{x_1\in \pi(A_8^{\dagger})}+1_{x_2\in \pi(A_8^{\dagger})}= 2\end{cases}$$

For $\vec{v}\in \{-1,0,1\}^k$ we have 
from \eqref{copiA8dagger} and \Cref{Surfobs} (possibly applied after some coordinate hyperplane reflections),
\begin{align*}|\{x\in \{0\}\times\mathbb{Z}^{k-1}: 1_{x\in \pi(A_8^{\dagger})}+1_{x+\vec{v}\in \pi(A_8^{\dagger})}=1\}|&\le 2|\co(\pi(A^\dagger_8))\setminus \pi(A^\dagger_8)|+2(k-1)n_1^{-1}n_{k,0}^{-1}|B|\\
&\le \delta^{\frac{1}{320}-3c}n_1^{-1}|B|\end{align*}
(a bound which in particular applies when $x=w+e_2$ and $x+\vec{v}=w+e_{s_j}$) and
\begin{align*}|\{x\in \{0\}\times \mathbb{Z}^{k-1}: 1_{x\in \pi(A_8^{\dagger})}+1_{x+\vec{v}\in \pi(A_8^{\dagger})}=2\}|\le |\pi(B)|\le n_{k,0}^{-1}|B|\le \delta^{\frac{1}{320}-3c}|B|.\end{align*}

Therefore, we have (taking $\vec{v}\in \{0\}\times \{0,1\}^{k-1}$ and $x\in \{0\}\times \mathbb{Z}^{k-1}$)
\begin{align*}|A_8^{\dagger}+A_8^{\dagger}| \ge& |Z| + |Z_1| + |Z_2|\\
=&\left(\sum_{\vec{v}}\sum_{x} |R_x(+)R_{x+\vec{v}}|\right)+\sum_{i=1}^{r-1}\sum_{j\in [\min I_i,\max I_{i}-1]}|H_j|\\&+\sum_{j \in \mathcal{I}}\sum_{w\in \{0\}\times \{y_j\}\times \mathbb{Z}^{k-2}}|R_{w+e_2}(+)R_{w+e_{s_j}}|\\
\ge&\left(\sum_{\vec{v}}\sum_{x} |R_x|+|R_{x+\vec{v}}|\right)+\sum_{i=1}^{r-1}\sum_{j\in [\min I_i,\max I_{i}-1]}|H_j|\\&+\left(\sum_{j \in \mathcal{I}}\sum_{w\in \{0\}\times \{y_j\}\times \mathbb{Z}^{k-2}}|R_{w+e_{s_j}}|\right)-\delta^{\frac{1}{320}-4c}|B|\\
=&2^k|A_8^{\dagger}|+\left(\sum_{j\in [1,\min I_{r}-1]}|H_j|\right)-\delta^{\frac{1}{320}-4c}|B|.
\end{align*}
Note that in the second inequality above, we only need to apply the bound for $x=w+e_2$ and $x+\vec{v}=w+e_{s_j}$ only once for every one of the $k-2$ possible values of $e_{s_j}-e_2$.

If we consider the same process ran in reverse, we produce another collection of intervals $I'_1, \hdots I'_{r'}\subset \{1,\ldots \ell\}$ with $\ell\in I_1'$ such that $$ |A_8^{\dagger}+A_8^{\dagger}| \ge 2^k|A_8^{\dagger}|+\left(\sum_{j\in [\max I'_{r'}+1,l]}|H_j|\right)-\delta^{\frac{1}{320}-4c}|B|.$$

As $A_8^{\dagger}$ is reduced, we have that $r'\ge 2$. As $\sqcup_{i \in I_r} H_{i}$ is not reduced, we have that $I_r \subset I_1'\subset [\max I_r'+1,\ell]$. Averaging the two inequalities, we get by \eqref{AA8dagger} that
$$d_k(A_8^{\dagger})\ge \frac{1}{2}|A_8^{\dagger}|-\delta^{\frac{1}{320}-4c}|B|\ge \frac{\epsilon_0}{4}|B|,$$
contradicting \eqref{A8dk} that $d_k(A_8^{\dagger})\le \delta^{\frac{1}{320}-8c}|B|$. The conclusion that $d=1$ follows.

\subsection{Reductions Part 4: Filling in the rows to create $A_+\supset A_5$}
\label{reductionAplus}
We recall that we have just shown that $d=1$, or equivalently all rows $R_x$ of $A_5$ satisfy $\co(R_x)=\coo(R_x)$. We now show that filling in all of the rows of $A_5$ does not change the size of $A_5$ or $d_k(A_5)$ too much.

\subsubsection{$A_+ \supset A_5$ with all rows filled in}

Let $A_+ \supset A_5$ be obtained by replacing each row $R_x$ of $A_5$ with $\co(R_x)$. 
Now $A_+$ is reduced as $A_5$ is reduced. Also by \eqref{AA5dagger} and \eqref{A5beta} we have \begin{align}\label{AA+}|A \Delta A_+|\le |A\setminus A_5|+|A_+\setminus A_5|\le \delta^{\frac{1}{20}-15c}|B|+2\delta^{\frac{1}{10}}|A_5|\le \delta^{\frac{1}{20}-16c}|B|.\end{align}
Furthermore, $\pi(A_+)=\pi(A_5)$, so by \eqref{A5daggerinfo} we have
$$|\co(\pi(A_+))\setminus \pi(A_+)|=|\co(\pi(A_5))\setminus \pi(A_5)|\le \delta^{\frac{1}{20}}|\pi(B)|.$$

\begin{obs}
\label{obs:dkAplus}
$d_k(A_+)\leq \delta^{\frac{1}{20}-17c}|B|.$
\end{obs}
\begin{proof}
We begin the proof with the following general claim.
\begin{clm}\label{disjointunion}
Given a finite family $\mathcal{F}$ of finite subsets $Z\subset \mathbb{Z}$ and a parameter $\rho$ such that $|\co(Z)| \le (1+\rho)|Z|$ for all $Z\in \mathcal{F}$, we have $$\left|\bigcup_{Z\in \mathcal{F}} \co(Z)\right| \le (1+2\rho)\left|\bigcup_{Z\in \mathcal{F}} Z\right|.$$
\end{clm}
\begin{proof}
Consider a minimal subfamily $\mathcal{G}$ such that $\bigcup_{Z\in \mathcal{G}}\co(Z)=\bigcup_{Z\in \mathcal{F}} \co(Z)$. Then we can enumerate the sets $Z_1,Z_2,\ldots,Z_r$ appearing in $\mathcal{G}$ such that for $i<j$ we have $\co(Z_i)\cap \co(Z_{j})\ne \emptyset$ implies $j=i+1$.

Let $\mathcal{G}', \mathcal{G}''$ be the $Z_i$ with odd and even indices respectively. Then $\bigcup_{Z\in \mathcal{F}} \co(Z)=(\bigsqcup_{Z \in \mathcal{G}'} \co(Z))\cup (\bigsqcup_{Z \in \mathcal{G}''}\co(Z))$, where both are disjoint unions. Then
$$\left|\bigcup_{Z\in \mathcal{F}} \co(Z)\setminus \bigcup_{Z\in \mathcal{F}} Z\right|\le \left|\bigcup_{Z\in \mathcal{G}} \co(Z)\setminus \bigcup_{Z\in \mathcal{G}} Z\right| \le \sum_{Z\in \mathcal{G}'\cup \mathcal{G}''} |\co(Z)\setminus Z| \le \sum_{Z\in \mathcal{G}'\cup \mathcal{G}''} \rho|Z|\le 2\rho\left|\bigcup_{Z\in \mathcal{F}} Z\right|.$$
\end{proof}

Now, for nonempty rows $R_{x_{1}},R_{x_{2}}$ of $A_5$, we have by \eqref{A5beta} that $|\co(R_{x_{1}})|\le (1+2\delta^{\frac{1}{10}})|R_{x_{1}}|$ and $|\co(R_{x_{2}})| \le (1+2\delta^{\frac{1}{10}})|R_{x_{2}}|$, so $$|\co(R_{x_{1}}+R_{x_{2}})|=|\co(R_{x_{1}})|+|\co(R_{x_{2}})|-1 \le (1+2\delta^{\frac{1}{10}})(|R_{x_{1}}|+|R_{x_{2}}|)-1\le (1+4\delta^{\frac{1}{10}})|R_{x_{1}}+R_{x_{2}}|.$$
Taking the sets $Z_i$ to be the pairwise row sums $R_{x_{1}}+R_{x_{2}}$ with $x_1+x_2=x$ fixed, and summing the inequality in \Cref{disjointunion} over all $x\in\{0\}\times\mathbb{Z}^{k-1}$, we obtain
$$|A_++A_+| \le (1+8\delta^{\frac{1}{10}})|A_5+A_5|.$$ 
Hence by \eqref{A5dk}, we thus have
\begin{align*}d_k(A_+) &\le (1+8\delta^{\frac{1}{10}})d_k(A_5)+2^k\cdot8\delta^{\frac{1}{10}}|A_5|\\
&\le \delta^{\frac{1}{20}-17c}|B|.\end{align*}
\end{proof}

\subsection{Reductions Part 5: Approximating $A_+$ by $A_\star\subset A_+$ with few vertices in $\co(A_\star)$ and an extra technical condition}
\label{sectionAstarapprox}
We now construct a set $A_\star\subset A_+$ with $|A_+\setminus A_\star|=o(1)|B|$, for which the following $3$ conditions hold.
\begin{enumerate}
    \item $|V(A_\star)|$, which we recall is the number of vertices of $\cooo(A_\star)$, is bounded by a function of $\delta$.
    \item $|\co(\pi(A_\star))\setminus \pi(A_\star)|=o(1)|\pi(B)|$.
    \item The technical condition \Cref{pullover} holds.
\end{enumerate}
We show this using a double recursion, and the bounds we obtain will no longer be powers of $\delta$.

In \Cref{AstarPart1}, we prove \Cref{polyapprox}, which shows that we can ensure that 1 holds. This is accomplished by showing an analogous approximation result for (continuous) polytopes, and then transitioning to the discrete setting using \Cref{contdisc}.

In \Cref{AstarPart2}, we prove \Cref{firstit}, which shows that we can ensure that both 1 and 2 hold. \Cref{alphaprop'} by itself shows that 2 holds, so we alternate applications of \Cref{polyapprox} and \Cref{alphaprop'}, and show that at some point both 1 and 2 hold simultaneously.

In \Cref{AstarPart3}, we prove \Cref{Astarreduction}, which shows that we can ensure that all of 1,2,3 hold. To do this, we show \Cref{overpulling}, which shows that we can ensure 1 and 3 hold. Similarly to the proof of \Cref{firstit}, we alternate applications of \Cref{firstit} and \Cref{overpulling}, and show that at some point 1,2,3 hold simultaneously.

Finally, in \Cref{AstarConstruction}, we apply \Cref{Astarreduction} to $A_+$ to construct $A_\star$.

\subsubsection{$A_\star\subset A_+$ with   $|V(A_\star)|$ small: Setup Part 1}\label{AstarPart1}
\begin{prop}\label{polyapprox}
There exist $\ll$-dependencies such that for constants
$$n_{k,0}^{-1}\ll \epsilon_0\le 1 \text{ and }16 \epsilon_0^{-1} k(k+1) \min\{n_i\}^{-1}\le \alpha \le 1 ,\text{ and }\ell^{-1}\ll \alpha,$$
the following holds.

For any set of points $C\subset B$ with $|\co(C)|\ge \frac{\epsilon_0}{2}|B|$, there exists $Q=\co(Q)\cap C\subset C$ with $|V(Q)|\le \ell$ and $V(Q)\subset V(C)$, such that $|\co(Q)|\ge (1-\alpha)|\co(C)|$, and if $C$ has all rows intervals, then $Q$ has all rows intervals.

Furthermore, if $\alpha=\min\{n_i\}^{-\frac{2}{(k-1)\lfloor k/2\rfloor+2}}$, then there is a constant $\tau_k$ depending only on $k$ (concretely $\tau_k=(\frac{\tau}{2k})^{-\frac{k-1}{2}}$ where $\tau $ is as in \Cref{contpolyapprox}) such that the dependency between $\ell$ and $\alpha$ can be taken to be the (decreasing) function $$\ell\ge \ell(\alpha):=\tau_k\min\{n_i\}^{\frac{k-1}{(k-1)\lfloor k/2\rfloor+2}}$$
\end{prop}

To do this, we first consider a continuous analogue, which was proved constructively by Gordon, Meyer, and Reisner \cite{BestApproximation}.
\begin{lem}\label{contpolyapprox}
For any $\alpha>0$, there exists $\ell'=\ell'(\alpha)$ such that the following is true. For any polytope $\widetilde{C}$, there is a polytope $\widetilde{Q}$ which is the convex hull of at most $\ell'$ vertices of $\widetilde{C}$ with $|\widetilde{Q}|\ge (1-\alpha)|\widetilde{C}|$. There is an absolute constant $\tau$ independent of $k$ such that, for $\alpha$ sufficiently small in terms of $k$, we can take $\ell'=(\frac{\tau}{k}\alpha)^{-\frac{k-1}{2}}$.
\end{lem}

\begin{proof}[Proof of \Cref{contpolyapprox}]
It is enough to find such a polytope $\widetilde{Q}$ with vertices contained inside $\widetilde{C}$. Indeed, a simple convexity argument shows that as we vary the vertices of $\widetilde{Q}$ the maximum volume is attained when all vertices of $\widetilde{Q}$ are among the vertices of $\widetilde{C}$. The result then follows from \cite[Theorem 3]{BestApproximation}.
\end{proof}

\begin{proof}[Proof of \Cref{polyapprox}]
Let $\widetilde{C}=\cooo(C)$ be the continuous convex hull of $C$. By \Cref{contpolyapprox}, there exists a polytope $\widetilde{P}$ with $|V(\widetilde{P})|\le \ell'(\frac{\alpha}{2})$,  $V(\widetilde{P})\subset V(\widetilde{C})=V(C)$, and $|\widetilde{P}|\ge (1-\frac{\alpha}{2})|\widetilde{C}|$. Let $Q=\widetilde{P}\cap C$, and note that $\cooo(Q)=\widetilde{P}$.

We thus have by \Cref{contdisc} that
\begin{align*}
|\co(Q)|&\ge |\widetilde{P}|-2k(k+1)\min\{n_i\}^{-1}|B|\\
&\ge \left(1-\frac{\alpha}{2}\right)|\widetilde{C}|-2k(k+1)\min\{n_i\}^{-1}|B|\\
&\ge \left(1-\frac{\alpha}{2}\right)|\co(C)|-4k(k+1)\min\{n_i\}^{-1}|B|\\
&\ge (1-\alpha)|\co(C)|.
\end{align*}
For $\alpha=\min\{n_i\}^{-\frac{2}{(k-1)\lfloor k/2\rfloor+2}}$, we see that $\ell=\ell'(\frac{\alpha}{2})$, yielding  $\ell=\tau_k\min\{n_i\}^{\frac{k-1}{(k-1)\lfloor k/2\rfloor+2}}$.
\end{proof}

\subsubsection{$A_\star\subset A_+$ with $|V(A_\star)|$ and $|\co(\pi(A_\star))\setminus \pi(A_\star)|$ small: Setup part 2}
\label{AstarPart2}

At this point in the proof, we will lose polynomial control over the doubling constant, so for convenience we will work with purely qualitative statements from now on. The following proposition is a simple to use qualitative analogue of \Cref{alphaprop'}.

\begin{prop}\label{qualalphaprop}
There exist $\ll$-dependencies such that for constants
$$n_{k,0}^{-1}\ll f \ll \epsilon_0\le 1\text{ and }f\ll h_1,h_2$$
the following is true.

If $A'\subset B$ and $|A'|\ge \frac{\epsilon_0}{2}$ with $d_k(A')\le f|B|$, then there is a subset of the rows $A''\subset A'$ such that
$$|A'\setminus A''|\le h_{1}|B|,\qquad|\co(\pi(A''))\setminus \pi(A'')|\le h_{2}|\pi(B)|.$$
\end{prop}

\begin{proof}[Proof of \Cref{qualalphaprop}]
Applying \Cref{alphaprop'} with $\sigma=f$, $\lambda=1$ and $\alpha=\frac12$, we see that the result is true with the $\ll$ dependencies $f^{\frac{1}{2}-c}\le h_1$ and $f^{\frac{1}{2}}\le h_2$.

\end{proof}

\begin{prop}\label{firstit}
There exist $\ll$-dependencies such that for constants
$$n_{k,0}^{-1}\ll f \ll \epsilon_0\le 1\text{ and }f\ll h_3,h_4 $$
the following is true.

If $A'\subset B$ has all rows intervals, $|A'|\ge \frac{2\epsilon_0}{3}|B|$ and $d_k(A')\le f|B|$, then there exists a subset $A''\subset A'$ with each row an interval, such that
\begin{enumerate}
    \item $|\co(\pi(A''))\setminus \pi(A'')|\le h_3|\pi(B)|$
    \item $|A'\setminus A''|\le h_4|B|$
    \item $|V(A'')|\le \ell(f)$ with $\ell$ as in \Cref{polyapprox}.
\end{enumerate}
\end{prop}
\begin{proof}
We may assume that $h_3=h_4=h$ for some $h$. We start by taking the dependencies
$$n_{k,0}^{-1}\ll f \ll \epsilon_0 \le 1$$ to work for both \Cref{polyapprox} with $\alpha=f$, and for \Cref{qualalphaprop}, and we take the dependency $f\ll h$ to be the dependency for $f\ll h_2$ from \Cref{qualalphaprop}.

First, we show by induction that for any $\gamma\in \mathbb{N}$ there are dependencies
$$f\ll g_0 \ll \ldots \ll g_{\gamma-1}\ll \epsilon_0,h/2$$ with the composite dependencies $f\ll \epsilon_0,h$ refining the existing dependencies,
such that we can create a nested sequence of sets $$A'=A_0'\supset A_0'' \supset A_1' \supset A_1''\supset \ldots \supset A_\gamma'$$
where
\begin{itemize}
    \item $A_i''\subset A_i'$ has all rows intervals, and $|\co(\pi(A_i''))\setminus \pi(A_i'')|\le g_i |\pi(B)|$ and $|A_i'\setminus A_i''|\le g_i|B|$
    \item $A_{i+1}'\subset A_i''$ has all rows intervals, $|V(A_{i+1}')|\le \ell(f)$, and $|A_i''\setminus A_{i+1}'|\le f|B|$.
    \item $|A'\setminus A'_\gamma|\le h|B|$.
\end{itemize}
Assume we can find such dependencies for $\gamma$, we will show we can find dependencies $g_{\gamma-1}\ll g_{\gamma}\ll \epsilon_0,h/2$ to work for $\gamma+1$. By requiring the dependency $g_{\gamma-1}\ll g_\gamma$ to be at least as strong as the existing $g_{\gamma-1}\ll \epsilon_0,h/2$ dependencies, we can ensure the composite dependency is at least as strong as the existing $g_{\gamma-1}\ll \epsilon_0,h/2$ dependencies. 

We start by showing that we can choose the dependencies so that \Cref{qualalphaprop} can be applied to $A_\gamma'$ to produce the desired $A_\gamma''$. Note that
$|A'\setminus A_\gamma'|\le (\gamma f + g_0 +\ldots + g_{\gamma-1})|B|\le 2\gamma g_{\gamma-1}|B|$, so by requiring $g_{\gamma}\le \frac{1}{2\gamma}\frac{\epsilon_0}{12}$, we can ensure that
$$|A_\gamma'|\ge \left(\frac{2\epsilon_0}{3}-\frac{\epsilon_0}{12}\right)|B|=\frac{7\epsilon_0}{12}|B|.$$
Next, because $|A'\setminus A'_\gamma|\le (2\gamma g_{\gamma-1})|B|$, by \Cref{dkobs} we have
$$d_k(A'_\gamma)\le (f+2^{k+1}\gamma g_{\gamma-1})|B|\le (2^{k+1}\gamma+1)g_{\gamma-1}|B|.$$
By requiring that the dependencies $g_{\gamma-1}\ll \frac{1}{2^{k+1}\gamma+1}g_\gamma$ and $\frac{1}{2^{k+1}\gamma+1}g_\gamma\ll h$ are at least as strong as the $f\ll h_1$ and $f\ll h_2$ dependencies from \Cref{qualalphaprop}, we can apply \Cref{qualalphaprop} to obtain the desired $A_\gamma''$.

To construct $A'_{\gamma+1}$ satisfying the second point, note that because we have ensured the $n_{k,0}^{-1}\ll f \ll \epsilon_0$ dependencies are at least as strong as when we started, we may apply \Cref{polyapprox} to $A_\gamma''$ with $\alpha=f$ to create $A_{\gamma+1}'=\co(A_{\gamma+1}')\cap A_\gamma''$ with $|V(A_{\gamma+1}')|\le \ell(f)$, and
$$|A_\gamma''\setminus A_{\gamma+1}'|=|A_\gamma''\setminus \co(A_{\gamma+1}')|\le |\co(A_\gamma'')\setminus \co(A_{\gamma+1}')|\le f|\co(A_\gamma'')|\le f|B|.$$
Finally, to satisfy the third point, note that by the first and second points we have $|A'\setminus A'_{\gamma+1}|\le ((\gamma+1)f+g_0+\ldots+g_\gamma)|B| \le 2(\gamma+1)g_{\gamma}|B|$, so by requiring $g_{\gamma}\le \frac{h}{2(\gamma+1)}$ we can ensure that $|A'\setminus A'_{\gamma+1}|\le h|B|$.

By construction, for any fixed $\gamma$ if we set in the above construction $A''=A_i'$ for $1 \le i \le \gamma$, then this satisfies all of the requirements of $A''$ except possibly that $|\co(\pi(A''))\setminus \pi(A'')|\le h|B|$. Take now $$\gamma=\left\lceil \frac{2}{h} \right\rceil +2$$ so that $\frac{h(\gamma-1)}{2}>1$. If all of  $A_1',\ldots,A_{\gamma}'$ have the property that
$|\co(\pi(A_i'))\setminus \pi(A_i')| > h|\pi(B)|$, then noting that because $\pi(A_{i+1}')\subset \pi(A_i'')$ we have $|\co(\pi(A'_{i+1}))\setminus\pi(A'_{i+1})|-|\co(\pi(A''_i))\setminus\pi(A''_i)|=|\pi(A''_{i})\setminus \pi(A'_{i+1})|-|\co(\pi(A''_{i}))\setminus \co(\pi(A'_{i+1}))|$, and we deduce that
\begin{align*}
|\co(\pi(A_1'))\setminus \pi(A'_\gamma)|&\ge |\co(\pi(A_1'))\setminus \pi(A_1')|+\sum_{i=1}^{\gamma-1} |\pi(A_i'')\setminus \pi(A_{i+1}')|\\&\geq \sum_{i=1}^{\gamma-1} |\co(\pi(A'_{i+1}))\setminus\pi(A'_{i+1})|-|\co(\pi(A''_i))\setminus\pi(A''_i)|\\
&\geq \sum_{i=1}^{\gamma-1}(h-g_i)|\pi(B)|\\
&\ge \frac{h(\gamma-1)}{2}|\pi(B)|>|\pi(B)|,
\end{align*}
a contradiction. Let $1\le i_0\leq \gamma$ be an index such that $|\co(\pi(A'_{i_0}))\setminus \pi(A'_{i_0})|\leq h|\pi(B)|$. Then we conclude by setting $A''=A'_{i_0}$.
\end{proof}

\subsubsection{$A_\star\subset A_+$ with  $|V(A_\star)|$ and $|\co(\pi(A_\star))\setminus \pi(A_\star)|$ small and one further technical condition: Setup part 3}
\label{AstarPart3}

\begin{defn}For every $A'\subset B$, let $\mathcal{T}^+(A')$ (resp. $\mathcal{T}^-(A')$) be a triangulation of the upper (resp. lower) convex hull of $\cooo(A')$ with respect to the $e_1$ direction, projected under $\pi$ to $\{0\}\times\mathbb{R}^{k-1}$, so in particular every $\widetilde{T}\in\mathcal{T}^+(A')$ has $\widetilde{T}\subset \{0\}\times\mathbb{R}^{k-1}$. We ensure that if $\cooo(A'_1)=\cooo(A'_2)$, then $\mathcal{T}^+(A'_1)=\mathcal{T}^+(A'_2)$ and $\mathcal{T}^-(A'_1)=\mathcal{T}^-(A'_2)$.
\end{defn}

\begin{notn}
For a simplex $\widetilde{T}\subset \{0\}\times \mathbb{R}^{k-1}$, we will write $$T:=\widetilde{T}\cap \{0\}\times \mathbb{Z}^{k-1},\text{ and }T^o:=\widetilde{T}^o\cap (\{0\}\times \mathbb{Z}^{k-1})$$
where $\widetilde{T}^o$ is the interior of $\widetilde{T}$.
\end{notn}

\begin{defn}\label{YofT}
Given $\widetilde{T} \subset \{0\}\times \mathbb{R}^{k-1}$ with integral vertices, and a set $\mathcal{W} \subset \{0\}\times \mathbb{Z}^{k-1}$, for every $x\in T$ we define the set $$Y_{\mathcal{W}}(x):=((x+\mathcal{W})\cap T)\cup V(T)\subset T.$$

\end{defn}

\begin{prop}\label{overpulling}
There exists $\ll$-dependencies on constants
$$n_{k,0}^{-1}\ll f_1,f_2$$
such that the following is true.
Let $\nu>0$, $A'\subset B$ with  $d_k(A')\le f_1|B|$ and $|\co(\pi(A'))\setminus \pi(A')|\le f_2|\pi(B)|$. Suppose we have sets $\mathcal{W}_T\subset \{0\}\times \mathbb{Z}^{k-1}$ with $|\mathcal{W}_T| \le \nu$ for every $\widetilde{T} \in\mathcal{T}^+(A')\cup\mathcal{T}^{-}(A')$. Then there exists a subset of rows $A''\subset A'$ such that $$|A'\setminus A''|\le (2\nu+2)(f_1+2^{k+1}f_2)|B|$$
which satisfies the following additional properties.
\begin{enumerate}
    \item $\co(A'')=\co(A')$.
    \item 
For every $\widetilde{T}\in \mathcal{T}^+(A')\cup \mathcal{T}^-(A')=\mathcal{T}^+(A'') \cup \mathcal{T}^-(A'')$, if $x\in T^o \setminus V_\pi(A'')$,  $y\in Y_{\mathcal{W}_T}(x)$, $z \in \pi(B)$ and $\vec{v}\in\{0\}\times\{0,1\}^{k-1}$ with $x+y=z+z+\vec{v}$ and $R_{x}+R_{y}, R_{z}+R_{z+\vec{v}}$ nonempty, then $(R_{x}+R_{y})\cap (R_{z}+R_{z+\vec{v}})\ne \emptyset$.
\end{enumerate}
\end{prop}
\begin{proof}
For every $\widetilde{T} $, write $\mathcal{W}_T:=\{\vec{w}_{T,i} : 1\le i \le \nu \}$. Let $X_{T,i}\subset T^o\setminus V_\pi(A')$ be those $x$ such that $x+\vec{w}_{T,i}\in Y_{\mathcal{W}_T}(x)$, and such that, when writing $x+x+\vec{w}_{T,i}=z+z+\vec{v}$ with $\vec{v}\in\{0\}\times\{0,1\}^{k-1}$, we have $$R_{x}+R_{x+\vec{w}_{T,i}},R_{z}+R_{z+\vec{v}}\text{ nonempty and }(R_{x}+R_{x+\vec{w}_{T,i}}) \cap (R_{z}+R_{z+\vec{v}}) = \emptyset.$$ For $\star\in \{+,-\}$, let $X^\star_i:=\bigsqcup_{\widetilde{T}\in \mathcal{T}^\star(A')} X_{T,i}$, let $y_i^\star$ be defined on $X_i^\star$ by setting $y^\star_i(x)=x+\vec{w}_{T,i}$ for $x\in X_{T,i}$, and set the disjoint union  $$Z^\star_i:=\bigsqcup_{\widetilde{T}\in \mathcal{T}^\star(A') }\bigsqcup_{x\in X_{T,i}} R_x+R_{y^\star_i(x)}\subset A'+A'.$$
Here we note the union is disjoint as $\frac{1}{2}(x+y_i^\star(x))\in \widetilde{T}^{\circ}$, which are disjoint for distinct $\widetilde{T}$, and for a given $\widetilde{T}$ we have $\{2x+\vec{w}_{T,i}\}_{x\in T}$ are distinct.

For $\star \in \{+,-\}$ let $X^\star_0\subset \pi(B)\setminus V_\pi(A')$ be those $x$ such that there exists $\widetilde{T}\in\mathcal{T}^\star(A')$ with $x\in T^o$ and $y^\star_0(x)\in V(T)$, such that writing $x+y^\star_0(x)=z+z+\vec{v}$ with $\vec{v}\in\{0\}\times\{0,1\}^{k-1}$, $$R_{x},R_{z}+R_{z+\vec{v}}\text{ nonempty, and }(R_{x}+R_{y^\star_0(x)}) \cap (R_{z}+R_{z+\vec{v}})= \emptyset.$$ Set the disjoint union
$$Z^{\star}_0:=\bigsqcup_{x\in X^{\star}_0} R_x+R_{y^\star_0(x)}\subset A'+A'.$$
Here, the union is disjoint because $\frac{1}{2}(x+y_0^\star(x))\in \widetilde{T}^{\circ}$, which are disjoint for distinct $\widetilde{T}$, and for a given $\widetilde{T}$ we have $\{\frac{1}{2}(\widetilde{T}^{\circ}+y)\}_{y\in V(T)}$ are disjoint.

Finally, set $X:=\bigcup_{(\star,i) \in \{+,-\}\times \{0,\ldots,\nu\}}X^\star_i$, and let $A''=A'\setminus \bigsqcup_{x\in X} R_x$, so that $|A'\setminus A''|=\sum_{x \in X}|R_x|$. By construction $A''$ satisfies the properties 1 and 2, so it suffices to show $|A'\setminus A''|\le (2\nu+2)(f_1+2^{k+1}f_2)|B|$.

Set the disjoint union
$$Z:=A'(+)A'=\bigsqcup_{\vec{v}\in \{0\}\times \{0,1\}^{k-1}}\bigsqcup_{x\in \{0\}\times\mathbb{Z}^{k-1}}R_x(+)R_{x+\vec{v}}\subset A'+A',$$
and note that by construction $Z\cap Z^\star_i=\emptyset$ for all $\star,i$.  Choose  $(\star,i)\in \{+,-\}\cup \{0,\ldots,\nu\}$ so that $$|A'\setminus A''|=\sum_{x \in X}|R_x| \le (2\nu+2) \sum_{x \in X^{\star}_i}|R_x|.$$

Note that for $0\ne \vec{v}\in \{0\}\times \{0,1\}^{k-1}$ and $x\in \{0\}\times \mathbb{Z}^{k-1}$ we have
$$|R_{x}(+)R_{x+\vec{v}}|-|R_{x}|-|R_{x+\vec{v}}|\ge \begin{cases}0&|\{x_1,x_2\}\cap \pi(A')|=0\\-n_1 & |\{x_1,x_2\}\cap \pi(A')|=1\\-1 & |\{x_1,x_2\}\cap \pi(A')|= 2.\end{cases}$$
By \Cref{Surfobs}, we have for $0\ne \vec{v}\in \{0\}\times\{0,1\}^{k-1}$ that
\begin{align*}|\{x\in \{0\}\times\mathbb{Z}^{k-1}: |\{x,x+\vec{v}\}\cap \pi(A')|=1\}|&\le 2|\co(\pi(A'))\setminus \pi(A')|+2(k-1)n_1^{-1}n_{k,0}^{-1}|B|\\
&\le 3f_2n_1^{-1}|B|\end{align*}
and
\begin{align*}|\{x\in \{0\}\times\mathbb{Z}^{k-1}: |\{x,x+\vec{v}\}\cap \pi(A')|=2\}|\le |\pi(B)|\le n_{k,0}^{-1}|B|\le f_2|B|.\end{align*}
We therefore have (taking $\vec{v}\in \{0\}\times \{0,1\}^{k-1}$ and $x\in \{0\}\times \mathbb{Z}^{k-1}$)
\begin{align*}|A'+A'| \ge& |Z| + |Z^\star_i| \\
=&\left(\sum_{\vec{v}}\sum_{x} |R_x(+)R_{x+\vec{v}}|\right)+\sum_{x\in X^\star_i}|R_x+R_{y^\star_i(x)}|\\
\ge&2^k|A'|-(2^{k-1}-1)(3f_2+f_2)|B|-|\pi(B)|+\sum_{x\in X^\star_i}|R_x|\\\ge& 2^k|A'|-2^{k+1}f_2|B|+\frac{1}{2\nu+2}|A'\setminus A''|.
\end{align*}
Hence, 
$$|A'\setminus A''|\leq (2\nu+2)( d_k(A')+2^{k+1}f_2|B|)\leq (2\nu+2)(f_1+2^{k+1}f_2)|B|.$$
\end{proof}

\begin{prop}
\label{Astarreduction}
There exist $\ll$-dependencies such that for constants
$$n_{k,0}^{-1}\ll e\ll \epsilon_0\le 1,\text{ and }e\ll h_5,h_6$$ and a function $H_7(e)$ with $H_7(e)\to \infty$ as $e\to 0$ (we may take $H_7=\ell$, the function from \Cref{polyapprox}) such that
the following is true. If $A'\subset B$ has all rows intervals and $|A'|\ge \frac{3\epsilon_0}{4}|B|$, $d_k(A')\le e|B|$, and if for every simplex $\widetilde{T}\subset \{0\}\times \mathbb{R}^{k-1}$ with integral vertices we have a set $\mathcal{W}_T\subset \{0\}\times \mathbb{Z}^{k-1}$ with $|\mathcal{W}_T| \le \nu$ for a constant $\nu=\nu(k)$, then there exists $A''\subset A'$ with all rows intervals and $$|A'\setminus A''|\le h_5|B|,\qquad|\co(\pi(A''))\setminus \pi(A'')|\le h_6|\pi(B)|$$
which satisfies the following additional properties. 
\begin{enumerate}
    \item $|V(A'')|\le H_7(e)$
    \item We have for every $\widetilde{T}\in\mathcal{T}^+(A'')\cup\mathcal{T}^{-}(A'')$, if $x\in T^o \setminus V_\pi(A'')$,  $y\in Y_{\mathcal{W}_T}(x)$, $z \in \pi(B)$ and $\vec{v}\in\{0\}\times\{0,1\}^{k-1}$ with $x+y=z+z+\vec{v}$ and $R_{x}+R_{y}, R_{z}+R_{z+\vec{v}}$ nonempty, that $(R_{x}+R_{y}) \cap (R_{z}+R_{z+\vec{v}}) \ne \emptyset$.
\end{enumerate}
\end{prop}

\begin{proof}
We may assume that $h_5=h_6=h$ for some $h$. Taking $H_7=\ell$ ensures that $H_7(e)\le \ell(h)$ because $\ell$ is a decreasing function. What follows below is essentially the same proof as \Cref{firstit} with \Cref{polyapprox} replaced with \Cref{overpulling} and \Cref{qualalphaprop} replaced with \Cref{firstit}.

We start by taking the dependencies
$$n_{k,0}^{-1}\ll e \ll \epsilon_0 \ll 1\text{ and }\epsilon_0\ll h$$ to work for both \Cref{overpulling} with $f_1=f_2=e$, and for \Cref{firstit} with $f=e$ and $h_3=h_4=h$.

First, we prove by induction that for any $\gamma\in \mathbb{N}$ there are dependencies
$$e=g_{-1}'\ll g_0 \ll g_0'\ll g_1 \ll g_1'\ll \ldots \ll g_{\gamma-1} \ll g_{\gamma-1}'\ll \epsilon_0,h/2$$ with the composite dependencies $e\ll \epsilon_0,h$ refining the existing dependencies,
such that we can create a nested sequence of sets $$A'=A_0'\supset A_0'' \supset A_1' \supset A_1''\supset \ldots \supset A_\gamma'$$
where
\begin{itemize}
    \item $A_i''\subset A_i'$ has all rows intervals with $|\co(\pi(A_i''))\setminus \pi(A_i'')|\le g_i |\pi(B)|$, $|A_i'\setminus A_i''|\le g_i|B|$, and $|V(A_i'')|\le \ell(e)$, with $\ell$ as in \Cref{polyapprox}.
    \item $A_{i+1}'\subset A_i''$ has all rows intervals, $\co(A_{i+1}')=\co(A_i'')$ (so in particular $|V(A_{i+1}')|\le \ell(e)$), property 2 holds for $A_{i+1}'$, and $|A_i''\setminus A_{i+1}'|\le g_i'|B|$.
    \item $|A'\setminus A'_\gamma|\le h|B|$.
\end{itemize}
Assume we can find such dependencies for $\gamma$, we will show we can find dependencies $g_{\gamma-1}'\ll g_{\gamma}\ll g_{\gamma}'\ll \epsilon_0,h/2$ to work for $\gamma+1$. By requiring the dependency $g_{\gamma-1}'\ll g_\gamma$ to be at least as strong as the existing $g_{\gamma-1}'\ll \epsilon_0,h/2$ dependencies, we can ensure the composite dependency is at least as strong as the existing $g_{\gamma-1}'\ll \epsilon_0,h/2$ dependencies.

We start by showing that we can choose the dependencies so that \Cref{firstit} can be applied to $A_\gamma'$ to produce the desired $A_\gamma''$. Note that
$|A'\setminus A_\gamma'|\le (g_0 + g_0'+\ldots + g_{\gamma-1}+g_{\gamma-1}')|B|\le 2\gamma g_{\gamma-1}'|B|$, so by requiring $g_{\gamma}'\le \frac{1}{2\gamma}\frac{\epsilon_0}{12}$, we can ensure that
$$|A_\gamma'|\ge \left(\frac{3\epsilon_0}{4}-\frac{\epsilon_0}{12}\right)|B|=\frac{2\epsilon_0}{3}|B|.$$
Next, because $|A'\setminus A'_\gamma|\le (2\gamma g_{\gamma-1}')|B|$, by \Cref{dkobs} we have 
$$d_k(A'_\gamma)\le (e+2^{k+1}\gamma g_{\gamma-1}')|B|\le (2^{k+1}\gamma+1)g_{\gamma-1}'|B|.$$
By requiring that  $g_{\gamma-1}'\le \frac{1}{2^{k+1}\gamma+1}g_\gamma$, the dependency $g_\gamma\ll h$ is at least as strong as the $f\ll h_4$ and $f\ll h_3$ dependencies from \Cref{firstit}, so we can apply \Cref{firstit} with $h_3=h_4=h$ and $f=g_\gamma$ to obtain the desired $A_\gamma''$ (noting $\ell(e)\ge \ell(g_\gamma)$).

Because we chose the dependency $n_{k,0}^{-1}\ll e$ to be strong enough to apply \Cref{overpulling} whenever $f_1,f_2\ge e$, to construct $A'_{\gamma+1}$ satisfying the second point we can apply \Cref{overpulling} to $A_\gamma''$ with $f_1=f_2=g_\gamma$, setting the $g_\gamma \ll g_\gamma'$ dependency to be $(2\nu+2)(g_\gamma+2^{k+1}g_\gamma) \le g_\gamma'$.

Finally, to satisfy the third point, note that by the first and second points we have $|A'\setminus A'_{\gamma+1}|\le (g_0+g_0'+\ldots+g_\gamma+g_\gamma')|B|\le 2(\gamma+1)g_\gamma'|B|$, so by requiring $g_\gamma'\le \frac{h}{2(\gamma+1)}$ we can ensure that $|A'\setminus A'_{\gamma+1}|\le h|B|$.

The end of the proof of \Cref{firstit} applied verbatim then yields the desired result.
\end{proof}

\subsubsection{$A_\star\subset A_+$ with $|V(A_\star)|$ and $|\co(\pi(A_\star))\setminus \pi(A_\star)|$ small and one further technical condition: Construction}
\label{AstarConstruction}

Before we proceed we need to introduce the following definition.

\begin{defn}\label{sij}
Given a simplex $\widetilde{T}\subset \mathbb{R}^k$ with vertices $x_0, \hdots, x_k$, construct inductively a family of translates of $\frac{1}{2^i}\widetilde{T}$ inside $\widetilde{T}$ as follows. Set 
\begin{align*}
\mathcal{S}_{i,0}(\widetilde{T})&:= \left\{\left(1-\frac{1}{2^i}\right)x_r+ \frac{1}{2^i}\widetilde{T} : 0\le r \le k \right\}\\
\mathcal{S}_{i,j+1}(\widetilde{T})&:= \left\{ \frac{\widetilde{S}_1+ \widetilde{S}_2}{2}: \widetilde{S}_1, \widetilde{S}_2 \in \mathcal{S}_{i,j} \right\}
\end{align*}
\end{defn}
We demonstrate these definitions with a graphic. For a simplex $\widetilde{T}\subset \mathbb{R}^2$, on the left one triangle from $\mathcal{S}_{2,0}(\widetilde{T})$ is shaded, and on the right all triangles from $\mathcal{S}_{2,1}(\widetilde{T})$ are shaded.
\begin{center}
\resizebox{5.77cm}{5.0cm}{
\begin{tikzpicture}
\draw (0,0) -- (4,4) -- (8,0) -- (0,0);
\draw (2,2)--(6,2)--(4,0)--(2,2);

\draw (3,3) -- (5,3);
\draw (3.5,3.5) -- (4.5,3.5);
\draw(3.75,3.75)--(4.25,3.75);

\draw (1,1)--(2,0);
\draw (0.5,0.5)--(1,0);
\draw (0.25,0.25)--(0.5,0);

\draw (7,1)--(6,0);
\draw (7.5,0.5)--(7,0);
\draw (7.75,0.25)--(7.5,0);

\draw[fill=black, opacity=0.2] (4,4)--(3,3)--(5,3)--cycle;
\end{tikzpicture}
}
\resizebox{5.77cm}{5.0cm}{
\begin{tikzpicture}
\draw (0,0) -- (4,4) -- (8,0) -- (0,0);
\draw[fill=black, opacity=0.2] (4,4)--(3,3)--(5,3)--cycle;
\draw[fill=black, opacity=0.2] (0,0)--(1,1)--(2,0)--cycle;
\draw[fill=black, opacity=0.2] (8,0)--(7,1)--(6,0)--cycle;

\draw[fill=black, opacity=0.2] (3,0)--(4,1)--(5,0)--cycle;
\draw[fill=black, opacity=0.2] (1.5,1.5)--(2.5,2.5)--(3.5,1.5)--cycle;
\draw[fill=black, opacity=0.2] (4.5,1.5)--(5.5,2.5)--(6.5,1.5)--cycle;
\end{tikzpicture}
}
\end{center}

An important fact about $\mathcal{S}_{i,j}(\widetilde{T})$ is that for fixed $i$ this family of translates is dense in the sense that, for any translate $\frac{1}{2^i}\widetilde{T}+y\subset \widetilde{T}$, there are $\frac{1}{2^i}\widetilde{T}+y_j\in \mathcal{S}_{i,j}(\widetilde{T})$ such that $y_j \to y$.

\begin{defn}
Given a simplex $\widetilde{T}\subset \mathbb{R}^k$ we define 
\begin{align*}
\mathcal{U}_{i,j}({\widetilde{T}}):=\{\vec{u}:\exists \widetilde{S}_1,\widetilde{S}_2\in \mathcal{S}_{i,j}\text{ with }\widetilde{S}_1+\vec{u}=\widetilde{S}_2\}.
\end{align*}
\end{defn}
Before proceeding, we remark that we will now need a future result, \Cref{coveringfamily}, to define certain constants $\mu_1$ and $\mu_2$ depending only on $k$. The proof is entirely self-contained, and while we could include the result and its proof at this point, we feel it is better to defer them.
\begin{defn}\label{VofT}
We define constants $\mu_1=\mu_1(k), \mu_2=\mu_2(k)$ as those produced by \Cref{coveringfamily}. 
Given a simplex $\widetilde{T}\subset \mathbb{R}^k$ we set $$\mathcal{W}_T:= \bigcup_{\vec{u}\in \mathcal{U}_{\mu_1, \mu_2}(\widetilde{T})} \mathcal{R}(\vec{u})\cup (-\mathcal{R}(\vec{u})),$$  where $\mathcal{R}(\vec{u})=\lfloor \vec{u} \rfloor + \{0\} \times \{0,1\}^{k-1} $ (here $\lfloor \vec{u} \rfloor:=(\lfloor \vec{u}_1\rfloor,\ldots, \lfloor \vec{u}_k\rfloor)$). This satisfies $|\mathcal{W}_T| \le \nu$ for $\nu=\nu(k)$ taken to be the constant $2^k|\mathcal{U}_{\mu_1,\mu_2}(\widetilde{T})|$, which is independent of the simplex $\widetilde{T}$. Note that $0\in\mathcal{U}_{\mu_1, \mu_2}(\widetilde{T})$ so $\{0\}\times \{0,1\}^{k-1}\subset \mathcal{W}_T$, and $\mathcal{W}_T=-\mathcal{W}_T$.\end{defn}
Throughout \Cref{1.6section}, we note that we will be taking $\widetilde{T}\subset \{0\}\times \mathbb{R}^{k-1}$.

We now fix functions $h_5(e),h_6(e)$ from $\mathbb{R}_{>0}\to \mathbb{R}_{>0}$ with $h_5,h_6\to 0$ as $e\to 0$ such that these functions realize the $\ll$ dependency between $e$ and the constants $h_5,h_6$ in \Cref{Astarreduction} (note that in \Cref{Astarreduction} there is a dependency $e\ll 1$ which says that $e$ is smaller than some absolute constant - above this constant we note that we can simply take $h_5,h_6$ to be an arbitrary constant. Alternatively, note that the conclusion of \Cref{Astarreduction} follows trivially when $h_5=h_6=1$ so we can take these functions for $e$ large). There exist $\ll$ dependencies $n_{k,0}^{-1}\ll\delta\ll\epsilon_0\ll 1$ that imply the $\ll$ dependencies $n_{k,0}^{-1}\ll e =\delta^{\frac{1}{20}-17c}\ll \epsilon_0\le 1$ needed to apply \Cref{Astarreduction} to $A_+$ taking $e=\delta^{\frac{1}{20}-17c}$, so we obtain a subset $A_\star\subset A_+$ with all rows intervals satisfying the following properties. By \eqref{AA+}, we have
\begin{align}\label{AAstar}|A\Delta A_\star|\le|A\Delta A_{+}|+|A_+\setminus A_\star|\le (\delta^{\frac{1}{20}-16c}+h_5(\delta^{\frac{1}{20}-17c}))|B|=:h_8(\delta)|B|,
\end{align}
and by \Cref{dkobs} and \Cref{obs:dkAplus}, $A_\star$ is reduced and
\begin{align}\label{dkAstar}d_k(A_{\star}) \le (\delta^{\frac{1}{20}-17c} + 2^kh_8(\delta) )|B| = :h_9(\delta)|B|.
\end{align} Also,
\begin{align}\label{copiAstar}|\co(\pi(A_{\star})) \setminus \pi(A_{\star})|&\le h_6(\delta^{\frac{1}{20}-17c})|\pi(B)|,\text{ and }\\
\label{Astarfewvert}|V(A_\star)|&\le H_7(\delta^{\frac{1}{20}-17c}).\end{align} Here, $h_6,h_8,h_9 \rightarrow 0$ as $\delta \rightarrow 0$.
\begin{obs}\label{pullover}
Finally, we have for every $\widetilde{T}\in\mathcal{T}^+(A_\star)\cup\mathcal{T}^{-}(A_\star)$, if $x\in T^o \setminus V_\pi(A_\star)$,  $y\in Y_{\mathcal{W}_T}(x)$, $z \in \pi(B)$ and $\vec{v}\in\{0\}\times\{0,1\}^{k-1}$ with $x+y=z+z+\vec{v}$ and $R_{x}+R_{y}, R_{z}+R_{z+\vec{v}}$ nonempty, then $(R_{x}+R_{y}) \cap (R_{z}+R_{z+\vec{v}}) \ne \emptyset$.
\end{obs}

\subsection{$A_\star$ is close to $\co(A_\star)$}
\label{Astarclosetoco}
In this section, we show that $|\co(A_\star)\setminus A_\star|= o(1)|B|$. It will be easy to show that $2^k|\co(A_\star)\setminus A_\star|\le |\co(A_\star+A_\star)\setminus (A_\star+A_\star)|+o(1)|B|$. This will follow from $|A_\star+A_\star|\le (2^k+o(1))|A_\star|$ and $(2^k+o(1))|\co(A_\star)|\le |\co(A_\star+A_\star)|$. Hence it will suffice to show that \begin{align}\label{vaguedescription2kck}|\co(A_\star+A_\star)\setminus (A_\star+A_\star)|\le (2^k-c_k')|\co(A_\star)\setminus A_\star|+o(1)|B|\end{align} for some constant $c_k'>0$. We now give a motivating outline.

For $\widetilde{T}\in \mathcal{T}^+(A_\star)\cup \mathcal{T}^-(A_\star)$, we will define functions $g_{T}^+,g_T^-:T\to [0,n_1]$ (actually we will need $[0,2n_1]$ for technical reasons), which encode the distances from the nonempty rows of $A_\star$ in $T$ to the upper and lower convex hulls of $A_\star$ respectively. Then we can estimate
\begin{align}\label{vaguedesc1}|\co(A_\star)\setminus A_\star|=o(1)|B|+\sum_{\ast \in \{+,-\}}\sum_{\widetilde{T}\in \mathcal{T}^\ast} \sum_{x\in T}g_T^\ast(x).\end{align}
Moreover, we will define functions $g^{+\square}_T,g^{-\square}_T$ as certain restricted infimum convolutions of $g^+_T$ and $g^-_T$ with themselves. These will encode the distance between the rows of a certain subset of $A_\star+A_\star$ (which we will guarantee to be intervals by \Cref{pullover}) to the upper and lower convex hulls of $A_\star+A_\star$ respectively. This subset accounts for almost all rows. Then we can similarly estimate
\begin{align}\label{vaguedesc2}|\co(A_\star+A_\star)\setminus (A_\star+A_\star)| \le o(1)|B|+\sum_{\ast \in \{+,-\}}\sum_{\widetilde{T}\in \mathcal{T}^\ast}\sum_{x'\in 2\widetilde{T}\cap \{0\}\times \mathbb{Z}^{k-1}} g^{\ast\square}_{T}(x').\end{align}
To prove the inequality, it will therefore suffice to show for every $\ast \in \{+,-\}$ and $\widetilde{T}\in \mathcal{T}^\ast$, given a function $g:T\to [0,2n_1]$ which is $0$ at the vertices of $T$, that
\begin{align}\label{vaguedesc3}\sum_{x\in T}g(x) \le o(1)|B|+ (2^k-c_k')\sum_{x'\in 2\widetilde{T}\cap \{0\}\times \mathbb{Z}^{k-1}}g^{\square}(x').\end{align}

In \Cref{Transitioningsubsection}, we properly define the functions $g^\ast_T(x)$ and $g^{\ast \square}_T(x)$ and show \eqref{vaguedesc1} in \Cref{gsimple} and \eqref{vaguedesc2} in \Cref{gsquare}, thus reducing the problem to showing \eqref{vaguedesc3}. 

In \Cref{infconvolutionsubsection}, we prove \eqref{vaguedesc3} in \Cref{squareprop}.

Finally, in \Cref{AstarclosetocoAstarsubsection}, we combine these results and conclude \eqref{vaguedescription2kck} in \Cref{AstarclosetocoAstarprop}.
\subsubsection{Transitioning from $A_\star$ to functions and their infimum convolutions}
\label{Transitioningsubsection}

We focus on the gaps in the $e_1$-direction between $A_\star$ and the convex hull of $A_\star$ via functions on $\pi(A_\star)$. 
\begin{notn}We denote by $V_\pi=V_\pi(A_\star)=\pi(V(A_\star))$ the projection of the vertices of $\cooo(A_{\star})$ under $\pi$ to $\{0\}\times \mathbb{Z}^{k-1}$. We denote the empty rows  by $E:=\co(\pi(A_\star))\setminus \pi(A_\star).$ Finally, we write $\mathcal{T}^+:=\mathcal{T}^+(A_\star)$ and $\mathcal{T}^-:=\mathcal{T}^-(A_\star)$.
\end{notn}

Recall that for a simplex $\widetilde{T}\subset \cooo(\pi(B))$, we set $T=\widetilde{T}\cap  \mathbb{Z}^{k}=\widetilde{T}\cap (\{0\}\times \mathbb{Z}^{k-1})$.

\begin{defn}
\label{gplusgminusdefn}
Letting $\Psi_{A_\star}^+,\Psi_{A_\star}^-:\cooo(\pi(A_\star))\to \mathbb{R}$ be the upper and lower convex hull functions of $A_\star$ in the $e_1$-direction respectively, we define for $\widetilde{T}^+\in \mathcal{T}^+$ and $\widetilde{T}^-\in \mathcal{T}^-$ the functions
$$g^+_{T^+}:T^+\to [0,2n_1],\qquad g^-_{T^-}:T^-\to [0,2n_1]$$
according to the formulas
$$g^+_{T^+}(x)=\begin{cases}\Psi_{A^\star}^+(x)-\max R_x & \text{if }x\in  V(T^+)\text{ or }x\not\in V_\pi\cup E\\
2n_1 & \text{otherwise}.\end{cases},$$
and 
$$g^-_{T^-}(x)=\begin{cases}\min R_x-\Psi_{A^\star}^-(x) & \text{if }x\in  V(T^-)\text{ or }x\not\in V_\pi\cup E\\
2n_1 & \text{otherwise.}\end{cases}.$$
\end{defn}
\begin{rmk}
\label{gplusgminusrmk}
For $x\in V(T^\ast)$ or $x\not\in V_{\pi}\cup E$, $g^\ast_{T^\ast}(x)\in [0,n_1-1]$ is the distance in the $e_1$-direction from the row $R_x$ to the upper convex hull of $A_\star$ for $\ast=+$ and lower convex hull of $A_\star$ for $\ast=-$, and we always have $g^\ast_{T^\ast}(x)\le 2n_1$. 
In particular, for $x\in V(T^\ast)$ we note that $g^\ast_{T^\ast}(x)=0$.
\end{rmk}

\begin{obs}\label{gsimple}
There exist $\ll$ dependencies $n_{k,0}^{-1}\ll\delta\ll\epsilon_0\ll 1$ such that the following holds.
We have the following estimate for some function $h_{10}(\delta)\to 0$ as $\delta\to 0$:
\begin{align*}
\left(\sum_{\ast \in \{+,-\}}\sum_{\widetilde{T}^\ast\in\mathcal{T}^\ast}\sum_{x\in T^\ast}g^\ast_{T^\ast}(x)\right)-|\co(A_\star)\setminus A_\star|\le h_{10}(\delta)|B|.\end{align*}
\end{obs}
\begin{proof}
Writing $|\co(A_\star)\setminus A_\star|=\sum_{x\in \co(\pi(A_\star))}|(\co(A_\star)\setminus A_\star)\cap \pi^{-1}(x)|$, we upper bound the contribution separately on the left hand side for each $x\in \co(\pi(A_\star))$.
\begin{itemize}\item We estimate the contribution of $x\in \partial\widetilde{T}^+$ for some $\widetilde{T}^+\in \mathcal{T}^+$ or with $x\in\partial\widetilde{T}^-$ for some $\widetilde{T}^-\in \mathcal{T}^-$. There are at most $\binom{|V_\pi|}{k}$ simplices, each with $k$ facets, with each facet having at most $n_{k,0}^{-1}|\pi(B)|$ integral points by \Cref{hypboxsmall} applied to the box $\pi(B)$, each of which can in turn contribute $2n_1+2n_1$ to the left hand side.
\item We estimate the contribution of $x \in V_\pi\cup E$ that lie in the interior of at most one simplex in $\mathcal{T}^+$ and at most one simplex in $\mathcal{T}^-$. There are at most $|E|+|V_\pi|$ of these points, each of which can in turn contribute $2n_1+2n_1$ to the left hand side.
\item Finally, each remaining $x$ lies in a unique simplex of $\widetilde{T}^+\in\mathcal{T}^+$ and $\widetilde{T}^-\in\mathcal{T}^-$, and $x\not \in V_\pi\cup E$. For such $x$, we have $|(\co(A_\star)\setminus A_\star)\cap R_x|=\lfloor g^+_{T^+}(x)\rfloor + \lfloor g^-_{T^-}(x) \rfloor$. This discrepancy with $g^+_{T^+}(x)+g^-_{T^-}(x)$ is crudely bounded by $2|\pi(B)|$ for each of the at most $\binom{|V_\pi|}{k}$ simplices. 
\end{itemize}
Combining these errors, and noting that $|E|$ and $|V_\pi|$ are bounded by \eqref{copiAstar} and \eqref{Astarfewvert}, we conclude by choosing $h_{10}(\delta)$ and $n_{k,0}^{-1}\ll \delta$ so that
$$4k\binom{|V_\pi|}{k}n_{k,0}^{-1}|B|+(|E|+|V_\pi|)4n_1+2\binom{|V_\pi|}{k}|\pi(B)|\le h_{10}(\delta)|B|.$$
\end{proof}

Recall in \Cref{YofT}, we introduced for a simplex $\widetilde{T}\subset \{0\}\times \mathbb{R}^{k-1}$ with integral vertices and a subset $\mathcal{W}\subset \{0\}\times \mathbb{Z}^{k-1}$ the notation $Y_{\mathcal{W}}(x)=((x+\mathcal{W})\cap T)\cup V(T)$. We now define a restricted infimum convolution with respect to  $\mathcal{W}$. 

\begin{defn}\label{infconv}
Given a simplex $\widetilde{T}\subset \{0\}\times \mathbb{R}^{k-1}$ with integral vertices, a subset $\mathcal{W}\subset \{0\}\times \mathbb{Z}^{k-1}$, and a function $g : T \rightarrow \mathbb{R}_{\ge 0}$, we define the restricted infimum convolution $g^{\square}_{\mathcal{W}}:T+T\to \mathbb{R}_{\ge 0}$ by
$$g^{\square}_{\mathcal{W}}(x)=\begin{cases}\min\{g(x_1)+g(x_2)\}& \text{over all }x_1+x_2=x \text{ with } x_1\in T^{\circ}\text{ and }x_2 \in Y_{\mathcal{W}}(x_1)\\
0& \text{no such $x_1,x_2$ exist.}
\end{cases}$$
\end{defn}
\begin{notn}
For $g=g^\ast_{T^\ast}$ for some $\ast\in \{+,-\}$, we will always take $\mathcal{W}=\mathcal{W}_T$ as defined in \Cref{VofT} in the infimum convolution $(g^{\ast}_{T^{\ast}})^{\square}_{\mathcal{W}}=g^{\ast\square}_{T^\ast,\mathcal{W}}$. We will always omit the subscript $\mathcal{W}$, writing $g^{\ast\square}_{T^\ast}$ instead.
\end{notn}

\begin{obs}\label{gsquare}
There exist $\ll$ dependencies $n_{k,0}^{-1}\ll\delta\ll\epsilon_0\ll 1$ such that the following holds.
We have the following estimate for some function $h_{11}(\delta)\to 0$ as $\delta\to 0$:
\begin{align*}|\co(A_\star+A_\star)\setminus (A_\star+A_\star)|-\sum_{\ast\in \{+,-\}} \sum_{\widetilde{T}^\ast\in\mathcal{T}^\ast}\sum_{x'\in (T^\ast+T^\ast)}g_{T^\ast}^{\ast\square}(x')\le h_{11}(\delta)|B|.\end{align*}
\end{obs}
\begin{proof}
Writing $|\co(A_\star+A_\star)\setminus (A_\star+A_\star)|=\sum_{x'\in \co(\pi(A_\star+A_\star))}|(\co(A_\star+A_\star)\setminus (A_\star+A_\star))\cap \pi^{-1}(x')|$, we upper bound the contribution for each $x'\in \co(\pi(A_\star+A_\star))$, which we express uniquely as $x'=x+x+\vec{v}$ with $\vec{v}\in\{0\}\times\{0,1\}^{k-1}$, $x \in\{0\}\times\mathbb{Z}^{k-1}$.
\begin{clm}
Fix $x'=x+x+\vec{v}$ as above. Suppose for each $\ast \in \{+,-\}$ there exists $\widetilde{T}^\ast\in \mathcal{T}^\ast$ with $x,x+\vec{v}\in T^{\ast o}\setminus (V_\pi\cup E)$. Then
$$|\co(A_\star+A_\star)\cap \pi^{-1}(x')|\le g_{T^+}^{+\square}(x')+g_{T^-}^{-\square}(x').$$ 
\end{clm}
\begin{proof}
Note that we can write $x'=x+(x+\vec{v})$, with $x\in T^{\ast o}$ and $x+\vec{v}\in Y_{\mathcal{W}_{T^\ast}}(x)$, since $\vec{v}\in \mathcal{W}_{T^{\ast}}$. Hence by \Cref{infconv} there exists $x_1^*\in T^{\ast o}$ and $x_2^*\in Y_{\mathcal{W}_{T^\ast}}(x_1^\ast)$ so that $x_1^\ast+x_2^\ast=x'$ and $g^{\ast\square}_{T^\ast}(x')=g_{T^\ast}^\ast(x_1^\ast)+g_{T^\ast}^\ast(x_2^\ast)$.  If either $x_1^\ast$ or $x_2^\ast$ are in $(V_\pi\cup E)\setminus V(T^\ast)$, the corresponding term on the right hand side is at least $2n_1$ and the inequality is trivially true. Hence, we may assume $x_1^*,x_2^*\not \in (V_\pi\cup E)\setminus V(T^\ast)$. In particular, we have $x_1^*\in T^{\ast\circ}\setminus V_\pi$ and both $x_1^\ast,x_2^\ast\not\in E$. Let $\Psi_{A_\star+A_\star}^+,\Psi_{A_\star+A_\star}^-:\cooo(\pi(A_\star+A_\star))\to \mathbb{R}$ be the upper and lower convex hull function on $A_\star+A_\star$ in the $e_1$-direction respectively. Note that $\Psi^\star_{A_\star+A_\star}(x')=2\Psi^\star_{A_\star}(\frac{x'}{2})$ for any $x'$, and for $x,y\in \widetilde{T}$ we have $\frac{1}{2}(\Psi^\star_{A_\star}(x)+\Psi^\star_{A_\star}(y))=\Psi^\star_{A_\star}(\frac{x+y}{2})$. Hence we have
\begin{align*}
    \Psi^+_{A_\star+A_\star}(x')-\max (R_{x_1^+}+R_{x_2^+})&=\Psi_{A_\star}(x_1^+)+\Psi_{A_\star}(x_2^+)-\max R_{x_1^+}-\max R_{x_2^+}\\
    &=g_{T^+}^+(x_1^+)+g_{T^+}^+(x_2^+)=g^{+\square}_{T^+}(x'),
\end{align*}
since $x_1^+,x_2^+\in \widetilde{T}^+$, and
\begin{align*}
        \min (R_{x_1^-}+R_{x_2^-})-\Psi^-_{A_\star+A_\star}(x')&=\min R_{x_1^-}+\min R_{x_2^-}-\Psi_{A_\star}(x_1^-)-\Psi_{A_\star}(x_2^-)\\
    &=g_{T^-}^-(x_1^-)+g_{T^-}^-(x_2^-)=g^{-\square}_{T^-}(x'),
\end{align*}
since $x_1^-,x_2^-\in \widetilde{T}^-$.

 By \Cref{pullover}, as $x_1^*\in T^\circ\setminus V_\pi$, $x_2^\ast\in Y_{\mathcal{W}_{T^\ast}}(x_1^\ast)$, and $x_1^\ast,x_2^\ast,x,x+\vec{v}\not\in E$, the intervals $R_{x_1^+}+R_{x_2^+}$ and $R_{x_1^-}+R_{x_2^-}$ both overlap $R_x+R_{x+\vec{v}}$, so $$I_{x'}:=(R_{x_1^-}+R_{x_2^-})\cup (R_{x}+R_{x+\vec{v}})\cup (R_{x_1^+}+R_{x_2^+})$$ is an interval. Therefore \begin{align*}&|(\co(A_\star+A_\star)\setminus (A_\star+A_\star))\cap \pi^{-1}(x')|\\\le &|(\co(A_\star+A_\star)\cap \pi^{-1}(x'))\setminus I_{x'}|\\
 \le&\lfloor \Psi^+_{A_\star+A_\star}(x')\rfloor-\max (R_{x_1^+}+R_{x_2^+})+\min (R_{x_1^-}+R_{x_2^-})-\lceil \Psi^-_{A_\star+A_\star}(x')\rceil\\
 \le& g^{+\square}_{T^+}(x')+g^{-\square}_{T^-}(x').\end{align*}
 \end{proof}

 Returning to the proof of \Cref{gsquare}, we have the following estimates. Recall that we uniquely write $x'=x+x+\vec{v}$ with $x'\in \co(\pi(A_\star+A_\star))$, $x\in \{0\}\times \mathbb{Z}^{k-1}$ and $\vec{v}\in \{0\}\times \{0,1\}^{k-1}$.
\begin{itemize}
\item We estimate the contribution of $x,\vec{v}$ such that, for $\ast\in \{+,-\}$ there is some $\widetilde{T}^\ast \in \mathcal{T}^\ast$ so that $x$ lies in a facet of $\widetilde{T}^{\ast}$.  There are at most $\binom{|V_\pi|}{k}$ simplices, each with $k$ facets, each facet yields at most $n_{k,0}^{-1}|\pi(B)|$ locations for $x$ by \Cref{hypboxsmall}, and for each such $x$ there are $2^{k-1}$ possible choices for $v$, and each such $x,\vec{v}$ contributes at most $2n_1$ to the left hand side.
\item We estimate the contribution of $x,\vec{v}$ such that, for some $\ast \in \{+,-\}$, there is no $\widetilde{T}^\ast\in \mathcal{T}^\ast$ with $x,x+\vec{v}\in T^{\ast\circ}$. For the simplex $\widetilde{T}^\ast$ containing $\frac{1}{2}(x+(x+\vec{v}))=\frac{1}{2}x'\in \pi(\cooo(A_\star))$, there is a hyperplane $\widetilde{H}$ containing a facet of $\widetilde{T}^\ast$ that separates (or contains one of) $x$ and $x+\vec{v}$. For each $\vec{v}$, there are at most $\binom{|V_\pi|}{k}$ simplices, each with $k$ facets, and each facet separating (or containing) at most $2$ such $x,x+\vec{v}$ pairs on each of the at most $(k-1)n_{k,0}^{-1}|\pi(B)|$ lines in direction $\vec{v}$ intersecting $\pi(B)$, and each such $x,\vec{v}$ contributes at most $2n_1$ to the left hand side.

\item We now estimate the contribution for those $x,\vec{v}$ such that one of $x,x+\vec{v}$ lies in $V_\pi\cup E$. There are $2^{k-1}$ choices of $v$, and for each of these choices there are at most $2(|V_\pi|+|E|)$ such values of $x$, each of which contributes at most $2n_1$ to the left hand side.

\item The remaining $x,\vec{v}$ have $x,x+\vec{v}\in T^{+o}\setminus (V_\pi\cup E), T^{-o}\setminus (V_\pi\cup E)$ for unique simplices $\widetilde{T}^+\in \mathcal{T}^+$ and $\widetilde{T}^-\in \mathcal{T}^-$. But the above claim shows that the contribution of such $x'$ is non-positive.
\end{itemize}
Combining these errors, and noting that $|E|$ and $|V_\pi|$ are bounded by \eqref{copiAstar} and \eqref{Astarfewvert}, we conclude by taking $h_{11}(\delta)$ and $n_{k,0}^{-1}\ll \delta$ so that
\begin{align*}2^{k}k\binom{|V_\pi|}{k}n_{k,0}^{-1}|B|+2^{k+1}\binom{|V_\pi|}{k}k(k-1) n_{k,0}^{-1}|B|+2^{k+1}(|V_\pi|+|E|)n_1\le h_{11}(\delta)|B|.\end{align*}
\end{proof}

\subsubsection{Infimum convolution of functions}
\label{infconvolutionsubsection}
In this section we prove a general result about the infimum convolution (see \Cref{infconv}) of functions, related to the fact that small doubling implies being close to the convex hull.

\begin{prop}\label{squareprop}
There exist constants $c_k',c_k''>0$ such that the following is true.  Let $T\subset \pi(B)$ be a discrete simplex with integral vertices $x_0,\ldots,x_{k-1}\in \{0\}\times\mathbb{Z}^{k-1}$, and let $g:T\to [0,2n_1]$ with $g(x_i)=0$ for all $i$. Then
$$\sum_{x' \in T+T}g_{\mathcal{W}_T}^{\square}(x')\le (2^k-c_k')\sum_{x \in T}g(x)+c_k''\min\{n_i\}^{-1}|B|.$$
\end{prop}
We omit the subscript $\mathcal{W}_T$ from now on. Throughout the entire proof we shall consider the sets $\mathcal{S}_{i,j}=\mathcal{S}_{i,j}(\widetilde{T})$ in \Cref{sij} with parameters $i,j$ bounded above by $\mu_1,\mu_2$, respectively.

\begin{defn}
For a continuous or discrete subset $X \subset \widetilde{T}$, we define $g(X)= \sum_{x \in X\cap T} g(x)$, and for a continuous or discrete subset $Y\subset \widetilde{T}+\widetilde{T}$, define $g^{\square}(Y)=\sum_{x\in Y\cap (T+T)}g^{\square}(x)$.
\end{defn}

\begin{defn}
For a continuous or discrete subset $X \subset \widetilde{T}$, we define $d'_k(X)=-g^{\square}(X+X)+2^kg(X)$.
\end{defn}

\begin{obs}\label{dk'obs}
For a polytope $\widetilde{P} \subset \widetilde{T}$, we have $d'_k(\widetilde{P}) \ge -2^{2k+5}\min\{n_i\}^{-1}|B|$. In particular, with $P=\widetilde{P}\cap (\{0\}\times\mathbb{Z}^{k-1})$ we have $d'_k(P) \ge -2^{2k+5}\min\{n_i\}^{-1}|B|$. More generally, the same conclusion holds for any region $\widetilde{P}\subset \widetilde{T}$ defined as the intersection of open and closed half-spaces.
\end{obs}

\begin{proof}
Note that if $\widetilde{P}\subset \widetilde{T}$ is defined as the intersection of open and closed half-spaces, then we can perturb the open half-spaces to closed ones without changing the lattice points in $\widetilde{P}$ or $\widetilde{P}+\widetilde{P}$, so we may assume that $\widetilde{P}$ is a polytope.
Recall that $\{0\}\times \{0,1\}^{k-1} \subset \mathcal{W}_T$. Note that if we write $z\in (\widetilde{P}+\widetilde{P})\cap T=2\widetilde{P}\cap T$ as $z=x+(x+\vec{v})$ with $\vec{v}\in \{0\}\times \{0,1\}^{k-1}$, then as $\widetilde{P}$ is convex, either $x,x+\vec{v}\in P^o$ or the segment $[x,x+\vec{v}]$ intersects $\partial\widetilde{P}$. Therefore with $\vec{v}$ ranging over $\{0\}\times\{0,1\}^{k-1}$ we have
\begin{align}
    d'_k(\widetilde{P})&= 2^kg(\widetilde{P})-g^{\square}(\widetilde{P}+\widetilde{P}) \nonumber\\
    &\ge 2^k\sum_{x\in P} g(x) - \sum_{\vec{v}}\sum_{x,x+\vec{v}\in P^o}g^{\square}(x+x+\vec{v}) - \sum_{\vec{v}}\sum_{[x,x+\vec{v}]\cap\partial \widetilde{P}\ne \emptyset}g^{\square}(x+x+\vec{v})\nonumber\\\label{3.44line3}
    &\ge 2^k\sum_{x\in P} g(x) - \sum_{\vec{v}}\sum_{x,x+\vec{v} \in P^o  }(g(x)+g(x+\vec{v}))- \sum_{\vec{v}}\sum_{[x,x+\vec{v}]\cap\partial \widetilde{P}\ne \emptyset}4n_1\\
    &\ge-\sum_{\vec{v}}\sum_{[x,x+\vec{v}]\cap\partial \widetilde{P}\ne \emptyset}4n_1 \nonumber\\
    &\ge -4n_1 \sum_{\vec{v}}|((\widetilde{P}+[0,v])\setminus (\widetilde{P}\cap (\widetilde{P}+v))^{\circ}))\cap \{0\}\times \mathbb{Z}^{k-1}|\nonumber\\
    &\ge -4n_1\sum_{\vec{v}}(|\widetilde{P}+[0,v]|-|\widetilde{P}\cap (\widetilde{P}+v)|+2\cdot 2(k-1)k \min\{n_i\}^{-1} |\pi(B)|)\label{3.44line6}\\
    &\ge -4n_1\sum_{\vec{v}}(|\vec{v}||\partial \widetilde{P}|+2\cdot 2(k-1)k \min\{n_i\}^{-1})\nonumber\\
    &\ge -4n_12^{k-1}(\sqrt{k-1}\cdot2(k-1)\min\{n_i\}^{-1} |\pi(B)|+2\cdot 2(k-1)k \min\{n_i\}^{-1} |\pi(B)|)\label{3.44line8}\\
    &\ge -2^{2k+5}\min\{n_i\}^{-1}|B|.\nonumber
\end{align}

In \eqref{3.44line3} we have used the fact that $\vec{v}\in \mathcal{W}_T$ so $x+\vec{v}\in Y_{\mathcal{W}_T}(x)$ by  definition, and the trivial bounds $g^{\square}\le 2\max g \le 4n_1$. In \eqref{3.44line6} we have used \Cref{contdisc} to upper bound $$|((\widetilde{P}+[0,v])\setminus (\widetilde{P}\cap (\widetilde{P}+v))^{\circ}))\cap \{0\}\times \mathbb{Z}^{k-1}|\le |(\widetilde{P}+[0,v])\cap \{0\}\times \mathbb{Z}^{k-1}|-|(\widetilde{P}\cap (\widetilde{P}+v)) \cap \{0\}\times \mathbb{Z}^{k-1}|,$$ and in \eqref{3.44line8} the facts that $|\vec{v}|\le k-1$ and $|\partial \widetilde{P}|\le |\partial \widetilde{\pi(B)}|\le 2(k-1)\min\{n_i\}^{-1}|\pi(B)|$.
\end{proof}

\begin{lem}
For a vertex $x\in V(\widetilde{T})$ and simplices $\widetilde{S}=(1-2^{-i})x+2^{-i}\widetilde{T} \in \mathcal{S}_{i,0}$ and   $\widetilde{S}'=\frac{1}{2}(x+\widetilde{S}) \in \mathcal{S}_{i+1,0}$, we have
$$g(S')\leq 2^{-k} g(S)+2^{-k}d_k'(T) +2^{k+6}\min\{n_i\}^{-1}|B|.$$
\end{lem}

\begin{proof}
By \Cref{dk'obs} we have (letting $\widetilde{S}'^c$ be the complement of $\widetilde{S}'$ inside $\widetilde{T}$) that
\begin{align}
    \nonumber d_k'(T)\ge d_k'(\widetilde{T})&=2^k g(\widetilde{T})-g^\square(\widetilde{T}+\widetilde{T})\\\nonumber
    &\ge 2^k g(\widetilde{S}')-g^\square(\widetilde{S}'+\widetilde{S}') + 2^k g(\widetilde{S}'^c)-g^\square(\widetilde{S}'^c+\widetilde{S}'^c)\\\nonumber
    &= 2^k g(\widetilde{S}')-g^\square(\widetilde{S}'+\widetilde{S}') + d_k'(\widetilde{S}'^c)\\\nonumber
    &\ge  2^k g(\widetilde{S}')-g^\square(\widetilde{S}'+\widetilde{S}')-2^{2k+5}\min\{n_i\}^{-1}|B|\\\nonumber
    &\ge  2^k g(\widetilde{S}')-g^\square(\widetilde{S}^{\circ}+x)-g^{\square}(\partial\widetilde{S}+x)-2^{2k+5}\min\{n_i\}^{-1}|B|\\\label{3.49line6}
    &\ge 2^k g(\widetilde{S}')-g(\widetilde{S}^{\circ})-k(4n_1)\min\{n_i\}^{-1}|\pi(B)|-2^{2k+5}\min\{n_i\}^{-1}|B|\\\nonumber
    &\ge2^k g(S')-g(S)-2^{2k+6}\min\{n_i\}^{-1}|B|,
\end{align}
where in \eqref{3.49line6} we used that $x\in V(\widetilde{T})\subseteq Y_{\mathcal{W}_T}(y)$ for all $y\in \widetilde{T}^{\circ}$,  \Cref{hypboxsmall} to estimate the number of lattice points on each facet of $\widetilde{S}$, and the fact that $\max g^{\square}\le 4n_1$.

\end{proof}

\begin{cor}
For $\widetilde{S}'\in\mathcal{S}_{i,0}$, we have 
\begin{align*}g(S')&\leq 2^{-ik} g(T)+\max\left(0,\frac{1}{2^k-1} d_k'(T)\right)+ \frac{2^k}{2^k-1}2^{k+6}\min\{n_i\}^{-1}|B|\\&\le 2^{-ik} g(T)+\frac{1}{2^k-1} d_k'(T)+2^{2k+6}\min\{n_i\}^{-1}|B|.\end{align*}
\end{cor}

\begin{lem}\label{averageg}
Let $i\leq \mu_1$ and $j\leq \mu_2-1$.
For $\widetilde{S}'_1,\widetilde{S}'_2\in\mathcal{S}_{i,j}$ and $\widetilde{S}''=\frac12(\widetilde{S}'_1+\widetilde{S}'_2)\in\mathcal{S}_{i,j+1}$, we have
$$g(S'')\leq \frac12( g(S'_1)+g(S'_2))+2^{-k}d_k'(T)+ 2^{5k}\min\{n_i\}^{-1}|B|.$$
\end{lem}
\begin{proof}
Let $\vec{u}$ be the vector such that $\widetilde{S}'_1+\vec{u}=\widetilde{S}'_2$. Then by \Cref{VofT}, we have  $\mathcal{R}(\vec{u})=\lfloor \vec{u} \rfloor +\{0\}\times\{0,1\}^{k-1}\subset \mathcal{W}_{T}$.
Let $$\widetilde{S}''^{c}=\widetilde{P}_1\sqcup\dots\sqcup \widetilde{P}_{k}$$ be a partition into convex regions, the intersection of open and closed half-spaces. Indeed this can be obtained by taking the defining equations $x \cdot \vec{c}_i\le \vec{b}_i$ for $1 \le i \le k$ of $\widetilde{S}''$, and defining $\widetilde{P}_j$ by setting $x \cdot \vec{c}_i\le\vec{b}_i$ for $1 \le i <j$ and $x \cdot \vec{w}_j > \vec{b}_j$ inside $\widetilde{T}$. Hence, by \Cref{dk'obs} we find
\begin{align*}
    d_k'(T)\ge d_k'(\widetilde{T})=2^k g(\widetilde{T})-g^\square(\widetilde{T}+\widetilde{T})
    &=2^kg(\widetilde{S}'')-g^\square(\widetilde{S}''+\widetilde{S}'')+ \sum_j (2^kg( \widetilde{P}_{j})-g^\square( \widetilde{P}_{j}+\widetilde{P}_{j}))\\
    &=2^kg(\widetilde{S}'')-g^\square(\widetilde{S}_1'+\widetilde{S}_2')+ \sum_jd_k'(\widetilde{P}_{j})\\
    &\ge 2^kg(\widetilde{S}'')-g^\square(\widetilde{S}_1'+\widetilde{S}_2')  -k2^{2k+5}\min\{n_i\}^{-1}|B|.
\end{align*}

Note that every point $x' \in (\widetilde{S}_1'+\widetilde{S}_2') \cap \{0\}\times\mathbb{Z}^{k-1}$ can be written uniquely as $x'=x+x+\vec{w}$ for some $\vec{w}\in \mathcal{R}(\vec{u})$ (this is true in fact for every $x' \in \{0\}\times \mathbb{Z}^{k-1}$), and for this $\vec{w}$ (in fact for any $\vec{w}\in \mathcal{R}(u)$) we have $\vec{w}-\vec{u}\in\{0\}\times (-1,1]^{k-1}$. We have $\vec{u}\in \{0\}\times \prod_{i=2}^k[-n_i+1,n_i-1]$ and $x'\in (\pi(\widetilde{B})+\pi(\widetilde{B}))\cap \{0\}\times \mathbb{Z}^{k-1}=\{0\}\times \prod_{i=2}^k \{2,\ldots,2n_i\}$, so $$x=\left\lfloor \frac{x'-\lfloor\vec{u}\rfloor}{2}\right\rfloor\in B':=\{0\}\times\prod_{i=2}^k \left\{\left\lfloor\frac{3-n_{i}}{2}\right\rfloor,\ldots,\left\lfloor\frac{3n_i-1}{2}\right\rfloor\right\},$$ where $B'$ is a translate of $B(2n_2-1,\ldots,2n_k-1)$. Also, as the midpoint $\frac{1}{2}(x+(x+(\vec{w}-\vec{u})))$ lies in $\widetilde{S}'_1$ (as it is equal to $\frac{1}{2}(x'-\vec{u})$), either $x\in \widetilde{S}_1'$ and $x+(\vec{w}-\vec{u}) \in \widetilde{S}_1'$ (equivalently $x+\vec{w}\in \widetilde{S}_2'$), or $x$ and $x+(\vec{w}-\vec{u})$  are separated by some hyperplane $\widetilde{H}$ containing one of the $k$ facets of $\widetilde{S}_1'$.

Given $\vec{w}\in \mathcal{R}(\vec{u})$ and a hyperplane $\widetilde{H}$, there are at most $2^{2k}\min\{n_i\}^{-1}|\pi(B)|$ many choices of $x \in B'$ with $x,x+(\vec{w}-\vec{u})$ separated by $\widetilde{H}$. Indeed, set $\widetilde{G}_{\vec{w},\widetilde{H}}$ to be the convex region of the box $\widetilde{B'}$ between the hyperplanes $\widetilde{H}$ and $\widetilde{H}-(\vec{w}-\vec{u})$. Note that $$|\widetilde{G}_{\vec{w},\widetilde{H}}|\le |\vec{w}-\vec{u}|\cdot|\partial \widetilde{B'}|\le (k-1)^{\frac{1}{2}} 2(k-1)2^{k-2}\min\{n_i\}^{-1}|\pi(B)|\le 2^{2k-1}\min\{n_i\}^{-1}|\pi(B)|.$$ By \Cref{contdisc} applied to $B'$,
\begin{align}\label{discGsize}|\widetilde{G}_{\vec{w},\widetilde{H}}\cap (\{0\}\times \mathbb{Z}^{k-1})|\le (2^{2k-1}  + 2(k-1)k2^{k-2})\min\{n_i\}^{-1}|\pi(B)| \le 2^{2k}\min\{n_i\}^{-1}|\pi(B)|.\end{align}

From the above discussion, if $x+x+\vec{w}\in \widetilde{S}_1'+\widetilde{S}_2'$, and either $x\not \in \widetilde{S}_1'$ or $x+\vec{w}\not\in \widetilde{S}_2'$, then $x\in \widetilde{G}_{\vec{w},\widetilde{H}}$ for some $\widetilde{H}$ containing a facet of $\widetilde{S_1}'$. Hence from \eqref{discGsize} (taking $\vec{w}\in \mathcal{R}(\vec{u})$ and $x\in \{0\}\times\mathbb{Z}^{k-1}$) we deduce $$\sum_{\vec{w}}\sum_{\substack{x+x+\vec{w} \in \widetilde{S}_1'+\widetilde{S}_2'  \\ x \not\in \widetilde{S}_1' \text{ or }x+\vec{w} \not\in \widetilde{S}_2' }} g^{\square}(x+x+\vec{w}) \le 2^{k-1}k2^{2k}\min\{n_i\}^{-1}|\pi(B)|\max g^{\square} \le k2^{3k+1}\min\{n_i\}^{-1}|B|.$$
Also, as $\mathcal{R}(u) \subset \mathcal{W}_T$ and $\max g^{\square}\le 4n_1$, we have \begin{align}\sum_{\vec{w}}\sum_{\substack{x+x+\vec{w} \in \widetilde{S}_1'+\widetilde{S}_2'  \\ x \in \widetilde{S}_1' \text{ and }x+\vec{w} \in \widetilde{S}_2' }} g^{\square}(x+x+\vec{w}) &\le \sum_{\vec{w}}\left(\sum_{\substack{x+x+\vec{w} \in \widetilde{S}_1'+\widetilde{S}_2'\nonumber  \\ x \in (\widetilde{S}_1')^{\circ} \text{ and }x+\vec{w} \in \widetilde{S}_2' }} (g(x)+g(x+\vec{w}))+\sum_{x\in \partial \widetilde{S}_1'}4n_1\right)\\ &\le 2^{k-1}(g(\widetilde{S}_1')+ g(\widetilde{S}_2'))+2^{k+1}k\min\{n_i\}^{-1}|B|\label{3.51},\end{align}
where in \eqref{3.51} we used \Cref{hypboxsmall} on each of the facets of $\widetilde{S}_1'$. Putting this all together,
\begin{align*}
    d_k'(T)\ge d_k'(\widetilde{T})\ge& 2^kg(\widetilde{S}'')-k2^{2k+5}\min\{n_i\}^{-1}|B|\\&-\sum_{\vec{w}}\sum_{\substack{x+x+\vec{w} \in \widetilde{S}_1'+\widetilde{S}_2'  \\ x \in \widetilde{S}_1' \text{ and }x+\vec{w} \in \widetilde{S}_2' }} g^{\square}(x+x+\vec{w})- \sum_{\vec{w}}\sum_{\substack{x+x+\vec{w} \in \widetilde{S}_1'+\widetilde{S}_2'  \\ x \not\in \widetilde{S}_1' \text{ or }x+\vec{w} \not\in \widetilde{S}_2' }} g^{\square}(x+x+\vec{w})\\
    \ge& 2^{k}g(\widetilde{S}'')-2^{k-1}(g(\widetilde{S}_1')+ g(\widetilde{S}_2'))-(2^{k+1}k+k2^{2k+5}+k2^{3k+1})\min\{n_i\}^{-1}|B|\\
    \ge& 2^{k}g(S'')-2^{k-1}(g(S_1')+ g(S_2'))-2^{6k}\min\{n_i\}^{-1}|B|.
\end{align*}

\end{proof}

\begin{cor}\label{mu1mu2}
For $\widetilde{S}'\in\mathcal{S}_{\mu_1,\mu_2}$ we have
$$g(S')\leq 2^{-\mu_1 k} g(T)+\left(\frac{1}{2^k-1}+ \mu_22^{-k}\right)d_k'(T)+(2^{2k+6}+\mu_22^{5k})\min\{n_i\}^{-1}|B|.$$
\end{cor}
Finally, before we prove \Cref{squareprop}, we prove the following result \Cref{coveringfamily} which as mentioned before constructs the constants $\mu_1,\mu_2$.
\begin{prop}\label{coveringfamily}
Let $\widetilde{T}\subset \mathbb{R}^{k-1}$ be a simplex. Then
 there exist $\mu_1=\mu_1(k)$ and $\mu_2=\mu_2(k)$ and a family $\mathcal{F}\subset \mathcal{S}_{\mu_1,\mu_2}(\widetilde{T})$ such that $\widetilde{T}\subset \bigcup_{\widetilde{S}\in \mathcal{F}}\widetilde{S}$ and $\sum_{\widetilde{S}\in \mathcal{F}} |\widetilde{S}|\leq 2^{\mu_1-1}|\widetilde{T}|$, i.e. $|\mathcal{F}|\leq 2^{\mu_1k-1}$.
\end{prop}
\begin{proof}
We follow the proof strategy of \cite{HomoBM}, giving an essentially equivalent argument to \cite[Claim 4.2]{HomoBM}.


Without loss of generality assume $\widetilde{T}$ is regular of volume $1$ centered at the origin.
Extend a finite covering of $[0,1]^{k-1}$ with $q_k$ translates of $\widetilde{T}$  to a periodic covering $\mathcal{C}$ of $\mathbb{R}^{k-1}$ with average density $q_k$, and let $\mu_1(k):= \lceil\log_2(q_k)\rceil+2k-1$.

We will now produce a covering of $\widetilde{T}$ by translates of $2^{-\mu_1-1}\widetilde{T}$ (which we will call $\mathcal{C}'$), the sum of whose volumes is at most $2^{k-1}q_k$.

We have that $2^{-\mu_1-1} \mathcal{C}$ is a periodic covering of $\mathbb{R}^{k-1}$ by translates of $2^{-\mu_1-1}\widetilde{T}$ with average density $q_k$, so for any polytope $\widetilde{P}$ there exists a $\vec{u}$ with $\sum_{\widetilde{S}\in \vec{u}+2^{-\mu_1+1}\mathcal{C}}|\widetilde{S}\cap \widetilde{P}|\le q_k|\widetilde{P}|$. Take $\widetilde{P}=2\widetilde{T}$, and let $\mathcal{C}'\subset \vec{u}+2^{-\mu_1-1}\mathcal{C}$ be the set of simplices which intersect $\widetilde{T}$, so that  $\widetilde{T}\subset \bigcup_{\widetilde{S}\in\mathcal{C}'} \widetilde{S}$. Each $\widetilde{S}\in \mathcal{C}'$ is contained in $\widetilde{T}+2^{-\mu_1-1}\widetilde{T}-2^{-\mu_1-1}\widetilde{T}\subset \widetilde{T}+2^{-\mu_1-1}\widetilde{T}+2^{-\mu_1-1}(k-1)\widetilde{T}\subset 2\widetilde{T}$, so $$\sum_{\widetilde{S}\in\mathcal{C}'}|\widetilde{S}|=\sum_{\widetilde{S}\in \mathcal{C}'}|\widetilde{S}\cap 2\widetilde{T}|\leq q_k|2\widetilde{T}|=2^{k-1}q_k.$$ 

 For each $\widetilde{S} \in \mathcal{C}'$, there exists a translate $f(\widetilde{S})$ of $2^{-\mu_1-1}\widetilde{T}$ such that $\widetilde{S} \cap \widetilde{T} \subset f(\widetilde{S}) \subset \widetilde{T}$ (since the intersection of two homothetic simplices is a simplex homothetic to both), and we construct  $\mathcal{C}'':=\{f(S) : \widetilde{S} \in \mathcal{C}'\}$. Then  $\sum_{\widetilde{S}'\in\mathcal{C}''}|\widetilde{S}'|\leq 2^{k-1}q_k$, all simplices in $\mathcal{C}''$ are contained in $\widetilde{T}$, and $\widetilde{T}\subset \bigcup _{\widetilde{S}'\in\mathcal{C}''}\widetilde{S}'$.

We now roughly follow the proof strategy from \cite[Lemma 3.3]{HomoBM}. The collection $\cup_{j\ge 0}\mathcal{S}_{\mu_1,j}(\widetilde{T})$ is a dense collection of translates of $2^{-\mu_1}\widetilde{T}$ contained inside $\widetilde{T}$, and in fact for every (possibly lower dimensional) face $\widetilde{F}$ of $\widetilde{T}$, the sub-collection of simplices in $\cup_{j\ge 0}\mathcal{S}_{\mu_1,j}(\widetilde{T})$ intersecting $\widetilde{F}$ is dense among all translates of $2^{-\mu_1-1}\widetilde{T}$ contained in $\widetilde{T}$ which intersect $\widetilde{F}$. Therefore for each element $\widetilde{S} \in \mathcal{C''}$, there exist a translate $h(\widetilde{S})\in\cup_{j\ge 0}\mathcal{S}_{\mu_1,j}(\widetilde{T}')$  which contains $\widetilde{S}$. Finally, we can construct the family $\mathcal{F} := \{h(\widetilde{S}) : \widetilde{S} \in \mathcal{C}'' \}$. As $\mathcal{C''}$ is a fixed finite set, there exist $\mu_2=\mu_2(k)$ such that $\mathcal{F} \subset \mathcal{S}_{\mu_1,\mu_2}(\widetilde{T}')$. Hence, $\sum_{\widetilde{S}\in\mathcal{F}} |\widetilde{S}|\leq 2^{2k-2}q_k\leq 2^{\mu_1-1}$ as desired.
\end{proof}

\begin{proof}[Proof of \Cref{squareprop}]
Recall by \Cref{coveringfamily} we find a family $\mathcal{F}\subset \mathcal{S}_{\mu_1,\mu_2}$ such that $\widetilde{T}\subset \bigcup_{\widetilde{S}\in \mathcal{F}}\widetilde{S}$ and $\sum_{\widetilde{S}\in \mathcal{F}} |\widetilde{S}|\leq 2^{\mu_1-1}|\widetilde{T}|$, i.e. $|\mathcal{F}|\leq 2^{\mu_1k-1}$.
By \Cref{mu1mu2}, we conclude that 
\begin{align*}
g(T)&\leq \sum_{\widetilde{S}\in\mathcal{F}} g(S)\\
&\leq\sum_{\widetilde{S}\in\mathcal{F}}  \left[2^{-\mu_1 k} g(T)+\left(\frac{1}{2^k-1}+ \mu_22^{-k}\right)d_k'(T)+(2^{2k+6}+\mu_22^{5k})\min\{n_i\}^{-1}|B|\right]\\
&\leq 2^{\mu_1k-1}\left[2^{-\mu_1 k} g(T)+\left(\frac{1}{2^k-1}+ \mu_22^{-k}\right)d_k'(T)+(2^{2k+6}+\mu_22^{5k})\min\{n_i\}^{-1}|B|\right].
\end{align*}
Hence as $d_k'(T)=-g^{\square}(T+T)+2^kg(T)$, we have
$$g^\square(T+T)\leq \left(2^{k}-\frac{2^{-\mu_1k}}{\frac{1}{2^k-1} +\mu_22^{-k}}\right)g(T)+ \frac{2^{2k+6}+\mu_22^{5k}}{\frac{1}{2^k-1}+ \mu_22^{-k}}\min\{n_i\}^{-1}|B|. $$
\end{proof}

\subsubsection{$A_\star$ is close to $\co(A_\star)$: Construction}
\label{AstarclosetocoAstarsubsection}
In this section we prove that $|\co(A_\star)\setminus A_\star||B|^{-1}\to 0$ as $\delta \to 0$.
\begin{prop}
\label{AstarclosetocoAstarprop}
There exist $\ll$ dependencies $n_{k,0}^{-1}\ll\delta\ll\epsilon_0\ll 1$ such that the following holds. We have for some function $h_{\star}(\delta) \rightarrow 0$ as $\delta \rightarrow 0$ that $|\co(A_\star) \setminus A_\star| \le h_{\star}(\delta) |B|$.
\end{prop}
\begin{proof}
By \eqref{dkAstar}, \eqref{Astarfewvert},  and by \Cref{squareprop}, \Cref{gsimple}, \Cref{gsquare} and \Cref{negdk}, we have that \begin{align*} 2^k|\co(A_\star) \setminus A_\star|\le& d_k(A_\star) -d_k(\co(A_\star))+ |\co(A_\star+A_\star) \setminus (A_\star+A_\star)| \\
\le&  h_9(\delta)|B| +2^{2k}n_{k,0}^{-1}|B|+h_{11}(\delta)|B|\\
&+\sum_{T^+}\sum_{x\in T^++T^+}g_{T^+}^{+\square}(x)
+\sum_{T^-}\sum_{x\in T^-+T^-}g_{T^-}^{-\square}(x)\\
\le& h_9(\delta)|B| + 2^{2k}n_{k,0}^{-1}|B|+h_{11}(\delta)|B|\\
&+\binom{H_7(\delta^{\frac{1}{20}-17c})}{k}c_k''n_{k,0}^{-1}|B|\\&+(2^k-c_k')\left(\sum_{T^+}\sum_{x\in T^+}g^+_{T^+}(x)+\sum_{T^-}\sum_{x \in T^-}g^-_{T^-}(x)\right)\\
\le& h_9(\delta)|B| + 2^{2k}n_{k,0}^{-1}|B|+h_{11}(\delta)|B|\\
&+\binom{H_7(\delta^{\frac{1}{20}-17c})}{k}c_k''n_{k,0}^{-1}|B|+ (2^k-c_k') |\co(A_\star) \setminus A_\star|
\\&+ (2^k-c_k') h_{10}(\delta)|B|. 
\end{align*}
The first inequality follows by definition. The second makes use of \eqref{dkAstar} and \Cref{gsquare}. The third makes use of \eqref{Astarfewvert}, \Cref{squareprop} and \Cref{negdk}. The fourth makes use of \Cref{gsimple}.
We conclude that $|\co(A_\star) \setminus A_\star| \le h_{\star}(\delta) |B|$, for a function $h_{\star} \rightarrow 0$ as $\delta \rightarrow 0$.
\end{proof}

\subsection{$A$ is close to $\co(A)$}
\label{AclosetocoA}

Recall that we have $d_k(A)\le \delta|B|, |A|\ge \epsilon_0|B|$, and for some functions $h_8,h_9,h_\star\to 0$ as $\delta\to 0$ that $$|\co(A_\star)\setminus A_\star|\le h_\star(\delta)|B|\text{, }\quad |A\Delta A_\star|\le h_8(\delta)|B|\text{, and}\quad d_k(A_\star)\le h_9(\delta)|B|.$$
We note that for $\delta$ sufficiently small, these imply $|\co(A_\star)|\ge \frac{2}{3}\epsilon_0|B|$.

We will now show that $|\co(A)\setminus A|\le h(\delta)|B|$ for some function $h\to 0$ as $\delta \to 0$.

\begin{lem}\label{closestarlem}
Given a convex polytope $\widetilde{Q}\subset \mathbb{R}^k$ and $0<\lambda\le 2k+1$, let $o$ be the center of the John's ellipsoid $\widetilde{E}\subset \widetilde{Q}$. If $p\not \in (1+\lambda)\widetilde{Q}$ (the homothety being taken with respect to $o$), then there is a convex polytope $\widetilde{P}\subset \widetilde{Q}$ with $|\widetilde{P}|= \left(\frac{\lambda}{2k+1}\right)^k|\widetilde{Q}|$ such that $$\frac{\widetilde{P}+p}{2}\cap\widetilde{Q}=\emptyset.$$
\end{lem}
\begin{proof}
We may assume $p\in \partial (1+\lambda)\widetilde{Q}$, and by taking an affine transformation that the John's ellipsoid $\widetilde{E}$ is a ball of radius $1$ with $o$ at the origin. Then by John's Lemma \cite{John}, $$\widetilde{E}\subset \widetilde{Q}\subset -k\widetilde{E}.$$
Hence we estimate the diameter of $\widetilde{Q}$ is at most $2k$, so it strictly less than $2k+1$.
Let $q=op\cap \partial \widetilde{Q}$, let $H$ be the homothety with center $q$ and ratio $\frac{\lambda}{2k+1}\le 1$, and let $\widetilde{P}=H(\widetilde{Q})$. Clearly $\widetilde{P}\subset \widetilde{Q}$ has the desired volume, and has diameter strictly less than $\lambda$. Let $H'$ be the homothety with center $q$ and ratio $-\frac{\lambda}{2}$. Because $H'$ is a negative homothety at $q$, it leaves any plane through $q$ invariant and swaps the two halfspaces determined by such a plane. Hence $H'(\widetilde{Q})$ and $\widetilde{Q}$ are separated by the supporting hyperplane to $\widetilde{Q}$ at the point $q$, so it is enough to show that $\frac{\widetilde{P}+p}{2}$ is contained in the interior of $H'(\widetilde{Q})$.

As the distance from $o$ to $\partial \widetilde{Q}$ is at least $1$ (the radius of $\widetilde{E}$), the distance from $H'(o)$ to $\partial H'(\widetilde{Q})$ is at least $\frac{\lambda}{2}$. As
$\frac{\widetilde{P}+p}{2}$ is a set of diameter strictly less than $\frac{\lambda}{2}$ containing $H'(o)=\frac{p+q}{2}$, it is contained in the interior of $H'(\widetilde{Q})$ as desired.
\end{proof}

\begin{proof}[Proof of \Cref{equivquali}] Recall from the beginning of \Cref{1.6section} that we may assume that $A$ is reduced, so $\coo(A)=\co(A)$.
There exist $\ll$ dependencies $n_{k,0}^{-1}\ll\delta\ll\epsilon_0\ll 1$ such that all of the following holds. First, the functions $h_8,h_9,h_\star$ exist and have the previously established properties. Second, there are functions $h_{12}(\delta)$ and $\lambda(\delta)$ with $h_{12}(\delta)\to 0$ and $\lambda(\delta)\to 0$ as $\delta\to 0$ such that $$h_{12}(\delta)\ge 2^kh_8(\delta)+h_\star(\delta)+(2^{2k}+2k(k+1))n_{k,0}^{-1}$$ and  $$\frac{\epsilon_0}{2}\left(\frac{\lambda}{2k+1}\right)^k-h_{12}(\delta)>\delta.$$
Indeed, we can take $h_{12}(\delta)=2^kh_8(\delta)+h_9(\delta)+h_\star(\delta)+2k(k+1)\delta$, and $\lambda(\delta)=(2k+1)(\frac{4}{\epsilon_0}(\delta+h_{12}(\delta)))^{\frac{1}{k}}$.
Third, $$(1+(1+\lambda)^k)2k(k+1)n_{k,0}^{-1}\le \frac{\epsilon_0}{2}((1+2\lambda)^k-(1+\lambda)^k).$$ Indeed, this definition of $\lambda(\delta)$ also makes this hold.

Note that $|\cooo(A_\star)|\ge \frac{\epsilon_0}{2}|B|$ by \Cref{contdisc}. Let $o$ be its barycenter of the John's ellipsoid $\widetilde{E}\subset \cooo(A_\star)$. Consider the homothety $H$ with center $o$ and ratio $1+\lambda(\delta)$. Let $\widetilde{R}=H(\cooo(A_{\star}))$. We will show now that $A\subset \widetilde{R}$. Indeed, suppose not, and let $x \in A\setminus \widetilde{R}$. Then by \Cref{closestarlem}, there is a subset $\widetilde{P}\subset \cooo(A_\star)$ with volume $(\frac{\lambda}{2k+1})^k|\cooo(A_\star)|$ such that $\widetilde{P}+x$ is disjoint from $2\cooo(A_\star)$. Then, by \Cref{negdk}, \Cref{dkobs}, and \Cref{contdisc},
\begin{align*}|A+A| &\ge |x+(\widetilde{P}\cap A_\star)|+|A_\star+A_\star|\\&\ge|x+(\widetilde{P}\cap \co(A_\star))|-h_{\star}(\delta)|B|+2^k|A_\star|+d_k(A_\star)\\
&\ge|\widetilde{P}\cap \mathbb{Z}^k|-h_{\star}(\delta)|B|+ 2^k(|A|-h_8(\delta)|B|)-2^{2k}n_{k,0}^{-1}|B|\\
&\geq|\widetilde{P}|+2^k|A|-h_{12}(\delta)|B|\\
&\ge \left(\frac{\lambda}{2k+1}\right)^k|\cooo(A_\star)|+2^k|A|-h_{12}(\delta)|B|\\
&\ge 2^k|A|+\left(\frac{\epsilon_0}{2}\left(\frac{\lambda}{2k+1}\right)^k-h_{12}(\delta)\right)|B|.
\end{align*}
Hence,
$$\delta |B|\ge d_k(A)\ge \left(\frac{\epsilon_0}{2}\left(\frac{\lambda}{2k+1}\right)^k-h_{12}(\delta)\right)|B|>\delta|B|,$$
a contradiction.
Therefore $A\subset H(\cooo(A_\star))$, so $\cooo(A)\subset H(\cooo(A_\star))$. Recalling that $H(\cooo(A_\star))$ is a translate of $(1+\lambda)\cooo(A_\star)$, by \Cref{contdisc} applied to $\cooo(A)$ and $\cooo(A_\star)$, we have
\begin{align*}|\co(A)|&\le (1+\lambda)^k|\co(A_\star)|+(1+(1+\lambda)^k)2k(k+1)n_{k,0}^{-1}|B|\\
&\le (1+2\lambda)^k|\co(A_\star)|\\
&\le |\co(A_\star)|+((1+2\lambda)^k-1)|B|\\
&\le |A_\star|+((1+2\lambda)^k-1+h_\star(\delta))|B|\\
&\le |A|+(h_8(\delta)+(1+2\lambda)^k-1+h_\star(\delta))|B|.
\end{align*}
And hence for $\omega'(\delta)=h_8(\delta)+(1+2\lambda(\delta))^k-1+h_\star(\delta)$, which tends to $0$ as $\delta\to 0$, we have
$$|\co(A)\setminus A|\le \omega'(\delta)|B|\le \omega'(\delta)\epsilon_0^{-1}|A|.$$
Taking $\omega(\delta)=\sqrt{\omega'(\delta)}$, then $\omega(\delta)\to 0$ as $\delta\to 0$ and $\omega(\delta)\ge \epsilon_0^{-1} \omega'(\delta)$ for $\delta$ sufficiently small in terms of $\epsilon_0$, so 
$$|\co(A)\setminus A|\le \omega(\delta)|A|.$$
\end{proof}

\section{Proof of \Cref{quant} for $k$ given \Cref{qual} for $k$}
\label{quantfromqualsection}
In this section, we prove \Cref{quant} for dimension $k$ given \Cref{quali} for dimension $k-1$. A few important notes before we begin.
\begin{itemize}
    \item We be \textbf{exclusively} working in the equivalent reformulations \Cref{equivquali} (of \Cref{quali}) and \Cref{equivmainthm} (of \Cref{mainthm}) as established in \Cref{equivsection}. Our hypotheses on $A$ are therefore the ones from \Cref{equivquali}, that $\delta \ll \epsilon_0 \le 1$ and $A\subset B=B(n_1,\ldots,n_k)$ with $|A| \ge \epsilon_0|B|$ and $d_k(A)\le \delta |B|$, and our desired conclusion is that
$$|\coo(A)\setminus A|\le c_kd_k(A)+g_k(\epsilon_0)\min\{n_i\}^{-\frac{1}{1+\frac{1}{2}(k-1)\lfloor k/2 \rfloor}}|A|.$$
    \item As will be shown at the beginning of the proof of \Cref{equivmainthm}, we will be able to assume a $\min\{n_i\}^{-1}\ll \epsilon_0$ dependency. Hence by \Cref{nonreduced} we may and shall assume that $A$ is reduced.
\end{itemize}

To prove \Cref{equivmainthm}, we first prove the following closely related proposition. 
\begin{prop}\label{quantprop}
 There are constants $c_k<(4k)^{5k}$, $f_k$ and  $\rho_k(\epsilon_0),n_{k}(\epsilon_0)$ for all $\epsilon_0>0$ such that the following is true. For every box $B=B(n_1,\ldots,n_k)$ with $n_1,\ldots,n_k \ge n_k(\epsilon_0)$, and for $A'\subset B$ a reduced set with $|A'|\ge \epsilon_0|B|$, $|\co(A')\setminus A'|\le \rho_k(\epsilon_0)|A'|$, and a triangulation $\mathcal{T}$ of $\partial\cooo(A')$, we have that
$$|\co(A')\setminus A'|\le c_{k}d_k(A')+f_k|\mathcal{T}|\min\{n_i\}^{-1}|B|.$$
\end{prop}
We will see that this result follows from the following result.
\begin{prop}
\label{magicprop}
There are constants $c^1_k,c^2_k,\eta_k>0$ (we can take $c^1_k\leq 2^{2k}(2k)^{5k}$) with $\eta_k\le \frac{1}{2}$ such that the following is true. For every box $B=B(n_1,\ldots,n_k)$ and for $\widetilde{T}\subset \widetilde{B}$ a simplex with vertices $o=x_0,x_1,\ldots,x_k$, and $A\subset T=\widetilde{T}\cap \mathbb{Z}^k$ with $\{o,x_1,\ldots,x_k\}\subset A$ we have
$$|(T\setminus (1-\eta_k)\widetilde{T})\setminus A|\le \frac{1}{2}|T\setminus A|+c_k^1d_k(A)+c^2_k\min\{n_i\}^{-1}|B|,$$
where the $(1-\eta_k)$-scaling is done with respect to $o$.
\end{prop}
Until the end of the proof of \Cref{magicprop} later in this section, we fix the hypotheses of \Cref{magicprop} (so in particular, $A\subset T$).
\begin{notn}
We shall write $A_S:=A\cap S$.
\end{notn}

Recall \Cref{sij} in \Cref{AstarConstruction}, which recursively constructs a family of simplices $\mathcal{S}_{i,j}(\widetilde{T})$ such that $\mathcal{S}_{0,0}=\{\widetilde{T}\}$, such that every simplex in $\mathcal{S}_{i,0}$ is the average of a vertex in $V(T)$ with a simplex in $\mathcal{S}_{i-1,0}$, and every simplex in $\mathcal{S}_{i,j}$ is the average of two simplices in $\mathcal{S}_{i,j-1}$.
\begin{lem}
For $x$ a vertex of $\widetilde{T}$ and $\widetilde{S}=(1-2^{-i})x+2^{-i}\widetilde{T} \in \mathcal{S}_{i,0}$ and $\widetilde{S}'=\frac{1}{2}(x+\widetilde{S}) \in \mathcal{S}_{i+1,0}$, we have
$$|A_{S'}| \ge 2^{-k}|A_{ S}|-2^{-k}d_k(A)-2^k\min\{n_i\}^{-1}|B|.$$
\end{lem}
\begin{proof}
By \Cref{negdk}, with $\widetilde{S}'^c$ the complement of $\widetilde{S}'$ in $\widetilde{T}$,
\begin{align*}d_k(A)=|A+A|-2^k|A|&= |(A+A)\cap 2\widetilde{S'}|-2^k|A_{S'}|+|(A+A)\cap 2\widetilde{S'}^c|-2^k|A_{S'^c}|\\
&\ge |(A+A)\cap 2\widetilde{S'}|-2^k|A_{S'}|+|A_{S'^c}+A_{S'^c}|-2^k|A_{S'^c}|\\
&\ge |x+A_{S}|-2^k|A_{S'}|-2^{2k}\min\{n_i\}^{-1}|B|\\
&=|A_{S}|-2^k|A_{S'}|-2^{2k}\min\{n_i\}^{-1}|B|.\end{align*}
\end{proof}
\begin{cor}
For $\widetilde{S}\in \mathcal{S}_{i,0}$ we have
$$|A_{S}| \ge 2^{-ik}|A|-\frac{1-2^{-ik}}{2^k-1}d_k(A)-2^{k+1}\min\{n_i\}^{-1}|B|.$$
\end{cor}

\begin{lem}
For $\widetilde{S}_1, \widetilde{S}_2\in \mathcal{S}_{i,j}$, and $\widetilde{S}'=\frac{1}{2}(\widetilde{S}_1+\widetilde{S}_2)\in \mathcal{S}_{i,j+1}$, we have
$$|A_{S'}| \ge \min(|A_{S_1}|,|A_{S_2}|)-2^{-k}d_k(A)-(k+2)2^{k}\min\{n_i\}^{-1}|B|.$$
\end{lem}
\begin{proof}
Let $\widetilde{P}_1,\ldots,\widetilde{P}_{k+1}$ be a partition of $\widetilde{S}'^c$ into convex sets as in the proof of \Cref{averageg}. Then by \Cref{negdk}, we have
\begin{align*}
d_k(A)&=|A+A|-2^k|A|\\
&= |(A+A)\cap 2\widetilde{S'}|-2^k|A_{S'}|+|(A+A)\cap 2\widetilde{S'}^c|-2^k|A_{S'^c}|\\
&\ge|A_{S_1}+A_{S_2}|-2^k|A_{S'}|+\sum_{i=1}^{k+1}\left(|A_{\widetilde{P}_i}+A_{\widetilde{P}_i}|-2^k|A_{\widetilde{P}_i}|\right)\\
&\ge 2^k(\min(|A_{S_1}|,|A_{S_2}|)-2^k|A_{S'}|-(k+2)2^{2k}\min\{n_i\}^{-1}|B|.\end{align*}
\end{proof}
\begin{cor}\label{FinalAScor}
For $\widetilde{S}\in \mathcal{S}_{i,j}$ we have \begin{align*}|A_{S}|&\ge 2^{-ik}|A|-\left(\frac{1-2^{-ik}}{2^k-1}+j2^{-k}\right)d_k(A)-\left(2^{k+1}+j(k+2)2^{k}\right)\min\{n_i\}^{-1}|B|.
\end{align*}
In particular, by \Cref{contdisc} applied to $S$ and $T$, we have
$$|S\setminus A_{S}|\leq 2^{-ik}|T\setminus A|+c^{1}_{i,j}d_k(A)+c^{2}_{i,j}\min\{n_i\}^{-1}|B|,$$
with $c^{1}_{i,j}=\frac{1-2^{-ik}}{2^k-1}+j2^{-k}$ and $c^{2}_{i,j}=(1+2^{-ik})2k(k+1)+2^{k+1}+j(k+2)2^{k}$.
\end{cor}
\begin{proof}[Proof of \Cref{magicprop}]
Let $i=\left\lceil\log_{\frac12}\left(\frac{k^{1/k}}{(2k)^5}\right)\right\rceil$ and  $j=\lfloor 16k\log(2k) \rfloor$. Let $c_k^1=(2k)^{5k}c^1_{i,j}\leq 2^{2k}(2k)^{5k}$ and $c_k^2=(2k)^{5k}c_{i,j}^2$ where $c^1_{i,j}$ and $c^2_{i,j}$ are as in \Cref{FinalAScor}.
By \cite[Claim 4.2]{HomoBM}, there exists a constant $\eta_k>0$ and a family of simplices $\mathcal{F}\subset \mathcal{S}_{i,j}$ with $|\mathcal{F}|\leq (2k)^{5k}$ such that $$\sum_{\widetilde{S}\in \mathcal{F}}|\widetilde{S}|\le \frac{1}{2}|\widetilde{T}|,$$ and $$\widetilde{T}\setminus (1-\eta_k)\widetilde{T}\subset \bigcup_{\widetilde{S}\in \mathcal{F}} \widetilde{S}.$$
Note that by taking volumes, $(1-(1-\eta_k)^k)\le \frac{1}{2}$, so in particular $\eta_k\le \frac{1}{2}$.
We prove \Cref{magicprop} with parameters $c_k^1,c_k^2, \eta_k$ as above. Noting that $2^{-ik}=\frac{|\widetilde{S}|}{|\widetilde{T}|}$, we have
\begin{align*}
    |(T\setminus A)\setminus (1-\eta_k)\widetilde{T}|&\le \sum_{\widetilde{S}\in \mathcal{F}}|S\setminus A_{S}|\\
    &\le \sum_{\widetilde{S}\in \mathcal{F}}\left(2^{-ik}|T\setminus A|+c^1_{i,j}d_k(A)+c^2_{i,j}\min\{n_i\}^{-1}|B|\right)\\
    &\le \frac{1}{2}|T\setminus A|+c_{k}^1d_k(A)+c_{k}^2\min\{n_i\}^{-1}|B|.
\end{align*}
\end{proof}

We fix  $\eta_k$ as in \Cref{magicprop}. We need one final lemma to prove \Cref{quantprop}.
\begin{lem}\label{containedheart}
For every $\epsilon_0>0$, there exists a constant $\rho_k(\epsilon_0)>0$ such that if $A'\subset B$ with $|A'|\ge \epsilon_0|B|$,  $|\co(A')\setminus A'|\le \rho_k(\epsilon_0)|B|$, and $n_1,\ldots,n_k$ larger than some constant depending on $k$ and $\epsilon_0$, then there exists $o\in\mathbb{Z}^k\cap \cooo(A')$, such that (recalling the constant $\eta_k$ from \Cref{magicprop}) we have
$$(1-\eta_k)(\cooo(A'+A')-2o)\cap \mathbb{Z}^k\subset A'+A'-2o.$$
\end{lem}
\begin{proof}
By \Cref{contdisc}, we have $|\cooo(A')|\geq \left(\epsilon_0-2k(k+1) \min\{n_i\}^{-1}\right)|B|\ge \frac{\epsilon_0}{2}|B|$, so by John's Lemma \cite{John}, there exists an ellipsoid $\widetilde{F'}\subset \cooo(A')$ with $|\widetilde{F'}|\geq k^{-k}\frac{\epsilon_0}{2}|B|$. Let $o'$ be the centre of this ellipsoid $\widetilde{F'}$. Let $o\in\mathbb{Z}^k$ be a point closest to $o'$. If $o\not\in \widetilde{F'}$ then the smallest axis of $\widetilde{F'}$ would have length less than $2$. But because the largest cross-section of $\widetilde{F}'$ spanned by the remaining axes has area at most $\frac{1}{2}|\partial{\widetilde{B}}|$, we have $|\widetilde{F'}|\le |\partial{\widetilde{B}}|\le 2k\min\{n_i\}^{-1}|B|$, which is strictly less than $k^{-k}\frac{\epsilon_0}{2}$ provided $\min\{n_i\}\ge 2k(k^k)\frac{2}{\epsilon_0}$. Hence we may assume that $o\in \widetilde{F'}$. Let $p\in\partial \widetilde{F'}$ be the intersection of the ray $o'o$ with  $\partial\widetilde{F'}$. Let $H'$ be the homothety centred at $p$ with ratio $\frac{|op|}{|o'p|}\geq 1-\frac{\sqrt{k}}{|o'p|}$, so that $H'(o')=o$. If $|o'p|\le 2\sqrt{k}$, then as the cross-sectional area of $\widetilde{F}'$ perpendicular to $o'p$ is at most $|\partial \widetilde{B}|$, we see that $|\widetilde{F}'|\le  2\sqrt{k}|\partial\widetilde{B}|\le 4k^{3/2}\min\{n_i\}^{-1}|B|<k^{-k}\frac{\epsilon_0}{2}|B|$, a contradiction. Hence $\widetilde{F}=H'(\widetilde{F'})\subset \cooo(A')$ is an ellipse with center $o$ and  $|\widetilde{F}|\ge(2k)^{-k}\frac{\epsilon_0}{2}|B|$.

Taking a point $x'\in(1-\eta_k)(\cooo(A'+A')-2o)\cap \mathbb{Z}^k$, we want to show that $x'\in A'+A'-2o$. Let $x=\frac{1}{2}x'\in (1-\eta_k)\cooo(A'-o)$, and let $y$ be the intersection of the ray $\mathbb{R}_{\ge 0}x$ with $\partial \cooo(A'-o)$. Note that the ratio $r=|xy|/|y| \ge \eta_k$. Let $H$ be the homothety with center $y$ and ratio $r$. This homothety sends $0$ to $x$ and $\cooo(A'-o)$ to $H(\cooo(A'-o))$.
Note that $H(\cooo(A'-o)) \subset \cooo(A'-o)$ because $\cooo(A'-o)$ is convex. Note that $H(\widetilde{F}-o)$ is symmetric around $x$ and satisfies $|H(\widetilde{F}-o)|= r^k |\widetilde{F}|$. By \Cref{contdisc}, $$|H(\widetilde{F}-o)\cap \mathbb{Z}^k|\geq r^k |\widetilde{F}|-2k(k+1) \min\{n_i\}^{-1}|B|\ge \eta_k^k(2k)^{-k}\frac{\epsilon_0}{4}|B|> 2\rho_k(\epsilon_0)|B|,$$
for $\rho_k(\epsilon_0)$ sufficiently small. In particular, as $H(\widetilde{F}-o)\subset\cooo(A'-o)$, 
\begin{align*}
    |H(\widetilde{F}-o)\cap (A'-o)|&\geq |H(\widetilde{F}-o)\cap \mathbb{Z}^k|-|\co(A'-o)\setminus (A'-o)|> \frac{1}{2}|H(\widetilde{F}-o)\cap \mathbb{Z}^k|.
\end{align*}
By the symmetry of $H(\widetilde{F}-o)$ around $x$, we have that $z\in H(\widetilde{F}-o)\cap \mathbb{Z}^k$ implies that also $2x-z \in H(\widetilde{F}-o)\cap \mathbb{Z}^k$. Hence, as $H(\widetilde{F}-o)\cap (A'-o)$ contains more than half the elements in $H(\widetilde{F}-o)\cap \mathbb{Z}^k$, we can find $z,z'\in H(\widetilde{F}-o)\cap (A'-o)$, such that $x'=z+z'$ and thus $x'\in A'+A'-2o$.
\end{proof}

\begin{proof}[Proof of \Cref{quantprop}]
Let $c_k=2c_k^1+2^{1-k}\leq (4k)^{5k}$ and $f_k=(2(k+1)(\frac{1}{2}+c_k^1)+8k(k+1)+2^{k+1}+2c_k^2+(2+2^{1-k})2k(k+1)$.
Let $o$ be the point supplied by \Cref{containedheart}, $\eta_k$ the constant supplied by \Cref{magicprop}. Note that $A'-o\subset \frac{1}{2}(\cooo(A'+A')-2o)\cap \mathbb{Z}^k \subset A'+A'-2o$ as $\frac{1}{2}\le 1-\eta_k$, so we have $o+A'\subset A'+A'$. Therefore we find $d_k(A'\cup\{o\})\leq d_k(A')$ and $|\co(A')\setminus A'|\ge |\co(A'\cup\{o\})\setminus (A'\cup\{o\})|\geq |\co(A')\setminus A'|-1 $, so we may assume $o\in A'$. Let $\mathcal{T}'$ be a triangulation of $\cooo(A')$ obtained by considering the $k-1$ dimensional simplices in $\mathcal{T}$ and adding $o$ as a vertex to each, so in particular $|\mathcal{T}|=|\mathcal{T}'|$. For each $\widetilde{T}\in \mathcal{T}'$ all vertices are in $A'$. For $x=o$ or $x=2o$, let $\Gamma_{x,1-\eta_k}$ denote the homothety with center $x$ scaling by $1-\eta_k$. Then by \Cref{containedheart}, we have
\begin{align*}|\co(A'+A')\setminus (A'+A')| &\le \sum_{\widetilde{T}\in \mathcal{T}'}|\co(T+T)\setminus(A'+A')|\\
&= \sum_{\widetilde{T}\in\mathcal{T}'}|(\co(T+T)\setminus (A'+A'))\setminus \Gamma_{2o,1-\eta_k}2\widetilde{T}|\\
&\le \sum_{\widetilde{T}\in\mathcal{T}'}|(2\widetilde{T}\setminus \Gamma_{2o,1-\eta_k}2\widetilde{T})\cap \mathbb{Z}^k)\setminus (A'_{\widetilde{T}\setminus \Gamma_{o,1-\eta_k}\widetilde{T}}+A'_{\widetilde{T}\setminus \Gamma_{o,1-\eta_k}\widetilde{T}})|\\
&=\sum_{\widetilde{T}\in \mathcal{T}'}\left(|(2\widetilde{T}\setminus \Gamma_{2o,1-\eta_k}2\widetilde{T})\cap \mathbb{Z}^k|-|A'_{\widetilde{T}\setminus \Gamma_{o,1-\eta_k}\widetilde{T}}+A'_{\widetilde{T}\setminus \Gamma_{o,1-\eta_k}\widetilde{T}}|\right),
\end{align*}
where in the final equality we use that $\widetilde{T}\setminus \Gamma_{o,1-\eta_k}\widetilde{T}$ is convex (as $o$ is a vertex of $\widetilde{T}$), so $A'_{\widetilde{T}\setminus \Gamma_{o,1-\eta_k}\widetilde{T}}+A'_{\widetilde{T}\setminus \Gamma_{o,1-\eta_k}\widetilde{T}}\subset 2\widetilde{T}\setminus \Gamma_{2o,1-\eta_k}2\widetilde{T}$. By \Cref{negdk}, \Cref{contdisc}, and \Cref{magicprop}, this is
\begin{align*}
\le& \sum_{\widetilde{T}\in\mathcal{T}'}\left(|2\widetilde{T}\setminus\Gamma_{2o,1-\eta_k}2\widetilde{T}|-2^k|A'_{\widetilde{T}\setminus \Gamma_{o,1-\eta_k}\widetilde{T}}|+(2^{k}2k(k+1)+2^{2k})\min\{n_i\}^{-1}|B|\right)\\
=&2^k\sum_{\widetilde{T}\in \mathcal{T}'}\left(|\widetilde{T}\setminus\Gamma_{o,1-\eta_k}\widetilde{T}|-|A'_{\widetilde{T}\setminus \Gamma_{o,1-\eta_k}\widetilde{T}}|\right)+|\mathcal{T}|(2^{k+1}k(k+1)+2^{2k})\min\{n_i\}^{-1}|B|\\
\le& 2^k\sum_{\widetilde{T}\in \mathcal{T}'}\left(|(T\setminus A')\setminus \Gamma_{o,1-\eta_k}\widetilde{T}|\right)+|\mathcal{T}|(2^{k+2}k(k+1)+2^{2k})\min\{n_i\}^{-1}|B|\\
\le& 2^{k}\sum_{\widetilde{T}\in \mathcal{T}'}\left(\frac{1}{2}|T\setminus A'| + c_k^1d_k(A'\cap T)\right)+|\mathcal{T}|(2^{k+2}k(k+1)+2^{2k}+2^kc_k^2)\min\{n_i\}^{-1}|B|\\
\le& 2^{k-1}|\co(A')\setminus A'|+2^kc_k^1d_k(A')\\&+|\mathcal{T}|\left(2^k(k+1)\left(\frac{1}{2}+c_k^1\right)+2^{k+2}k(k+1)+2^{2k}+2^kc_k^2\right)\min\{n_i\}^{-1}|B|,
\end{align*}
where in the last line we estimated the errors coming from the boundaries of the facets of the simplices in $\mathcal{T}$ (noting that there are at most $(k+1)|\mathcal{T}|$ boundary simplices, and each of them is contained in a hyperplane so contains at most $\min\{n_i\}^{-1}|B|$ points by \Cref{hypboxsmall}).
By \Cref{contdisc}, we have
\begin{align*}2^k| \co(A')|-|\co(A'+A')|\le& 2^k|\co(A')|-|\cooo(A'+A')|+2k(k+1)\min\{n_i\}^{-1}|B|\\
=&2^k|\co(A')|-2^k|\cooo(A')|+2k(k+1)\min\{n_i\}^{-1}|B|\\
\le& 2^k|\co(A')|-2^k|\co(A')|+2^k2k(k+1)\min\{n_i\}^{-1}|B|\\&+2k(k+1)\min\{n_i\}^{-1}|B|\\
=&(2^k+1)2k(k+1)\min\{n_i\}^{-1}|B|,\end{align*} so as $d_k(A')=|A'+A'|-2^k|A'|$ we conclude
\begin{align*}
|\co(A')\setminus A'|=&\frac{1}{2^{k-1}}|\co(A'+A')\setminus (A'+A')|-|\co(A')\setminus A'|\\&+\frac{1}{2^{k-1}}(|A'+A'|-2^k|A'|)+\frac{1}{2^{k-1}}(2^k|\co(A')|-|\co(A'+A')|)\\\le& (2c_k^1+2^{1-k})d_k(A')\\
&+|\mathcal{T}|\left(2(k+1)\left(\frac{1}{2}+c_k^1\right)+8k(k+1)+2^{k+1}+2c_k^2+(2+2^{1-k})2k(k+1)\right)\\&\quad\min\{n_i\}^{-1}|B|\\
=&c_kd_k(A')+f_k|\mathcal{T}|\min\{n_i\}^{-1}|B|.
\end{align*}
\end{proof}
\begin{proof}[Proof of \Cref{equivmainthm}]
Let $\alpha=\min\{n_i\}^{-\gamma}$ with $\gamma=\frac{1}{1+\frac{1}{2}(k-1)\lfloor k/2 \rfloor}$ and $\ell=\tau_{k}\alpha^{-\frac{k-1}{2}}$ with $\tau_{k}$ as in \Cref{polyapprox}. Note that this $\gamma$ satisfies $-\gamma=\frac{k-1}{2}\lfloor k/2 \rfloor \gamma-1$. Let $\omega$ be the function from \Cref{equivquali}. Let $\delta(\epsilon_0)$ be a function of $\epsilon_0$ which realizes the $\delta\ll \epsilon_0$ dependency from \Cref{equivquali}, and also satisfies $\omega(\delta(\epsilon_0))\le \rho_k(\frac{\epsilon_0}{2})$ for $\rho_k$ the function from \Cref{quantprop}.

Because $g_k(\epsilon_0)\min\{n_i\}^{-\gamma}|A| \ge g_k(\epsilon_0)\epsilon_0\min\{n_i\}^{-\gamma}|B|$, if $g_k(\epsilon_0)\epsilon_0\min\{n_i\}^{-\gamma}\ge 1$ the statement holds trivially. Hence by choosing the function $g_k(\epsilon_0)$ we may assume any $\min\{n_i\}^{-1}\ll \epsilon_0$ dependency we need (in particular to apply previous results) in the remainder of the proof. In particular, by \Cref{nonreduced} we may assume that $A$ is reduced.

Take the dependencies $\min\{n_i\}^{-1}\ll\ \delta(\epsilon_0) \ll \epsilon_0\le 1$ sufficiently strong so that we may apply \Cref{equivquali}. Applying \Cref{equivquali}, we conclude $|\co(A)\setminus A|\le \rho_k(\frac{\epsilon_0}{2})|B|$. By \Cref{polyapprox}, we obtain the function $\alpha=\alpha(\min\{n_i\})$ and a subset $A'\subset A$ such that $|\co(A)\setminus \co(A')|\le \alpha |B|$ and $A'=A\cap \co(A')$ such that $\co(A')$ has at most $\ell$ vertices. In particular, we have $|A\setminus A'|\le \alpha  |B|$. By \Cref{dkobs}, we have $d_k(A')\le 2^k\alpha |B|+d_k(A)$.

Note also that $A'$ is reduced provided our $\ll$-dependencies are strong enough. Indeed, if $A'$ is not reduced, let $a\in A\setminus A'$ be an element not in the coset of $A'$. Note that by \Cref{negdk},  $(2^k+\delta)|A|\ge |A+A| \ge |(a\cup A')+A'| \ge |A'+A'|+|A'| \ge (2^k+1)|A'|-2^{2k}\min\{n_i\}^{-1}|B|$, which contradicts $|A|-|A'| \le \alpha |B|$ and $|A| \ge \epsilon_0 |B|$.

Now, by the upper bound theorem \cite{Stanley}, if we take a triangulation $\mathcal{T}$ of $\partial \cooo(A')$ we have $|\mathcal{T}|\le f_k'\ell^{\lfloor k/2 \rfloor}=f_k''\alpha^{-\frac{k-1}{2}\lfloor k/2\rfloor}$ for some constants $f_k',f_k''$. Also $|\co(A')\setminus A'|\le |\co(A)\setminus A| \le \rho_k(\frac{\epsilon_0}{2})|B|$, where the first inequality is because $A'=\co(A')\cap A$, so we may apply \Cref{quantprop} (as we may take $\min\{n_i\}$ to be greater than the function $n_k(\epsilon_0)$ from \Cref{quantprop}) to $A'=\co(A')\cap A\subset A$, which gives a constant $f_k$ such that $$|\co(A')\setminus A'|\le c_kd_k(A')+\alpha^{-\frac{k-1}{2}\lfloor k/2\rfloor}f_k''f_k\min\{n_i\}^{-1}|B|.$$ Because $d_k(A')\le 2^k\alpha|B|+d_k(A)$, $|\co(A)\setminus \co(A')|\le \alpha|B|$, and $-\gamma=\frac{k-1}{2}\lfloor k/2 \rfloor \gamma-1$, we conclude that
\begin{align*}
|\co(A)\setminus A|&\le |\co(A')\setminus A'|+|\co(A)\setminus \co(A')|\\&\le (2^kc_k+1)\alpha |B|+c_kd_k(A)+\alpha^{-\frac{k-1}{2}\lfloor k/2 \rfloor}f_k''f_k\min\{n_i\}^{-1}|B|\\
&\leq c_kd_k(A)+g_k\min\{n_i\}^{-\frac{1}{1+\frac{1}{2}(k-1)\lfloor k/2 \rfloor}}|A|,
\end{align*}
where we take $g_k=g_k(\epsilon_0)$ sufficiently large in terms of $\epsilon_0$ to guarantee this last inequality.
\end{proof}

\bibliographystyle{abbrv}
\bibliography{references}

\end{document}